\documentclass[letterpaper, 11 pt]{article} 

\title{Joint Design of Time-of-Use Schedules and Price Levels: a Bilevel Model with Client Participation and Load Flexibility\thanks{This work was supported by the D\'efi Inria-EDF.}}

\usepackage{authblk}

\author[1, 2, 3]{Abdellah Bulaich Mehamdi}
\author[2]{Wim van Ackooij}
\author[3]{Luce Brotcorne}
\author[1, 3]{St\'ephane Gaubert}
\author[2]{Quentin Jacquet}

\affil[1]{CMAP, École polytechnique, Institut Polytechnique de Paris, Palaiseau, France.}
\affil[2]{EDF Lab Paris-Saclay, OSIRIS department, Palaiseau, France.}
\affil[3]{Institut National de Recherche en Informatique et en Automatique (Inria), France.}

\date{\vspace{-3em}}

\usepackage{enumitem}
\usepackage{xltabular}
\usepackage{booktabs}
\usepackage{placeins} 

\usepackage{tabularx}
\usepackage{dsfont}
\usepackage{amssymb}
\usepackage{lipsum}
\usepackage{algorithm}
\usepackage{algpseudocode}
\usepackage{url}
\usepackage{multicol}
\usepackage{sansmath}
\usepackage{hyperref}
\usepackage{amsmath} 
\usepackage{adjustbox}
\usepackage{stmaryrd}
\usepackage{array}
\usepackage{moresize}
\usepackage{float}
\usepackage{bbold}
\usepackage{siunitx}
\usepackage{subfig}
\usepackage{textcomp}
\usepackage{mathtools}
\usepackage{notoccite}
\usepackage[english]{babel}
\usepackage{tikz}
\usepackage{circuitikz}
\usetikzlibrary{positioning}
\usepackage{tikz-3dplot}
\usepackage{csquotes}
\usepackage{wallpaper}
\usepackage[capitalise]{cleveref}
\usepackage{mathrsfs}
\usepackage[toc]{appendix}
\usepackage{amsthm}
\usepackage{dutchcal}
\usepackage{graphicx}
\usepackage[official]{eurosym}
\usepackage{pgfplots}
\usepackage{xcolor}
\usepgfplotslibrary{ternary}
\pgfplotsset{compat=1.18}
\usepgfplotslibrary{fillbetween}
\allowdisplaybreaks

\definecolor{myblue}{RGB}{0, 102, 204} 
\definecolor{mylightblue}{RGB}{0, 153, 255} 
\definecolor{myred}{RGB}{204, 0, 0} 
\definecolor{myorange}{RGB}{255, 140, 0} 
\definecolor{mygreen}{RGB}{0, 153, 76} 
\definecolor{myblack}{RGB}{20, 20, 20} 

\newtheorem{theorem}{Theorem}[section]
\newtheorem{corollary}[theorem]{Corollary}
\newtheorem{proposition}[theorem]{Proposition}
\newtheorem{definition}{Definition}
\newtheorem{lemma}[theorem]{Lemma}
\newtheorem{assumption}{Assumption}
\theoremstyle{definition}
\newtheorem{example}[theorem]{Example}
\newtheorem{remark}[theorem]{Remark}

\usepackage{multirow}

\newcommand{\CCH}{\operatorname{CCH}}
\newcommand{\CCO}{\operatorname{CCO}}
\newcommand{\CP}{\operatorname{CP}}
\newcommand{\BOP}{\operatorname{BOP}}
\newcommand{\RToU}{\operatorname{RToU}}
\newcommand{\PoR}{\operatorname{PoR}}
\newcommand{\HO}{\operatorname{HO}}
\newcommand{\AHO}{\operatorname{A-HO}}
\newcommand{\Bill}{\mathsf{B}}
\newcommand{\Vill}{\mathsf{V}}
\newcommand{\Proj}{\operatorname{\Pi}}
\newcommand{\upperb}{\operatorname{ub}}
\newcommand{\lowerb}{\operatorname{lb}}
\newcommand{\Discomfort}{\mathsf{D}}
\newcommand{\Cost}{\mathsf{C}}
\newcommand{\argmin}{\operatornamewithlimits{argmin}}
\newcommand{\N}{\mathbb{N}}
\newcommand{\R}{\mathbb{R}}

\usepackage[sorting=nyt,natbib=true,backend=biber,doi=false,url=false,isbn=false,eprint=true,giveninits=true,maxbibnames=99, style=alphabetic]{biblatex}
\usepackage{indentfirst}
\usepackage{fullpage}
\setlist[itemize]{nosep}
\setlist[enumerate]{nosep}    

\begin{document}

\maketitle

\begin{abstract}
We study the joint design of a time-of-use (ToU) tariff (the partition of the day into pricing blocks and the price level attached to each block) by a profit-maximizing electricity retailer facing flexible clients who may refuse the offer in favour of an outside option. The problem is modeled as a bilevel program in which each client, at the lower level, reshapes its consumption under pointwise bounds and a fixed daily energy total, and simultaneously chooses a participation level penalized by a quadratic regularizer. We prove the client's best response is computable in loglinear time. We derive two equivalent single-level reformulations: a closed-form continuous one tackled via stochastic gradient ascent (using a continuous reparametrization of block start times and durations) and a mixed-integer nonlinear program solved via dedicated solvers. The stochastic ascent scheme provides Clarke-critical accumulation points almost surely, in a scalable way. On a case study from the French retail electricity market with $5000$ simulated households, the gradient method matches a commercial mixed-integer solver on the pointwise problem to within $0.3\%$ of the optimal profit while running two to eight times faster on $100$ representative clients. Restricting unconstrained pointwise prices to admissible ToU schedules yields an estimated profit loss of $1.3\%$- $13.8\%$ for two-period tariffs and $0.6\%$-$6.6\%$ for three-period tariffs on the instances considered.

\smallskip

\end{abstract}

\section{Introduction}
\label{sec:intr}

\subsection{Context and motivation}

The opening of electricity markets to competition has altered the management problems faced by retailers. In particular, where tariffs were once set by regulators on the basis of average costs, retailers must now actively design pricing schemes that simultaneously recover procurement costs, attract and retain clients, and comply with an evolving regulatory landscape. As in other businesses, the announced prices will have a significant impact on both the retailer and various customers with this difference that electricity is a vital resource that one can not go without. One way to create a pricing menu is to temporally differentiate prices.

\smallskip

One such instrument is the time-of-use (ToU) tariff, which divides the day into a small number of periods (peak and off-peak being the canonical example) and charges a different price in each. This temporal differentiation serves a dual purpose. From the retailer's perspective, it reflects the varying cost of procuring electricity on wholesale markets, where prices can fluctuate substantially between daytime and nighttime hours. From the system's perspective, it sends a price signal that can discourage consumption during periods of high network stress, reducing the need for costly peaking capacity. The practical appeal of ToU tariffs lies precisely in their simplicity: unlike real-time pricing, which exposes clients to continuous price fluctuations, a two- or three-period ToU schedule is transparent, predictable, and easy to bill. It is therefore no coincidence that ToU structures form the backbone of most residential demand response programs worldwide~\cite{albadi2008summary, vardakas2015survey}, e.g., the long-established \emph{heures pleines, heures creuses} schedule in the French retail electricity market (peak and off-peak hours).

\smallskip

Properly setting ToU prices requires anticipating how clients will respond and how far they will shift consumption away from expensive periods, given that each has only a limited capacity to adjust. This behavioral response, known as \emph{demand response} \cite{albadi2008summary,vardakas2015survey}, is both the point of the tariff and its main complication from the retailer's perspective. A retailer who ignores it will systematically overestimate revenues in peak periods and underestimate them in off-peak periods. Furthermore, the notions of peak and off-peak periods are implicitly linked to high and low consumption periods, thus implicitly ``redefined" under significant tariff changes. 
Modelling this faithfully, however, forces one to confront a hierarchical decision problem: the retailer chooses prices knowing that clients will re-optimize their own consumption in response, subject to their own constraints and preferences. Such a hierarchy is captured by so called bilevel optimization, e.g., \cite{bard2013practical,dempe2002foundations,dempe2015bilevel, colson2005bilevel, kleinert2021survey}. This class of optimization problems allows one to represent so called leader-follower games. At the upper (or leader) level, the retailer selects both the temporal structure of the tariff (which hours belong to which pricing block) and the price level associated with each block, with the objective of maximizing profit. At the lower (or follower) level, each client decides how to allocate their electricity usage across the day. Load shifting is possible but not free: deviating from one's baseline (ideal) consumption profile carries a cost, which can be modeled through a quadratic disutility that penalizes the magnitude of the shift. Consumption must also respect physical bounds; there is a minimum and a maximum load in each period, and the total daily energy consumption cannot be changed. Beyond these operational constraints, clients retain the option to simply reject the proposed tariff and fall back on a default external offer. This participation constraint ensures that the retailer's pricing strategy must be competitive against the outside option, which could be a regulated tariff or a competitor's offer.

\smallskip

The resulting problem adds a coupling on top of the inherent difficulty of each level individually, making it algorithmically more challenging. Indeed, the retailer's optimization problem is nonconvex and nonsmooth since the client's response is, in general, a nonsmooth function of the announced prices, while the partitioning of the day into pricing blocks introduces a combinatorial aspect. This work proposes a comprehensive modelling framework covering this situation as well as effective solution methods.

\subsection{Contribution}
We study the design of time-of-use (ToU) prices in which a profit-maximizing retailer faces a set of flexible clients who each decide whether to accept the proposed offer and, if they do, optimize their consumption to minimize the sum of their bill and the discomfort of shifting consumption. The main results are the following.

\begin{itemize}[leftmargin=0pt, itemindent=*]
\item \emph{Model.} We formulate the ToU pricing problem as a bilevel program in which the retailer chooses both the partition of the day into pricing blocks and the price level of each block, while every client decides whether to accept the offer and how to reshape its consumption under pointwise bounds and a fixed daily energy total. 
\item \emph{Follower's optimal response.} We characterize the follower's optimal response in closed-form: it decouples into a consumption choice and a participation choice, each a Euclidean projection computable in $O(T\log{T})$ operations (Theorem~\ref{thm:cp:closed-form}, Corollary~\ref{cor:cp:closed-form-complexity}). Theorem~\ref{thm:cp:kkt-conditions} derives a sufficient equilibrium system with explicit big-$M$ bounds on all multipliers under a mild non-degeneracy condition (Assumption~\ref{assu:cons}), yielding two equivalent single-level reformulations: a closed-form one~\eqref{model:cp:closed-form} and a mixed-integer one~\eqref{model:cp:kkt-conditions}. 
\item \emph{Leader's problem theory.} The upper level objective is proven to be an elementary selection in the sense of Bolte and Pauwels~\cite{bolte2020mathematical}, and its selection derivative is evaluable in $O(T\log{T})$ operations (Theorem~\ref{thm:pointwise-selection}), when the time-of-use constraints are dropped or relaxed through reparametrization. Consequently, any accumulation point of the stochastic gradient ascent algorithm is almost surely Clarke critical (Corollary~\ref{corollary:sgd-convergence}).
\item \emph{Price of representability.} To quantify the performance gap between the schedule-free pointwise model and the time-of-use model, we estimate the \emph{price of representability}, defined as the profit loss incurred by restricting unconstrained pointwise prices to admissible ToU schedules. 
\item \emph{Case study: French retail electricity market.} We conduct numerical experiments on datasets generated by EDF R\&D's SMACH simulator~\cite{Huraux}, which simulates 5,000 residential households clustered into distinct profiles, with a distribution based on real data. On 100 representative clients, our gradient-based method matches the profit obtained by a commercial mixed-integer solver on the pointwise problem to within $0.3\%$ of optimality, while running two to eight times faster. We further quantify the cost of restricting unconstrained pointwise prices to ToU tariffs: on the instances considered, this restriction induces an estimated profit loss of $1.3\%$--$13.8\%$ for two-period tariffs and $0.6\%$--$6.6\%$ for three-period tariffs. We report the price menus obtained by the model and compare the solution methods across different cluster sizes, flexibility parameters, and numbers of ToU blocks. 
\end{itemize}
\subsection{Related Work}

The design of time-of-use tariffs requires understanding how clients respond to prices, and casting that response within a hierarchical optimization. This has been approached from different angles:
\begin{enumerate}[leftmargin=0pt, itemindent=*]
\item \emph{Demand response.} Broadly defined as the adjustment of end-use consumption to time-varying prices, both motivates ToU tariffs and complicates their design. The surveys of Albadi and El-Saadany~\cite{albadi2008summary} and Vardakas et al.~\cite{vardakas2015survey} emphasize that the benefit a tariff extracts depends on demand elasticity, flexible-load availability, and metering sophistication. They also note that ignoring this response systematically misestimates revenue across peak and off-peak periods. Flath~\cite{flath2013optimization} proposed a mixed-integer program to mitigate the complexity arising from hourly real-time prices while retaining the flexibility to match the underlying market dynamics. He cast the design of time-of-use tariffs as the projection of consumption costs onto a discrete set of time blocks, treating the partition of the day as a decision, whereas Yang et al.~\cite{yang2012game} took the periods as exogenous and adopted a game-theoretic view of the utility--client interaction. This contrast between periods as a decision and periods as data is one axis along which we positioned our contribution.
\item \emph{Bilevel optimization and Stackelberg games.} The retailer commits to a tariff and the clients optimize in response; this hierarchy is naturally a Stackelberg game and, in optimization terms, a bilevel program~\cite{von1952theory}. Bilevel optimization has several applications in different domains, including transportation, energy, and machine learning, the interested reader is referred to~\cite{bard2013practical, sinha2017review, dempe2020bilevel} for further details. Such models are NP-hard even in the case of linear lower and upper levels~\cite{buchheim2023bilevel}. Several bilevel models of pricing with and without ToU have been proposed. They differ by their description of the leader and follower, and by their solution approaches. 
\begin{enumerate}[leftmargin=0pt, itemindent=*]
\item \emph{Leader's model.} Since Zugno et al.~\cite{zugno2013bilevel} embedded tariff design in a demand-response market this way, the leader--follower lens has become standard, and prior work divides according to what the retailer may decide. The larger group fixes the partition and optimizes only price levels, most often by replacing the lower level with its KKT conditions: Kov{\'a}cs~\cite{kovacs2016bilevel, kovacs2019bilevel} gave early formulations for time-variant tariffs and a day-ahead multi-follower extension with battery-equipped prosumers; Besan{\c{c}}on et al.~\cite{besanccon2020bilevel} and Anjos et al.~\cite{anjos2021optimal} treated time-and-level-of-use tariffs via mixed-integer reformulations; and Grimm et al.~\cite{grimm2021optimal} and Abate et al.~\cite{abate2024dynamic} compared pricing schemes and market structures within the same framework. A smaller group co-designs periods and prices, but only heuristically, through evolutionary or genetic algorithms~\cite{alves2020optimizing, venkatraman2022optimal}, while other studies relax lower-level global optimality~\cite{soares2020designing} or add a regulatory tier~\cite{aussel2020trilevel}; broader surveys are given in~\cite{antunes2020bilevel, huang2025review}. The bilevel ToU model is thus mature in its treatment of price optimization, yet the combinatorial design of the temporal partition has not been optimized jointly with price levels inside an exact single-level reformulation.
\item \emph{Follower's model.} Beyond what the retailer decides, a further question is whether the client accepts the tariff at all. Within these models the client decides only how to consume, not whether to participate; yet it is the outside option that makes the problem economically meaningful, as a client-level individual-rationality constraint. This is usually left implicit, the follower being assumed to stay while the average price is below a threshold~\cite{kovacs2019bilevel, soares2020designing}. The closest explicit precedent is Crowley et al.~\cite{crowley2024energy}, who enforce budget balance and individual rationality as leader constraints in an energy community, later relaxing perfect leader knowledge through online learning~\cite{crowley2025learning}. Methodologically, Jacquet et al.~\cite{jacquet2024quadratic} study quadratic regularization of the follower's contract choice, providing the smoothed responses on which we build without considering the flexibility and the time-of-use schedules. An explicit, client-level acceptance decision has not, however, been combined with a combinatorially designed ToU tariff.
\item \emph{Solution approaches.} Because bilevel models embed an optimization problem as a constraint, they resist standard solvers, and the literature offers two remedies. The first replaces the lower level with its Karush-Kuhn-Tucker (KKT) conditions to yield a single-level Mathematical Program with Complementarity Constraints (MPCC); this underlies most of the ToU works cited above and our indirect method, and once the leader also partitions the day it falls under mixed-integer bilevel programming~\cite{kleinert2021survey}. The second is gradient-based: when the lower-level solution varies smoothly with the leader's decision, one differentiates through its optimality conditions to obtain a hypergradient of the upper-level objective~\cite{gould2016differentiating}. This route classically requires lower-level strong convexity, which our nested lower problem lacks because of its bilinear terms; Chen et al.~\cite{chen2023bilevel} delineate where hypergradient methods break down once this assumption is dropped, while penalty- and envelope-based approaches recover convergence guarantees for merely convex or nonconvex lower levels~\cite{shen2023penalty, lumei2024firstorder, liu2024moreau}. Our follower problem is likewise nonconvex, but its solution is a composition of projections onto boxed simplices, which are piecewise differentiable and evaluable in log-linear time; this structure lets us appeal to Bolte and Pauwels~\cite{bolte2020mathematical} to guarantee that almost surely any accumulation point of our stochastic gradient scheme is Clarke critical.
\end{enumerate}
\end{enumerate}

Table~\ref{tab:positioning} summarizes the position of this work relative to the models discussed above, along the dimensions that define our contribution. To the best of our knowledge, no prior work has considered the interaction between the three ingredients (the partition of the day into pricing blocks, the optimization of the price level of each block, the participation and the flexibility of clients) within a single model.

\begin{table}[!htbp]
\centering
\small
\begin{adjustbox}{max width=\textwidth}
\begin{tabular}{@{}l c c c l c@{}}
\toprule
Reference & Partition & Price levels & Participation & Lower-level treatment & Exact \\
 & decided? & decided? & modelled? & & single-level? \\
\midrule
Flath~\cite{flath2013optimization} & Yes & Yes & No & clustering / projection & No \\
Yang et al.~\cite{yang2012game} & No & Yes & No & game-theoretic best response & No \\
Zugno et al.~\cite{zugno2013bilevel} & No & Yes & No & KKT & Yes \\
Kov{\'a}cs~\cite{kovacs2016bilevel, kovacs2019bilevel} & No & Yes & Implicit & KKT & Yes \\
Besan{\c{c}}on et al.~\cite{besanccon2020bilevel} & No & Yes & No & KKT / mixed-integer & Yes \\
Anjos et al.~\cite{anjos2021optimal} & No & Yes & No & KKT / mixed-integer & Yes \\
Grimm et al.~\cite{grimm2021optimal} & No & Yes & No & KKT & Yes \\
Abate et al.~\cite{abate2024dynamic} & No & Yes & No & KKT & Yes \\
Alves et al.~\cite{alves2020optimizing} & Yes & Yes & No & heuristic (evolutionary)  & No \\
Venkatraman et al.~\cite{venkatraman2022optimal} & Yes & Yes & No & heuristic (genetic) & No \\
Soares et al.~\cite{soares2020designing} & No & Yes & Implicit & relaxed optimality & No \\
Crowley et al.~\cite{crowley2024energy} & No & Yes & Community-level & KKT  & Yes \\
\midrule
This work & Yes & Yes & Client-level, explicit & closed form \emph{and} optimality system & Yes \\
\bottomrule
\end{tabular}
\end{adjustbox}
\caption{Positioning against the closest bilevel time-of-use pricing models. ``Partition decided?'' indicates whether the assignment of time steps to pricing periods is a decision variable rather than data; ``Participation modelled?'' whether the client's accept/reject decision appears explicitly. ``Exact single-level?'' indicates whether the bilevel problem is replaced by an equivalent single-level program, not whether that program is solved to global optimality.}
\label{tab:positioning}
\end{table}

\smallskip

The remainder of the paper is organized as follows. Section~\ref{sec:mod} formalizes the bilevel model, and its components. Section~\ref{sec:analytical} derives the client response mappings and the retailer's problem, and presents the reformulations of the bilevel problem. Section~\ref{sec:algorithmic} describes our solution methodology, including the descent algorithm for pointwise pricing and the projection procedure for ToU schedules. Section~\ref{sec:num} reports the numerical experiments and sensitivity analysis. Section~\ref{sec:concl} concludes with implications, limitations, and directions for future research.

\section{Electricity pricing model}
\label{sec:mod}
Building on the bilevel perspective introduced in Section \ref{sec:intr}, we now formalize the electricity pricing model used throughout this paper. The complete notation is summarized in Appendix~\ref{app:notation}.

\subsection{Retailer's optimization problem}
The electricity retailer maximizes its profit by jointly determining the price levels $q = (q_n)_{n \in \llbracket N \rrbracket} \in \R^N$ and the time-of-use schedule $h = (h_{t,n})_{t \in \llbracket T \rrbracket, n \in \llbracket N \rrbracket} \in \{0,1\}^{T \times N}$, while accounting for each client's best response to the proposed prices. A pair $(q,h)$ is seen by the clients only through the effective price vector $p(q,h) \in \R^T_+$ it induces, obtained by aggregating the price levels along the schedule:
\begin{equation} \label{eq:effective-price}
p_t(q,h) := \textstyle\sum_{n=1}^N q_n h_{t,n}, \qquad \forall t \in \llbracket T \rrbracket,
\end{equation}
where $h_t = (h_{t,n})_{n \in \llbracket N \rrbracket}$ denotes the $t$-th row of the schedule. The retailer's optimization problem reads:
\begin{align*} \tag{BOP} \label{model:retailer}
\max_{q, h, x, y} \quad & \textstyle\sum_{k=1}^K \left[ \Bill(p(q,h), x_k) - \Cost(x_k) \right] y_k w_k \\
\text{s.t.} \quad & q \in \mathcal{Q}, \quad h \in \mathcal{H}, \\
& (x_k,y_k) \text{ solves $\eqref{model:cp}$ at $p=p(q,h)$}, \qquad \forall k \in \llbracket K \rrbracket. 
\end{align*}

The optimization variables are the price levels $q$, constrained to lie in the feasible set $\mathcal{Q}$, and the schedule $h$, constrained to lie in the schedule set $\mathcal{H}$; both sets are defined in Section~\ref{subsec:feasible-pricing} and encode, respectively, the bounds and the ordering imposed on individual prices, and the institutional and operational constraints the schedule must satisfy. The problem is subject to the clients' best responses, captured by the lower-level problem~\eqref{model:cp} through the variables $\big(x_k=(x_{t,k})_{t\in \llbracket T \rrbracket}, y_k\big) \in \R^T \times [0,1]$, which denote, respectively, the optimal consumption $x_{t,k}$ at each instant $t \in \llbracket T \rrbracket$ and the participation decision $y_k$ of client $k$. The retailer must anticipate these responses when setting prices, which is what makes the problem bilevel. The functions $\Bill$ and $\Cost$ denote, respectively, the client's bill and the retailer's supply cost. The bill is the total amount paid by the client, that is the inner product of the effective price vector $p=p(q,h)$ and the consumption vector $x_k$:
\begin{equation} \label{eq:bill}
\Bill(p, x) := \langle p, x\rangle = \textstyle\sum_{t=1}^T p_t x_t.
\end{equation}
The parameter $w_k$ is the weight of client $k$ in the profit function, interpreted as a measure of the client's importance or market share. The supply cost $\Cost$ is left unspecified, subject to the following standing hypotheses.

\begin{assumption}[Standing hypotheses on the supply cost] \label{assu:cost}
The supply cost $\Cost : \R^T \to \R$ satisfies:
\begin{enumerate}[leftmargin=0pt, itemindent=*]
\item $\Cost$ is continuous;
\item $\Cost$ is Lipschitz continuous on the consumption box $[x^{\lowerb}_k, x^{\upperb}_k]$, with constant $L^{\Cost}_k$, for every $k \in \llbracket K \rrbracket$;
\item $\Cost$ is an elementary selection with a known selection derivative computable in $O(T)$ elementary operations (see Appendix~\ref{app:elementary-selections} for the definitions).
\end{enumerate}
\end{assumption}

Item~(1) is used for the existence result of Corollary~\ref{cor:optpes}, item~(2) for the Lipschitz estimates of the hyper-objective, and item~(3) for the elementary-selection property of Theorem~\ref{thm:pointwise-selection}. A linear cost $\Cost(x) = \textstyle\sum_{t=1}^T c_t x_t$, which is the specification used in the numerical experiments of Section~\ref{sec:num}, satisfies all three with $L^{\Cost}_k = \|c\|$.

\smallskip

Following the usual convention for bilevel programs, \eqref{model:retailer} is written in its optimistic form, in which the leader may select among the lower-level's optimal responses, the one that aligns with the upper-level objective; Corollary~\ref{cor:optpes} shows that this choice is immaterial here, since the optimistic and pessimistic formulations coincide.

\subsection{Client's optimization problem}
Given the effective price signal $p = p(q,h)$ induced by the period prices $q \in \mathcal{Q}$ and the schedule $h \in \mathcal{H}$ defined above, each client $k$ jointly determines its optimal consumption profile $x_k$ and decides whether to accept the proposed offer $p$ (modeled by the variable $y_k$) by solving:
\begin{align*} \tag{$\CP_k(p)$} \label{model:cp}
\min_{x_k, y_k} \quad & \left[\Bill(p, x_k) + \Discomfort_k(x_k, x_k^0)\right] y_k
+ \Vill_k^0 (1 - y_k) + \tfrac{1}{2\beta_k} y_k^2 \\
\text{s.t.} \quad & x_k \in \mathcal{X}_k, \\
& 0 \le y_k \le 1,
\end{align*}
where $\mathcal{X}_k$ is the consumption feasible set given by the boxed simplex:
\begin{align} \label{set:consumption-feasible}
\mathcal{X}_k = \left\{ x_k \in \R^{T} \;\middle|\; 
\begin{aligned}
& \textstyle\sum_{t=1}^T x_{t,k} = \textstyle\sum_{t=1}^T x^0_{t,k}, \\
& x^{\lowerb}_{t,k} \le x_{t,k} \le x^{\upperb}_{t,k}, && \forall t \in \llbracket T \rrbracket
\end{aligned}
\right\}.
\end{align}
The variable $x_k$ is the consumption profile of client $k$, with $x_{t,k}$ denoting the consumption at time $t$. The bounds $x^{\lowerb}_{t,k}$ and $x^{\upperb}_{t,k}$, which satisfy $0 \le x^{\lowerb}_{t,k} \le x^{\upperb}_{t,k}$, limit how far the consumption may be shifted at each time step, while the equality constraint ensures that the total consumption matches a baseline (ideal) consumption profile, denoted by $x_k^0$. The disutility of load shifting is the nominal (quadratic) distance to $x_k^0$:
\begin{equation*}
\Discomfort_k(x_k, x_k^0) = \tfrac{1}{2\alpha_k} \|x_{k} - x^0_{k}\|_2^2,
\end{equation*}
where $\| \cdot \|_2$ is the Euclidean norm and $\alpha_k > 0$ is a flexibility parameter: the smaller $\alpha_k$, the larger the penalty on deviations, and hence the less flexible the client. For the discomfort to be expressed in $\EUR$, the parameter $\alpha_k$ carries units of $\mathrm{kWh}^2/\EUR$; equivalently, $\alpha_k$ is the amount of energy that client $k$ is willing to displace in response to a unit price differential, which is exactly the reading given by the closed-form of Theorem~\ref{thm:cp:closed-form}. 

\smallskip

The variable $y_k$ is the participation decision, with $y_k = 1$ meaning that the offer is accepted and $y_k = 0$ that it is refused in favor of the outside option. The constant $\Vill_k^0$ is the value of that outside option, that is, the optimal value of the client $k$'s problem under the outside tariff $p_k^0$, discomfort included:
\begin{equation*}
\Vill_k^0 = \min_{x_k \in \mathcal{X}_k} \left\{ \Bill(p_k^0, x_k) + \Discomfort_k(x_k, x_k^0) \right\}.
\end{equation*} 
Client $k$ compares the reference value $\Vill_k^0$ with the value of the proposed offer, which is the optimal value of the client $k$'s problem under the proposed tariff $p$, discomfort included:
\begin{equation} \label{model:cco} \tag{$\CCO_k(p)$}
\Vill_k(p) = \min_{x_k \in \mathcal{X}_k} \left\{ \Bill(p, x_k) + \Discomfort_k(x_k, x_k^0) \right\}.
\end{equation}
Since $y_k \ge 0$, the inner minimization over $x_k$ in~\eqref{model:cp} may be carried out independently of $y_k$ and absorbed into $\Vill_k(p)$, so that the participation decision is the solution of the one-dimensional problem
\begin{align} \label{model:cch} \tag{$\CCH_k(p)$}
\min_{y_k \in [0,1]} \quad & \Vill_k(p) y_k + \Vill_k^0 (1 - y_k) + \tfrac{1}{2\beta_k} y_k^2 .
\end{align}
Problems~\eqref{model:cco} and~\eqref{model:cch} are the two components of the sequential decomposition on which the whole analysis rests; Theorem~\ref{thm:cp:closed-form} makes this precise.

\smallskip

The quadratic term $\tfrac{1}{2\beta_k} y_k^2$, with $\beta_k > 0$, regularizes the accept/reject decision, replacing a binary choice by a continuous participation level $y_k \in [0,1]$. The parameter $\beta_k$ carries units of $1/\EUR$, equivalently, $1/\beta_k$ is a monetary quantity, namely the width of the band over which acceptance moves from $0$ to $1$. Client $k$ refuses outright when the offer is worse, accepts with certainty once the offer undercuts the outside option by more than $1/\beta_k$, and is hesitant in between; see Figure~\ref{fig:client-choice}. A larger $\beta_k$ therefore yields a sharper, more deterministic decision, and $\beta_k \to \infty$ recovers the hard individual-rationality constraint.

\begin{figure}[!htp]
\centering
\begin{tikzpicture}[x=1.2cm, y=3cm, >=stealth]

\draw[gray!20] (-0.25,-0.1) grid (4.2, 1.3);

\draw[->, thick] (-0.25, -0.1) -- (4.2, -0.1); 
\draw[->, thick] (-0.25, -0.1) -- (-0.25, 1.3);

\foreach \x in {0,1,...,4} {
\draw (\x, -0.13) -- (\x, -0.07) node[below=0.15cm, font=\small] {\x};
}

\foreach \y in {0, 0.2, 0.4, 0.6, 0.8, 1.0, 1.2} {
\draw[thick] (-0.28, \y) -- (-0.22, \y) node[left=0.1cm, font=\small] {\y};
}

\draw[ultra thick, myblue] 
(-0.25, 1) -- (1, 1) 
-- (3, 0)  
-- (4.2, 0); 

\draw[dashed, gray] (1, -0.1) -- (1, 1.05) node[above, black] {$\Vill_k^0 - \tfrac{1}{\beta_k}$};
\draw[dashed, gray] (3, -0.1) -- (3, 1.05) node[above, black] {$\Vill_k^0$};

\node[rotate=90] at (-1.1, 0.6) {Client participation level};
\node at (2, -0.3) {Client bill and discomfort};

\end{tikzpicture}
\caption{Client $k$'s participation level $y_k$ as a function of the value $\Vill_k(p)$ of the offer, for $\Vill_k^0 = 3$ and $\beta_k = 0.5$.}
\label{fig:client-choice}
\end{figure}
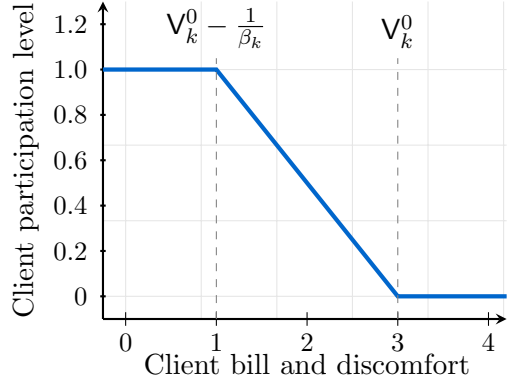

\smallskip 

The client's problem therefore couples the choice of consumption profile with the participation decision, balancing cost minimization, disutility from load shifting, and the attractiveness of the offer relative to the outside option.

\subsection{Feasible pricing strategies}
\label{subsec:feasible-pricing}
Under time-of-use (ToU) pricing, a price scheme is described by two ingredients: a set of period prices $q = (q_n)_{n \in \llbracket N \rrbracket}$ and a schedule $h$ that assigns each time step $t \in \llbracket T \rrbracket$ to one period $n \in \llbracket N \rrbracket$. The effective price seen by a client at time $t$ is then the price $p_t(q,h)$ of the period active at $t$, as defined in~\eqref{eq:effective-price}. We now specify the two feasible sets $\mathcal{Q}$ and $\mathcal{H}$ from which these ingredients are drawn.

\smallskip

The periods $n \in \llbracket N \rrbracket$ are ordered from the cheapest to the most expensive, so that the set of admissible price levels takes the form:
\begin{equation} \label{set:price-levels}
\mathcal{Q} := \big\{q \in \R^{N} \;\big|\; p^{\lowerb} \le q_1 \le \cdots \le q_N \le p^{\upperb}\big\},
\end{equation}
where $p^{\lowerb}$ and $p^{\upperb}$ are the regulatory bounds on the unit price, assumed to satisfy $0 \le p^{\lowerb} < p^{\upperb}$; we write $\mathcal{P} := [p^{\lowerb}, p^{\upperb}]^{T}$ for the corresponding box of effective price vectors. The schedule set $\mathcal{H}$, on the other hand, groups time steps into periods and must respect institutional and operational constraints, for example: regulatory off-peak windows, fixed period durations, a limited number of switches per day, or minimum block lengths. Admissible schedules must in addition assign exactly one period to each time step. We consider the following model:
\begin{align} \label{set:schedules}
\mathcal{H} = \left\{ h \in \{0,1\}^{T \times N} \;\middle|\;
\begin{aligned}
& \textstyle\sum_{n=1}^N h_{t,n} = 1, && \forall t \in \llbracket T \rrbracket, \\
& \textstyle\sum_{t=1}^T |h_{t,n}-h_{t-1,n}| \le 2S_n, && \forall n \in \llbracket N \rrbracket, \\
& \textstyle\sum_{t=1}^T h_{t,n} = L_n, && \forall n \in \llbracket N \rrbracket
\end{aligned}
\right\},
\end{align}
where each period $n \in \llbracket N \rrbracket$ is characterized by its length $L_n$ and its maximum number of subperiods $S_n$. Throughout, the time index is understood cyclically: we identify $T$ with the discrete torus $\mathbb{Z}/T\mathbb{Z}$, so that $t-1$ is taken modulo $T$ and, in particular, $h_{0,n} = h_{T,n}$. This reflects the daily periodicity of the tariff and makes the switch count below consistent across the day boundary: a period that is active at the end of the day and at its beginning forms a single subperiod rather than two. The linking constraint $\sum_{n=1}^N h_{t,n} = 1$ ensures that exactly one period is active at each time step $t$; the constraint $\sum_{t=1}^T |h_{t,n}-h_{t-1,n}| \le 2S_n$ counts the switches in and out of period $n$ and hence ensures that period $n$ consists of at most $S_n$ subperiods; and the last constraint fixes its total duration to $L_n$. Introducing the auxiliary variables $\delta_{t,n}$, the switch-counting constraint is equivalent to the following linear inequalities:
\begin{align} \label{eq:switch-counting}
h_{t,n} - h_{t-1,n} &\le \delta_{t,n}, \qquad h_{t-1,n} - h_{t,n} \le \delta_{t,n}, \qquad &\forall t \in \llbracket T \rrbracket, n \in \llbracket N \rrbracket, \\
\sum_{t=1}^T \delta_{t,n} &\le 2S_n, &\qquad \forall n \in \llbracket N \rrbracket. \notag
\end{align}

The data are assumed to satisfy $1 \le S_n \le L_n$ and the consistency condition $\sum_{n=1}^N L_n = T$, which is necessary for $\mathcal{H}$ to be nonempty and which we assume throughout; it is also sufficient, since assigning the periods in index order as $N$ consecutive cyclic blocks of respective lengths $L_n$ produces an element of $\mathcal{H}$.

\begin{example}
Figure \ref{fig:periods-subperiods} models the schedule $h$ as a partition of the time axis into periods, each of which has at most two subperiods $S_1=S_2=S_3=2$ with the lengths $L_1=L_2=9$ and $L_3=6$, where period (1) represents the off-peak hours, period (2) the shoulder hours, and period (3) the peak hours. The figure illustrates how the time steps are allocated to each period and how the subperiods are distributed within each period and notes the fact that the period between 21:00 and 3:00 is a single subperiod considering that time is seen as periodic.

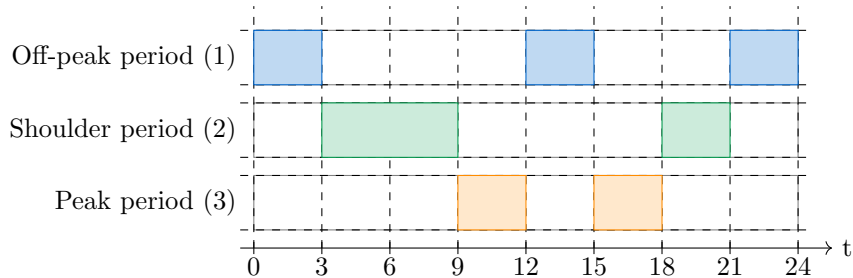
\begin{figure}[!ht]
\centering
\begin{tikzpicture}[x=0.9cm,y=0.8cm, font=\small]

\draw[->] (0.8,0.0) -- (9.5,0.0) node[right]{t};
\foreach \t in {1,...,9} {
    \draw (\t,-0.1) -- (\t,0.1);
}

\foreach \t in {0,...,8} {
    \node at (\t+1.,-0.3) {\pgfmathtruncatemacro{\tick}{3*\t}\tick};
}

\def\yone{3.6}
\def\ytwo{2.4}
\def\ythr{1.2}
\def\h{0.9} 

\draw[draw=gray] (1,\yone) rectangle (9,\yone-\h);
\draw[draw=gray] (1,\ytwo) rectangle (9,\ytwo-\h);
\draw[draw=gray] (1,\ythr) rectangle (9,\ythr-\h);

\foreach \x in {1,...,9} {
    \draw[dashed] (\x,0.0) -- (\x,4);
}

\draw[dashed] (0.8,\yone) -- (9.2,\yone);
\draw[dashed] (0.8,\yone-\h) -- (9.2,\yone-\h);
\draw[dashed] (0.8,\ytwo) -- (9.2,\ytwo);
\draw[dashed] (0.8,\ytwo-\h) -- (9.2,\ytwo-\h);
\draw[dashed] (0.8,\ythr) -- (9.2,\ythr);
\draw[dashed] (0.8,\ythr-\h) -- (9.2,\ythr-\h);

\draw[draw=myblue, fill=myblue!25] (1,\yone) rectangle (2,\yone-\h);
\draw[draw=myblue, fill=myblue!25] (5,\yone) rectangle (6,\yone-\h);
\draw[draw=myblue, fill=myblue!25] (8,\yone) rectangle (9,\yone-\h);
\node[anchor=east] at (0.9,\yone-\h/2) {Off-peak period (1)};

\draw[draw=mygreen, fill=mygreen!20] (2,\ytwo) rectangle (4,\ytwo-\h);
\draw[draw=mygreen, fill=mygreen!20] (7,\ytwo) rectangle (8,\ytwo-\h);
\node[anchor=east] at (0.9,\ytwo-\h/2) {Shoulder period (2)};

\draw[draw=myorange, fill=myorange!20] (4,\ythr) rectangle (5,\ythr-\h);
\draw[draw=myorange, fill=myorange!20] (6,\ythr) rectangle (7,\ythr-\h);

\node[anchor=east] at (0.9,\ythr-\h/2) {Peak period (3)};

\end{tikzpicture}
\caption{Example assignment; $N=3$, $S_1 = S_2 = S_3 = 2$, $L_1 = L_2 = 9$, and $L_3 = 6$.}
\label{fig:periods-subperiods}
\end{figure}
\end{example} 

Pointwise pricing, in which the retailer is free to propose a distinct price at each instant $t \in \llbracket T \rrbracket$, is the special case obtained by taking one period per time step, i.e. $N = T$ with $L_n = S_n = 1$ for every period $n \in \llbracket N \rrbracket$, so that $h$ is a permutation matrix and $p(q,h)$ ranges over the whole box $\mathcal{P}$: the ordering imposed in~\eqref{set:price-levels} is then a mere labeling convention, absorbed by the choice of the permutation.

\section{Mathematical properties and single-level reformulation}
\label{sec:analytical}

For a given price signal $p$, the client's problem~\eqref{model:cp} is a nested optimization problem: the client chooses its consumption profile $x_k$ and decides whether to accept the retailer's offer or take the outside option. The two decisions are coupled through the product $\Vill_k(p) y_k$, so the program is nonconvex. Nevertheless, it decomposes into two subproblems: a \emph{consumption choice problem} and a \emph{participation choice problem}. The former is a quadratic program with linear constraints; the latter is a one-dimensional quadratic problem given the consumption choice. Theorem~\ref{thm:cp:closed-form} provides closed-form characterizations of their solutions.

\begin{assumption}
\label{assu:cons}
For all $t \in T$ and $k \in K$, we assume that the baseline consumption $x_k^0$ satisfy $x^{\lowerb}_{t,k} \leq x^0_{t,k} \leq x^{\upperb}_{t,k}$, and satisfy:
\begin{equation*}
\sum_{t=1}^T x^{\lowerb}_{t,k} < \textstyle\sum_{t=1}^T x^0_{t,k} < \textstyle\sum_{t=1}^T x^{\upperb}_{t,k}.
\end{equation*}
\end{assumption}

The condition $x^{\lowerb}_{t,k} \leq x^0_{t,k} \leq x^{\upperb}_{t,k}$ ensures that the baseline consumption is feasible, and the second condition ensures that there are more than one point in the set. Throughout the remainder of this paper, we suppose that Assumption \ref{assu:cons} remains satisfied.

\begin{theorem} \label{thm:cp:closed-form}
Given a client $k$ and a price signal $p$, \eqref{model:cp} admits an optimal solution, satisfying the following.
\begin{enumerate}[leftmargin=0pt, itemindent=*]
\item The optimal participation choice is unique and admits the closed-form:
\begin{equation*}
y_k^\star(p) = \Proj_{[0,1]}\!\left(\beta_k\big(\Vill_k^0-\Vill_k(p))\right),
\end{equation*}
where $\Vill_k(p)$ is defined in \eqref{model:cco}.
\item The minimizer of the consumption choice problem~\eqref{model:cco}, whose optimal value is $\Vill_k(p)$, is unique and admits the closed-form:
\begin{equation*}
x_k^\star(p) = \Proj_{\mathcal{X}_k}\!\left(x^0_k - \alpha_k p\right).
\end{equation*}
\item If $y_k^\star(p) > 0$, then $(x_k^\star(p), y_k^\star(p))$ is the unique optimal solution of~\eqref{model:cp}. If $y_k^\star(p) = 0$, the set of optimal solutions is $\mathcal{X}_k \times \{0\}$, and $(x_k^\star(p), 0)$ is one of its elements.
\item The decision $x_k^\star(p)$ is piecewise affine on a polyhedral subdivision of $\R^T$, and the decision $y_k^\star(p)$ is piecewise quadratic on a semi-algebraic refinement of that subdivision.
\end{enumerate}
\end{theorem}

\begin{proof}
The proof exploits the nested structure of \eqref{model:cp}: the client first minimizes $\Bill(p,x_k) + \Discomfort_k(x_k, x_k^0)$ over $x_k$, independently of the participation decision $y_k$, and then chooses $y_k$ giving the resulting value $\Vill_k(p)$.

\begin{enumerate}[leftmargin=0pt, itemindent=*]
\item The participation decision $y_k$ is nonnegative, hence the part of the objective of the client's problem~\eqref{model:cp} that depends on $x_k$ can be embedded into the value function $\Vill_k(p)$ defined in \eqref{model:cco}, and the problem \eqref{model:cp} becomes solely the participation problem~\eqref{model:cch}. The choice of $y_k$ is therefore a one-dimensional quadratic program with linear constraints, which has a unique solution denoted by $y_k^\star(p)$ and it is given by the projection of $\beta_k(\Vill_k^0 - \Vill_k(p))$ onto $[0,1]$.
\item The problem \eqref{model:cco} is a quadratic program with linear constraints, which has a unique solution denoted by $x_k^\star(p)$ and it is given by the projection of $x^0_k - \alpha_k p$ onto the convex set $\mathcal{X}_k$.
\item Going back to the \eqref{model:cp}, we distinguish two cases. If $y_k^\star(p) > 0$, the variable $x_k$ must be chosen to minimize the objective function, which is achieved by $x_k^\star(p)$, and it is unique by the previous argument. If $y_k^\star(p) = 0$, then the objective function is independent of $x_k$, and any choice of $x_k \in \mathcal{X}_k$ yields the same optimal value. In particular, $(x_k^\star(p), 0)$ is an optimal solution.
\item \emph{Piecewise polynomial property.} The projection onto the polyhedron $\mathcal{X}_k$ is piecewise affine over polyhedral regions \cite[Proposition~4.1.4]{facchinei2003finite}, so $x_k^\star(p)$ is piecewise affine. Since $y_k^\star(p)$ is the projection of a piecewise quadratic function of $p$ onto $[0,1]$, it is piecewise quadratic. More precisely, since the projection onto [0,1] can be obtained by comparing that function to $0$ and $1$ (which delimit semi-algebraic regions), it is piecewise quadratic on algebraic regions. \qedhere
\end{enumerate}
\end{proof}

\begin{corollary}
\label{cor:cp:closed-form-complexity}
Given a client $k$ and a price signal $p$, the projection $x_k^\star(p)$ can be computed in $O(T\log{T})$ elementary operations, and given $x_k^\star(p)$, the participation decision $y_k^\star(p)$ can be computed in $O(T)$ elementary operations. Overall, the pair $(x_k^\star(p), y_k^\star(p))$ can be computed in $O(T\log{T})$ elementary operations.
\end{corollary}

\begin{proof}
The projection onto $\mathcal{X}_k$ can be computed by a sorting-based algorithm in $O(T\log{T})$ operations; see, e.g., Condat~\cite{condat2016fast} and Jacquot~\cite[Section~1.C]{jacquot2019game}. The projection onto $[0,1]$ is $O(1)$. Computing $x_k^\star(p)$ first, then $\beta_k(\Vill_k^0 - \Vill_k(p))$ in $O(T)$, and finally projecting onto $[0,1]$ yields an overall $O(T\log{T})$ complexity. \qedhere
\end{proof}

\begin{corollary}
\label{cor:optpes}
The optimistic and pessimistic formulations of \eqref{model:retailer} coincide, and they admit a solution.
\end{corollary}

\begin{proof} 
Client $k$'s contribution to the leader's objective function is given by $[\mathsf B(p,x_k)-\mathsf C(x_k)]y_k w_k$, which vanishes on the set $\mathcal X_k\times\{0\}$. Consequently, the optimistic and pessimistic formulations of \eqref{model:retailer} coincide. Indeed, when $y_k > 0$, the lower-level solution is unique; when $y_k=0$, the objective term vanishes, making any choice of $x_k$ yield the same upper-level payoff. Thus, the response map $p\mapsto(x^\star_k(p),y^\star_k(p))$ defines a valid single-valued selection for the upper-level problem. Substituting these closed-form solutions into \eqref{model:retailer} reduces it to a maximization of a continuous function of $(q,h)$, by Theorem~\ref{thm:cp:closed-form} and continuity of the projections involved, and continuity of the cost function (item (1) of Assumption \ref{assu:cost}). The feasible set is compact: $\mathcal{Q}$ is a closed bounded subset of $\R^{N}$, and $\mathcal{H}$ is a finite set, being a subset of $\{0,1\}^{T \times N}$. Weierstrass's theorem therefore applies and a maximizer exists. \qedhere
\end{proof}

Consider the function $F$ defined on $\R^T$ by:
\begin{equation} \label{model:ho} \tag{$\HO$}
F(p) := \textstyle\sum_{k=1}^K F_k(p) w_k,
\end{equation}
where $F_k(p) := \left[\Bill(p, x_k^\star(p)) - \Cost(x_k^\star(p))\right] y_k^\star(p)$. 

\begin{corollary}[Closed-form single-level reformulation] \label{cor:cp:closed-form}
The bilevel problem~\eqref{model:retailer} is equivalent to the single-level problem obtained by substituting the closed-form solutions $(x_k^\star(p), y_k^\star(p))$ of Theorem~\ref{thm:cp:closed-form} into the upper-level objective, which yields the {closed-form single-level reformulation (CF-SLF)}:
\begin{align*} \label{model:cp:closed-form} \tag{CF-SLF}
\max_{q \in \mathcal{Q}, h \in \mathcal{H}} \quad & F(p(q,h)),
\end{align*}
where $F$ is the hyper-objective defined in~\eqref{model:ho}, and $p(q,h)$ is the effective price signal defined in~\eqref{eq:effective-price}. Equivalence is understood in the sense that the two problems have the same optimal value and the same set of optimal pairs $(q,h)$.
\end{corollary}

\begin{proof}
By Theorem~\ref{thm:cp:closed-form}, $(x^\star_k(p), y^\star_k(p))$ is optimal for~\eqref{model:cp} at every $p=p(q,h)$, so every feasible point of~\eqref{model:cp:closed-form} is feasible for~\eqref{model:retailer} with the same objective value. Conversely, by Corollary~\ref{cor:optpes} any other selection of lower-level optima yields the same upper-level value, so no feasible point of~\eqref{model:retailer} attains more than~\eqref{model:cp:closed-form}. \qedhere
\end{proof}

\begin{proposition}[Lipschitz continuity of the hyper-objective]\label{prop:lipschitz}
Assume item~(2) of Assumption~\ref{assu:cost}. Then $F$ in~\eqref{model:ho} is Lipschitz on $\mathcal{P}=[p^{\lowerb},p^{\upperb}]^T$ with constant $L_F=\sum_{k\in\llbracket K\rrbracket}L^{F}_k \, w_k $, where $L^{F}_k$ is given by:
\begin{equation*}
L^{F}_k = L^{G}_k + \beta_k\|x^{\upperb}_k\|\Bigl(p^{\upperb}\sqrt{T}\|x^{\upperb}_k\|+\max_{x\in[x^{\lowerb}_k,x^{\upperb}_k]}|\Cost(x)|\Bigr).
\end{equation*}
where $L^{G}_k :=\|x^{\upperb}_k\|+\alpha_kp^{\upperb}\sqrt{T}+\alpha_kL^{\Cost}_k$.
\end{proposition}

\begin{proof}
Fix $k\in\llbracket K\rrbracket$ and $p,q\in\mathcal{P}$. The sign conditions on the boxes give the two bounds used repeatedly below:
\begin{equation}\label{eq:box}
\|p\|\le p^{\upperb}\sqrt{T}\quad\text{for }p\in\mathcal{P},
\qquad
\|x^\star_k(p)\|\le\|x^{\upperb}_k\|.
\end{equation}
\begin{itemize}[leftmargin=0pt, itemindent=*]
\item \emph{The participation level.} Danskin's theorem applies at every $p\in\mathcal{P}$ and gives $\nabla\Vill_k(p)=x^\star_k(p)$, whence $\|\nabla\Vill_k\|\le\|x^{\upperb}_k\|$ on $\mathcal{P}$ by~\eqref{eq:box}. Since $\mathcal{P}$ is convex, the mean value inequality upgrades this gradient bound to
\begin{equation*}
|\Vill_k(p)-\Vill_k(q)|\le\|x^{\upperb}_k\|\,\|p-q\|.
\end{equation*}
As $\Proj_{[0,1]}$ is nonexpansive, composing it with the $\beta_k\|x^{\upperb}_k\|$-Lipschitz map $p\mapsto\beta_k(\Vill^0_k-\Vill_k(p))$ shows that $y^\star_k$ is $\beta_k\|x^{\upperb}_k\|$-Lipschitz, with $0\le y^\star_k\le1$.
\item \emph{The pure profit.} Write $G_k(p):=\Bill(p,x^\star_k(p))-\Cost(x^\star_k(p))$. Inserting $\langle q,x^\star_k(p)\rangle$ and using Cauchy--Schwarz, \eqref{eq:box} and the $\alpha_k$-Lipschitzness of $x^\star_k$,
\begin{align*}
|\langle p,x^\star_k(p)\rangle-\langle q,x^\star_k(q)\rangle| &\le \|x^\star_k(p)\|\,\|p-q\|+\|q\|\,\|x^\star_k(p)-x^\star_k(q)\| \\
&\le \bigl(\|x^{\upperb}_k\|+\alpha_kp^{\upperb}\sqrt{T}\bigr)\|p-q\|,
\end{align*}
while $\Cost\circ x^\star_k$ is $\alpha_kL^{\Cost}_k$-Lipschitz as a composition. Adding the two constants gives that $G_k$ is $L^{G}_k$-Lipschitz with $L^{G}_k = \|x^{\upperb}_k\| + \alpha_kp^{\upperb}\sqrt{T} + \alpha_k L^{\Cost}_k$. Moreover, the function $G_k$ is bounded by: $|G_k|\le p^{\upperb}\sqrt{T}\|x^{\upperb}_k\|+\max_x|\Cost(x)|$.
\item \emph{Conclusion.} Both factors of $F_k=G_k y^\star_k$ are Lipschitz and bounded, so, using $\mathrm{Lip}(fg) \le \mathrm{Lip}(f)\sup|g| + \sup|f|\,\mathrm{Lip}(g)$ with $f = G_k$ and $g = y^\star_k$, and $0 \le y^\star_k \le 1$, the function $F_k$ is Lipschitz with constant $L^{F}_k$:
\begin{equation*}
 L^F_k = L^{G}_k + \beta_k\|x^{\upperb}_k\|\Bigl(p^{\upperb}\sqrt{T}\|x^{\upperb}_k\|+\max_{x\in[x^{\lowerb}_k,x^{\upperb}_k]}|\Cost(x)|\Bigr).
\end{equation*}
\end{itemize}
Finally $F$ is Lipschitz with constant $L_F = \textstyle\sum_{k=1}^K L^F_k \, w_k$ as a weighted sum of Lipschitz functions. \qedhere
\end{proof}

Theorem~\ref{thm:cp:closed-form} characterizes each client's optimal response to a price signal in closed-form. The bilevel problem can alternatively be reformulated by replacing the lower level with an optimality system, so that it can be solved with standard mixed-integer techniques. Corollary~\ref{cor:optpes} ensures that the selection $(x_k^\star(p), y_k^\star(p))$ is valid for the upper-level problem, so that the bilevel problem can be equivalently reformulated as a single-level problem, which is continuous but nonconvex. The next theorem provides an optimality system for this selection.

\begin{theorem} \label{thm:cp:kkt-conditions}
Let Assumption~\ref{assu:cons} hold, and let $k \in \llbracket K \rrbracket$ and a price signal $p \in \mathcal{P}$ be given. Consider the system:
\begin{align}
\alpha_k p_t + \left(x_{t,k} - x^0_{t,k}\right) + \lambda_{t,k}^x - \mu_{t,k}^x - \gamma_k^x &= 0, \quad\forall t \in \llbracket T \rrbracket, \label{eq:cp:kkt-stationarity} \\
\Bill(p, x_k) + \Discomfort_k(x_k, x^0_k) - \Vill^0_k + \tfrac{1}{\beta_k} y_k + \tfrac{1}{\alpha_k} \lambda_k^y - \tfrac{1}{\alpha_k} \mu_k^y &= 0, \quad \label{eq:cp:kkt-stationarity-y} \\
0 \le \lambda_{t,k}^x\perp (x^{\upperb}_{t,k} - x_{t,k}) &\ge 0, \quad\forall t \in \llbracket T \rrbracket, \label{eq:cp:kkt-complementarity-x-lambda}\\
0 \le \mu_{t,k}^x\perp (x_{t,k} - x^{\lowerb}_{t,k}) &\ge 0, \quad\forall t \in \llbracket T \rrbracket, \label{eq:cp:kkt-complementarity-x-mu} \\
0 \le \lambda_k^y\perp (1 - y_k) &\ge 0, \quad\label{eq:cp:kkt-complementarity-y-lambda} \\
0 \le \mu_k^y\perp y_k &\ge 0, \quad\label{eq:cp:kkt-complementarity-y-mu} \\
\sum_{t=1}^T x_{t,k} &= \textstyle\sum_{t=1}^T x^0_{t,k}. \quad\label{eq:cp:kkt-marginality}
\end{align}
Then the following hold.
\begin{enumerate}[leftmargin=0pt, itemindent=*, label=(\alph*)]
\item \emph{Sufficiency.} If $(x_k, y_k)$ together with some multipliers $(\lambda^x_k, \mu^x_k, \gamma^x_k, \lambda^y_k, \mu^y_k)$ satisfies~\eqref{eq:cp:kkt-stationarity}--\eqref{eq:cp:kkt-marginality} and $x^{\lowerb}_k \le x_k \le x^{\upperb}_k$, $0 \le y_k \le 1$, then $(x_k, y_k)$ is optimal for~\eqref{model:cp}, whatever the value of $y_k$.
\item \emph{Necessity in the accepting regime.} If $(x_k, y_k)$ is optimal for~\eqref{model:cp} and $y_k > 0$, then it satisfies the system for some such multipliers.
\item \emph{The refusing regime.} If $y_k = 0$, the system is satisfied by $(x^\star_k(p), 0)$ but not, in general, by every optimal point of~\eqref{model:cp}: the argmin set is then $\mathcal{X}_k \times \{0\}$, whereas the system forces $x_k = x^\star_k(p)$. Necessity therefore fails in this regime, by design; as explained above, the imposed selection is optimal and costs nothing at the upper level.
\end{enumerate}
Moreover, the multipliers are bounded by:
\begin{align*}
\gamma_k^x &\in [M_k^\mu, M_k^\lambda], \\
0 \le \lambda_{t,k}^x, \mu_{t,k}^x &\le M_k^\lambda - M_k^\mu, \\
0 \le \lambda_k^y &\le \Vill^0_k, \\
0 \le \mu_k^y &\le M_k,
\end{align*}
where $M_k := \alpha_k p^{\upperb}\sum_{t=1}^T x^{\upperb}_{t,k} + \alpha_k \Discomfort_k(x^{\upperb}_k, x^{\lowerb}_k)$ and:
\begin{equation*}
M_k^\lambda = \alpha_k p^{\upperb} + \max_{t \in \llbracket T \rrbracket}\!\left(x^{\upperb}_{t,k} - x^0_{t,k} \right), \qquad
M_k^\mu = \alpha_k p^{\lowerb} + \min_{t \in \llbracket T \rrbracket}\!\left(x^{\lowerb}_{t,k} - x^0_{t,k} \right).
\end{equation*}
Here, $\lambda_{t,k}^x$, $\mu_{t,k}^x$, and $\gamma_k^x$ are the multipliers associated with the upper bounds, lower bounds, and equality constraint of the consumption choice problem, respectively, while $\lambda_k^y$ and $\mu_k^y$ are the multipliers associated with the upper and lower bounds of the participation choice problem.
\end{theorem}

\begin{proof}
As in Theorem~\ref{thm:cp:closed-form}, the client's problem decomposes into the consumption choice problem~\eqref{model:cco} and the participation choice problem~\eqref{model:cch}, treated sequentially. We first show that the stated system is exactly the optimality system of that decomposition, which gives sufficiency in both regimes and necessity when $y_k > 0$; the failure of necessity when $y_k = 0$ is the content of item~(c) of the statement and is immediate, since every $x_k \in \mathcal{X}_k$ is then optimal while the system selects $x^\star_k(p)$. We then derive the multiplier bounds.
\begin{enumerate}[leftmargin=0pt, itemindent=*]
\item \emph{Optimality conditions.} Problem~\eqref{model:cco} minimizes a strongly convex quadratic function over the polyhedron $\mathcal{X}_k$, which is nonempty under Assumption~\ref{assu:cons}. It therefore admits a unique minimizer, and since all constraints are affine the KKT conditions are necessary and sufficient. Using $\lambda_{t,k}^x$, $\mu_{t,k}^x$, $\gamma_k^x$ for the multipliers of the upper bounds, lower bounds, and the equality constraint, respectively, stationarity reads:
\begin{equation*}
p_t + \tfrac{1}{\alpha_k}\left(x_{t,k} - x^0_{t,k}\right) + \tfrac{1}{\alpha_k} \lambda_{t,k}^x - \tfrac{1}{\alpha_k} \mu_{t,k}^x - \tfrac{1}{\alpha_k} \gamma_k^x = 0, \qquad \forall t \in \llbracket T \rrbracket,
\end{equation*}
together with the complementarity conditions stated in the theorem. Likewise, \eqref{model:cch} minimizes the strongly convex function $y_k \mapsto \tfrac{1}{2\beta_k} y_k^2 + (\Vill_k(p) - \Vill_k^0) y_k$ over $[0,1]$, where $\Vill_k(p) := \Bill(p, x_k^\star(p)) + \Discomfort_k(x_k^\star(p), x_k^0)$, so its KKT conditions:
\begin{equation*}
\Vill_k(p) - \Vill_k^0 + \tfrac{1}{\beta_k} y_k + \tfrac{1}{\alpha_k} \lambda_k^y - \tfrac{1}{\alpha_k} \mu_k^y = 0, \qquad 0 \le \lambda_k^y \perp (1 - y_k) \ge 0, \qquad 0 \le \mu_k^y \perp y_k \ge 0,
\end{equation*}
are necessary and sufficient as well. This establishes the KKT system.
\item \emph{Bounds on $\gamma_k^x$.} Under Assumption~\ref{assu:cons}, the equality $\sum_t x_{t,k} = \textstyle\sum_t x^0_{t,k}$ together with $\sum_t x^{\lowerb}_{t,k} < \textstyle\sum_t x^0_{t,k} < \textstyle\sum_t x^{\upperb}_{t,k}$ prevents all components from being at their upper (resp. lower) bound. Hence there exist $\overline{t}, \underline{t} \in \llbracket T \rrbracket$ with $x_{\overline{t},k} < x^{\upperb}_{\overline{t},k}$ and $x_{\underline{t},k} > x^{\lowerb}_{\underline{t},k}$. Complementarity constraints \eqref{eq:cp:kkt-complementarity-x-lambda} and \eqref{eq:cp:kkt-complementarity-x-mu} give $\lambda_{\overline{t},k}^x = 0$ and $\mu_{\underline{t},k}^x = 0$, so stationarity constraint \eqref{eq:cp:kkt-stationarity} at these indices yields:
\begin{align*}
\gamma_k^x &= \alpha_k p_{\overline{t}} + (x_{\overline{t},k} - x^0_{\overline{t},k}) - \mu_{\overline{t},k}^x \le \alpha_k p_{\overline{t}} + (x^{\upperb}_{\overline{t},k} - x^0_{\overline{t},k})\le M_k^\lambda, \\
\gamma_k^x &= \alpha_k p_{\underline{t}} + (x_{\underline{t},k} - x^0_{\underline{t},k}) + \lambda_{\underline{t},k}^x\ge \alpha_k p_{\underline{t}} + (x^{\lowerb}_{\underline{t},k} - x^0_{\underline{t},k})\ge M_k^\mu.
\end{align*}
Thus $\gamma_k^x \in [M_k^\mu, M_k^\lambda]$.

\item \emph{Bounds on $\lambda_{t,k}^x$ and $\mu_{t,k}^x$.} Let $u_t := \alpha_k p_t + (x_{t,k} - x^0_{t,k})$, we have that every $u_t$ lies in $[M_k^\mu, M_k^\lambda]$. Moreover, since $x^{\lowerb}_{t,k} < x^{\upperb}_{t,k}$, the bounds cannot be active simultaneously, so $\lambda_{t,k}^x$ and $\mu_{t,k}^x$ are not both positive:
\begin{itemize}[leftmargin=0pt, itemindent=*]
\item if $\mu_{t,k}^x = 0$, then stationarity \eqref{eq:cp:kkt-stationarity} reads as $\lambda_{t,k}^x = \gamma_k^x - u_t \le M_k^\lambda - M_k^\mu$; 
\item if $\lambda_{t,k}^x = 0$, then stationarity \eqref{eq:cp:kkt-stationarity} reads as $\mu_{t,k}^x = u_t - \gamma_k^x \le M_k^\lambda - M_k^\mu$.
\end{itemize}
In either case $0 \le \lambda_{t,k}^x, \mu_{t,k}^x \le M_k^\lambda - M_k^\mu$, and $M_k^\lambda - M_k^\mu \ge 0$ since $M_k^\lambda \ge M_k^\mu$.

\item \emph{Bounds on $\lambda_k^y$ and $\mu_k^y$.} If $\lambda_k^y > 0$, complementarity \eqref{eq:cp:kkt-complementarity-y-lambda} forces $y_k = 1$ and $\mu_k^y = 0$, and the participation stationarity gives $\tfrac{1}{\alpha_k} \lambda_k^y = \Vill_k^0 - \Vill_k(p) - \tfrac{1}{\beta_k} \le \Vill_k^0$, using $\Vill_k(p) \ge 0$ and $\beta_k > 0$; the bound $\lambda_k^y \le \alpha_k \Vill_k^0$ holds trivially when $\lambda_k^y = 0$. If instead $\mu_k^y > 0$, then by complementarity constraint \eqref{eq:cp:kkt-complementarity-y-mu} $y_k = 0$ and $\lambda_k^y = 0$, whence $\tfrac{1}{\alpha_k} \mu_k^y = \Vill_k(p) - \Vill_k^0 \le \Vill_k(p)$, using $\Vill_k^0 \ge 0$. It remains to bound $\Vill_k(p) = \Bill(p, x_k^\star(p)) + \Discomfort_k(x_k^\star(p), x_k^0)$. Since prices are nonnegative and bounded from above by $p^{\upperb}$, and consumption is bounded from above by $x^{\upperb}_{t,k}$, it follows that $\Bill(p, x_k^\star(p)) \le p^{\upperb} \textstyle\sum_{t=1}^T x^{\upperb}_{t,k}$. Moreover $x_{t,k}^\star, x^0_{t,k} \in [x^{\lowerb}_{t,k}, x^{\upperb}_{t,k}]$, so componentwise $(x_{t,k}^\star - x^0_{t,k})^2 \le (x^{\upperb}_{t,k} - x^{\lowerb}_{t,k})^2$ and therefore $\Discomfort_k(x_k^\star(p), x_k^0) \le \Discomfort_k(x^{\upperb}_k, x^{\lowerb}_k)$. Combining this leads to $\mu_k^y \le \alpha_k \Vill_k(p) \le M_k := \alpha_k p^{\upperb} \textstyle\sum_{t=1}^T x^{\upperb}_{t,k} + \alpha_k \Discomfort_k(x^{\upperb}_k, x^{\lowerb}_k)$, which also covers the case $\mu_k^y = 0$.\qedhere
\end{enumerate}
\end{proof}

\begin{remark}
The system we use is not the Karush--Kuhn--Tucker (KKT) system of the joint program~\eqref{model:cp}. Writing the KKT conditions of~\eqref{model:cp} directly would carry a factor $y_k$ on the terms coming from $[\Bill + \Discomfort_k] y_k$, and would therefore place no restriction at all on $x_k$ when $y_k = 0$. We instead state the optimality system of the sequential decomposition established in Theorem~\ref{thm:cp:closed-form}: the conditions of the consumption choice problem~\eqref{model:cco}, followed by those of the participation choice problem~\eqref{model:cch}. This system pins down $x_k = x^\star_k(p)$ even for clients who refuse the offer, and Corollary~\ref{cor:optpes} guarantees that doing so loses nothing: those clients contribute zero to the leader's objective, so selecting $x^\star_k(p)$ among their optima is without loss of generality. 
\end{remark}

\begin{corollary}[MINLP single-level reformulation] \label{cor:cp:kkt-conditions}
Let Assumption~\ref{assu:cons} hold. Then the bilevel problem~\eqref{model:retailer} is equivalent to the mixed-integer program obtained by replacing the lower level with the optimality system of Theorem~\ref{thm:cp:kkt-conditions} and encoding the complementarity conditions through big-$M$ disjunctions, leading to the following single-level reformulation (MINLP-SLF):
\begin{align*} \tag{MINLP-SLF} \label{model:cp:kkt-conditions}
\max_{\substack{q, h, p, x, y, \lambda^x,\mu^x,\gamma^x,\lambda^y,\mu^y, \\
z^{x, +}, z^{x, -}, z^{y, +}, z^{y, -}}} \quad
\sum_{k=1}^K \left[ \Bill(p, x_k) - \Cost(x_k) \right] y_k w_k & \\
\text{s.t.}\qquad \qquad \qquad \qquad q \in \mathcal{Q}, \quad h &\in \mathcal{H}, \\
p_t - q_n - (p^{\upperb} - p^{\lowerb}) (1 - h_{t,n}) & \le 0, \quad \forall t \in \llbracket T \rrbracket, \forall n \in \llbracket N \rrbracket, \\
p_t - q_n - (p^{\lowerb} - p^{\upperb}) (1 - h_{t,n}) & \ge 0, \quad \forall t \in \llbracket T \rrbracket, \forall n \in \llbracket N \rrbracket, \\
\alpha_k p_t + \left(x_{t,k} - x^0_{t,k}\right)
+ \lambda_{t,k}^x - \mu_{t,k}^x - \gamma_k^x & = 0, \quad \forall t \in \llbracket T \rrbracket,\forall k \in \llbracket K \rrbracket, \\
\Bill(p, x_k) + \Discomfort_k(x_k, x^0_k) - \Vill^0_k + \tfrac{1}{\beta_k} y_k + \tfrac{1}{\alpha_k} \lambda_k^y - \tfrac{1}{\alpha_k} \mu_k^y & = 0, \quad \forall k \in \llbracket K \rrbracket, \\
\textstyle\sum_{t=1}^T x_{t,k} & = \textstyle\sum_{t=1}^T x^0_{t,k}, \quad \forall k \in \llbracket K \rrbracket, \\
x^{\lowerb}_{t,k} \le x_{t,k} & \le x^{\upperb}_{t,k}, \quad \forall t \in \llbracket T \rrbracket,\forall k \in \llbracket K \rrbracket, \\
0 \le y_k & \le 1, \quad \forall k \in \llbracket K \rrbracket, \\
\lambda_{t,k}^x - \big(M_k^\lambda - M_k^\mu\big) z^{x, +}_{t,k} & \le 0, \quad \forall t \in \llbracket T \rrbracket,\forall k \in \llbracket K \rrbracket, \\
x^{\upperb}_{t,k} - x_{t,k} - \big(x^{\upperb}_{t,k} - x^{\lowerb}_{t,k}\big)\big(1 - z^{x, +}_{t,k}\big) & \le 0, \quad \forall t \in \llbracket T \rrbracket,\forall k \in \llbracket K \rrbracket, \\
\mu_{t,k}^x - \big(M_k^\lambda - M_k^\mu\big) z^{x, -}_{t,k} & \le 0, \quad \forall t \in \llbracket T \rrbracket,\forall k \in \llbracket K \rrbracket, \\
x_{t,k} - x^{\lowerb}_{t,k} - \big(x^{\upperb}_{t,k} - x^{\lowerb}_{t,k}\big)\big(1 - z^{x, -}_{t,k}\big) & \le 0, \quad \forall t \in \llbracket T \rrbracket,\forall k \in \llbracket K \rrbracket, \\
\lambda_k^y \le \Vill^0_k z^{y, +}_k, \qquad y_k & \ge z^{y, +}_k, \quad \forall k \in \llbracket K \rrbracket, \\
\mu_k^y \le M_k z^{y, -}_k, \qquad y_k & \le 1 - z^{y, -}_k, \quad \forall k \in \llbracket K \rrbracket, \\
\lambda_{t,k}^x \ge 0, \quad \mu_{t,k}^x &\ge 0, \quad \forall t \in \llbracket T \rrbracket,\forall k \in \llbracket K \rrbracket, \\
\lambda_k^y \ge 0, \quad \mu_k^y \ge 0, \quad M_k^\mu \le \gamma_k^x & \le M_k^\lambda, \quad \forall k \in \llbracket K \rrbracket, \\
z^{x, +}_{t,k} + z^{x, -}_{t,k} & \le 1, \quad \forall t \in \llbracket T \rrbracket,\forall k \in \llbracket K \rrbracket, \\
z^{y, +}_k + z^{y, -}_k & \le 1, \quad \forall k \in \llbracket K \rrbracket, \\
z^{x, +}, z^{x, -} \in \{0,1\}^{T \times K}, \quad z^{y, +}, z^{y, -} & \in \{0,1\}^{K},
\end{align*}
where ``$h \in \mathcal{H}$'' is understood as the linear description~\eqref{eq:switch-counting}, and the constants $M_k$, $M_k^\lambda$ and $M_k^\mu$ are those of Theorem~\ref{thm:cp:kkt-conditions}. Equivalence is understood in the following sense: the two problems have the same optimal value, and from any solution of~\eqref{model:cp:kkt-conditions} one recovers a solution $(q,h)$ of~\eqref{model:retailer} together with an associated optimal client response $(x_k, y_k)_{k \in \llbracket K \rrbracket}$.
\end{corollary}

\begin{proof}
The constraints $p_t - q_n \le (p^{\upperb}-p^{\lowerb}) (1 - h_{t,n})$, and $p_t - q_n \ge (p^{\lowerb}-p^{\upperb}) (1 - h_{t,n})$, for all $t \in \llbracket T \rrbracket$, and for all $n \in \llbracket N \rrbracket$ ensures that $p=p(q,h)$ since $h$ is a vector of binary variables and $\sum_{n=1}^N h_{t,n} = 1$, for all $t \in \llbracket T \rrbracket$.
\begin{enumerate}[leftmargin=0pt, itemindent=*]
\item \emph{The big-$M$ constraints encode the complementarity conditions.} Consider the pair $\big(\lambda_{t,k}^x, x^{\upperb}_{t,k} - x_{t,k}\big)$; the three others are treated identically. Both quantities are nonnegative, and the binary $z^{x, +}_{t,k}$ forces one of them to vanish: if $z^{x, +}_{t,k} = 0$ then $\lambda_{t,k}^x = 0$, and if $z^{x, +}_{t,k} = 1$ then $x_{t,k} = x^{\upperb}_{t,k}$. This is exactly $0 \le \lambda_{t,k}^x \perp (x^{\upperb}_{t,k} - x_{t,k}) \ge 0$. Conversely, any point satisfying the complementarity condition admits an admissible value of $z^{x, +}_{t,k}$ provided the coefficients multiplying the binary are valid upper bounds on the two quantities. This is the case here: $\lambda_{t,k}^x \le M_k^\lambda - M_k^\mu$ and $\mu_{t,k}^x \le M_k^\lambda - M_k^\mu$ by Theorem~\ref{thm:cp:kkt-conditions}, $\lambda_k^y \le \Vill^0_k$ and $\mu_k^y \le M_k$ by the same theorem, while $x^{\upperb}_{t,k} - x_{t,k}$ and $x_{t,k} - x^{\lowerb}_{t,k}$ are bounded by $x^{\upperb}_{t,k} - x^{\lowerb}_{t,k}$ and $1 - y_k, y_k$ by $1$, since $x_k \in \mathcal{X}_k$ and $y_k \in [0,1]$. The disjunctive constraints are therefore exact, and no optimal solution of the lower level is cut off. The inequalities $z^{x, +}_{t,k} + z^{x, -}_{t,k} \le 1$ and $z^{y, +}_k + z^{y, -}_k \le 1$ are valid rather than restrictive: under Assumption~\ref{assu:cons} one has $x^{\lowerb}_{t,k} < x^{\upperb}_{t,k}$, so the two consumption bounds cannot be active at the same $t$, and $y_k$ cannot equal both $0$ and $1$.
\item \emph{Every feasible point of~\eqref{model:cp:kkt-conditions} is feasible for~\eqref{model:retailer}.} By the previous paragraph, for each $k \in \llbracket K \rrbracket$ the pair $(x_k, y_k)$ satisfies the system of Theorem~\ref{thm:cp:kkt-conditions}, which is sufficient for optimality in~\eqref{model:cp} both when $y_k > 0$ and when $y_k = 0$. Hence $(x_k,y_k)$ solves~\eqref{model:cp}, the point is feasible for~\eqref{model:retailer}, and the two objectives agree. The optimal value of~\eqref{model:cp:kkt-conditions} is therefore at most that of~\eqref{model:retailer}.
\item \emph{Conversely.} Let $(q,h)$ be feasible for~\eqref{model:retailer} and let $\big(x_k^\star(p), y_k^\star(p)\big)$ be the client responses of Theorem~\ref{thm:cp:closed-form}. Under Assumption~\ref{assu:cons} the consumption problem~\eqref{model:cco} is a strongly convex quadratic program with affine constraints, so $x_k^\star(p)$ admits multipliers $(\lambda_k^x, \mu_k^x, \gamma_k^x)$ satisfying stationarity, complementarity and marginality. For the participation variable, write $\Vill_k = \Bill(p, x_k^\star(p)) + \Discomfort_k(x_k^\star(p), x_k^0)$ and recall that $y_k^\star(p) = \Proj_{[0,1]}\big(\beta_k(\Vill^0_k - \Vill_k)\big)$. Setting:
\begin{equation*}
\lambda_k^y = \alpha_k \max\Big(0,\Vill^0_k - \Vill_k - \tfrac{1}{\beta_k}\Big), \qquad \mu_k^y = \alpha_k \max\big(0,\Vill_k - \Vill^0_k\big),
\end{equation*}
at most one of the two is positive and the participation stationarity and complementarity conditions hold: $\lambda_k^y > 0$ forces $\beta_k(\Vill^0_k - \Vill_k) > 1$, hence $y_k^\star(p) = 1$, while $\mu_k^y > 0$ forces $\Vill^0_k < \Vill_k$, hence $y_k^\star(p) = 0$. Choosing the binaries as the indicators of the active constraints, that is $z^{x, +}_{t,k} = 1$ if and only if $x^\star_{t,k}(p) = x^{\upperb}_{t,k}$ and likewise for the three other families, all big-$M$ constraints are satisfied, since the multipliers obey the bounds of Theorem~\ref{thm:cp:kkt-conditions}. The point is thus feasible for~\eqref{model:cp:kkt-conditions} with the same objective value, so the optimal value of~\eqref{model:retailer} is at most that of~\eqref{model:cp:kkt-conditions}. The two values coincide.
\end{enumerate}
Finally, by Corollary~\ref{cor:optpes} the optimistic and pessimistic readings of~\eqref{model:retailer} coincide, so no ambiguity arises from the clients for which the lower-level solution is not unique, namely those with $y_k = 0$. \qedhere
\end{proof}

\section{Solution methods}
\label{sec:algorithmic}
Section~\ref{sec:mod} introduced the time-of-use (ToU) pricing design problem as
a bilevel optimization model: the retailer chooses a price scheme, while each
client reacts by adjusting its consumption and deciding whether to accept the
offer or revert to an outside option. The resulting problem is mixed-integer,
nonconvex, and nonsmooth. It can be tackled either by using the reformulation \eqref{model:cp:closed-form} based on Corollary~\ref{cor:cp:closed-form}, or by using the reformulation \eqref{model:cp:kkt-conditions} based on Corollary~\ref{cor:cp:kkt-conditions}. 

\smallskip 

We first study \emph{pointwise pricing} (Section~\ref{subsec:pointwise-pricing}), in which the retailer may set an independent price at each time step. This is the least-constrained version of the problem: it ignores the schedule constraints $\mathcal{H}$ and is therefore easier to solve. Its optimal value serves two purposes. First, it upper-bounds the optimal ToU profit, which allows us to estimate the potential profit loss induced by restricting the tariffs to be ToU, and to estimate the price of representability (Definition~\ref{def:por}). Second, after projection onto $\mathcal{H}$, it may provide a starting point for certain ToU methods. We then study \emph{ToU pricing} (Section~\ref{subsec:tou-pricing}), the regulator-admissible restriction obtained by grouping time steps into periods. Both models share the same lower level, so the structural results of Section~\ref{sec:analytical} apply throughout.

\begin{definition}[Price of representability] \label{def:por}
Let $F$ denote the leader's hyper-objective (the expected profit from proposing a price scheme), let $F^\star_{\mathrm{pw}} := \max_{p \in \mathcal{P}} F(p)$ be the optimal pointwise profit, and let $F^\star_{\mathrm{ToU}} := \max_{q \in \mathcal{Q}, h \in \mathcal{H}} F(p(q,h))$ be the optimal profit attainable with an admissible schedule. Assuming that the pointwise profit is positive ($F^\star_{\mathrm{pw}} > 0$), the price of representability of the schedule set $\mathcal{H}$ is defined as:
\begin{equation*}
\PoR(\mathcal{Q}, \mathcal{H}) := \frac{F^\star_{\mathrm{pw}} - F^\star_{\mathrm{ToU}}}{F^\star_{\mathrm{pw}}}.
\end{equation*}
\end{definition}
Since $\mathcal{P}(\mathcal{Q}, \mathcal{H}) := \{p(q,h) : q \in \mathcal{Q}, h \in \mathcal{H}\} \subseteq \mathcal{P}$, we have $F^\star_{\mathrm{ToU}} \le F^\star_{\mathrm{pw}}$ and hence $\PoR(\mathcal{H}) \in [0,1]$. Note that the definition is precise only when both the numerator and denominator are computed exactly, which is not always the case in practice, because the problem is numerically challenging. 

\subsection{The pointwise pricing problem}
\label{subsec:pointwise-pricing}
In pointwise pricing the retailer sets a price at each time step; equivalently, as noted in Section~\ref{subsec:feasible-pricing}, $N=T$ and the constraints on $(q,h)$ can be discarded, and it suffices to optimize the effective price on the box $\mathcal{P} = [p^{\lowerb},p^{\upperb}]^{T}$. We shall refer to this problem as pw-$\BOP$, as it is a relaxation of the original time-of-use model, \eqref{model:retailer}. Depending on how the lower level is reformulated, we consider two methods:
\begin{enumerate}[leftmargin=0pt, itemindent=*]
\item \emph{a direct method:} using Corollary~\ref{cor:cp:closed-form}, we apply first-order ascent algorithm to the obtained single-level nonconvex nonsmooth optimization problem.
\item \emph{an indirect method:} using Corollary~\ref{cor:cp:kkt-conditions}, we hand the mixed-integer nonlinear program \eqref{model:cp:kkt-conditions} to a standard solver.
\end{enumerate}

Each reformulation opens up a different family of solution methods: the nonsmooth program of the direct route can be tackled with nonsmooth first-order methods, such as stochastic gradient ascent, while the mixed-integer program of the indirect route can be handled by dedicated nonlinear solvers. We focus here on the direct method and establish that almost surely any accumulation point of the SGA recursion is a genuine Clarke critical point of the hyper-objective. We place the objective in the framework of \cite{bolte2020mathematical}: it admits an elementary selection (Definition~\ref{def:elementary-selection}), and consequently, under mild assumptions, almost every accumulation point of the SGA recursion is a genuine Clarke critical point of the hyper-objective rather than an artifact of differentiating through a particular program~\cite[Theorem 4]{bolte2020mathematical}. We recall the relevant definitions from \cite{bolte2020mathematical} in Appendix~\ref{app:elementary-selections}.

\begin{assumption}
\label{assu:cons:deg}
For all $k \in \llbracket K \rrbracket$, suppose that the baseline consumption $x_k^0$ and the bounds satisfy:
\begin{equation*}
\sum_{t \in \mathcal{T}} x^{\lowerb}_{t,k} + \textstyle\sum_{t \in \llbracket T \rrbracket \setminus \mathcal{T}} x^{\upperb}_{t,k} \neq \textstyle\sum_{t=1}^T x^0_{t,k}, \qquad \forall \mathcal{T} \subseteq \llbracket T \rrbracket.
\end{equation*}
\end{assumption}
The assumption is a non-degeneracy condition: it ensures that any element $x_k \in \mathcal{X}_k$ has a component strictly inside the bounds, which is what guarantees that the free set $I_k(p)$ used in Theorem~\ref{thm:pointwise-selection} is never empty and that the multiplier $\gamma^x_k$ is well defined. This condition will allow us to simplify the presentation of the algorithms and their proofs. From the modelling perspective, it is inessential, in particular, it can be ensured by adding an arbitrarily small perturbation. Assumption \ref{assu:cons:deg} is assumed to hold from here on. The next theorem shows that $F$ admits an elementary selection whose selection derivative is cheap to evaluate.

\begin{theorem} \label{thm:pointwise-selection}
Let item (3) of Assumption~\ref{assu:cost} hold. Then the hyper-objective \eqref{model:ho}, $F$, is an elementary selection, and its selection derivative can be evaluated in $O(T\log{T})$ elementary operations for each $F_k$.
\end{theorem}

\begin{proof}
\emph{Elementary selection.} By Theorem~\ref{thm:cp:closed-form}, the maps $p \mapsto (x_k^\star(p), y_k^\star(p))$ are continuous and piecewise polynomial over semi-algebraic regions. Such maps are elementary selections in the sense of Definition~\ref{def:elementary-selection}, understood componentwise for the vector-valued map $x^\star_k$: polynomials are obtained from affine maps by finitely many products and sums, hence lie in the class of elementary selections $\mathcal{S}$ (Definition \ref{def:elementary-selection}), and a semi-algebraic region is by definition the solution set of finitely many polynomial equalities and inequalities, hence of finitely many relations between elements of $\mathcal{S}$, so the associated index map lies in $\mathcal{I}$ (Definition \ref{def:elementary-index}). The box constraint is incorporated by precomposing with the clipping $\Proj_{[p^{\lowerb},p^{\upperb}]}$, itself an elementary selection; by stability of $\mathcal S$ under composition, sum, and product~\cite[Proposition 1]{bolte2020mathematical}, and since $\Cost \in \mathcal{S}$ by item (3) of Assumption~\ref{assu:cost}, each $F_k$, and hence $F = \textstyle\sum_k F_k w_k$, is an elementary selection. We now exhibit such a selection explicitly, which also yields the complexity bound.

\smallskip

Given the price signal $p$, client $k$'s optimal consumption takes the form
\begin{equation*}
x_{t,k}^\star(p) = \min\left\{x^{\upperb}_{t,k}, \max\left\{x^{\lowerb}_{t,k}, x^0_{t,k} - \alpha_k (p_t - \gamma_k^x(p))\right\}\right\},
\end{equation*}
where $\gamma_k^x(p)$ is the multiplier of the equality constraint. It is computed from the ordered partition of $\llbracket T \rrbracket$ into the lower-saturated, free, and upper-saturated coordinates of $x_k^\star(p)$, denoted by the active sets $(L_k(p), I_k(p), U_k(p))$, as:
\begin{equation*}
\gamma_k^x(p) = \tfrac{1}{\alpha_k |I_k(p)|} \left[ \textstyle\sum_{t=1}^T x^0_{t,k} - \textstyle\sum_{t \in I_k(p)} \left(x^0_{t,k} - \alpha_k p_t\right) - \textstyle\sum_{t \in U_k(p)} x^{\upperb}_{t,k} - \textstyle\sum_{t \in L_k(p)} x^{\lowerb}_{t,k} \right],
\end{equation*}
Assumption \ref{assu:cons:deg} ensures that the free set $I_k(p)$ is never empty, so that the expression is well-defined. The optimal participation $y_k^\star(p)$ is then obtained by comparing $z_k = \beta_k(\Vill^0_k - \Vill_k(p))$ to $0$ and $1$, where $\Vill_k(p) := \Bill(p, x_k^\star(p)) + \Discomfort_k(x_k^\star(p), x_k^0)$: it equals $0$ if $z_k < 0$, equals $1$ if $z_k > 1$, and equals $\beta_k(\Vill^0_k - \Vill_k(p))$ otherwise.

\smallskip

Enumerating the finitely many combinations of active sets and participation states leads to the notion of a \emph{phase} (denoted $\theta$, see Appendix \ref{app:notation}): a pair $\theta = (\pi, q)$ with $\pi = (L^\theta, I^\theta, U^\theta)$ the partition (lower-saturated, free, upper-saturated) and $r \in \{0, \star, 1\}$ the participation state (clamped at $0$, interior, clamped at $1$). There are at most $m = 3^{T} \cdot 3$ phases, labelled $\theta = 1, \dots, m$. For each phase, define the consumption map $\xi^\theta_k$ by:
\begin{align*}
\xi^\theta_{t,k}(p) &=
\begin{cases}
x^{\lowerb}_{t,k}, & t\in L^\theta,\\
x^0_{t,k}-\alpha_k (p_t - \gamma_k^\theta(p)), & t\in I^\theta,\\
x^{\upperb}_{t,k}, & t\in U^\theta,
\end{cases}
\end{align*}
and the participation map $\eta^\theta_k$ by:
\begin{align*}
\eta^\theta_k(p) =
\begin{cases}
0, & r=0,\\
\beta_k\big(\Vill^0_k-[\Bill(p,\xi^\theta_k(p))+\Discomfort_k(\xi^\theta_k(p),x^0_k)]\big), & r=\star,\\
1, & r=1,
\end{cases}
\end{align*}
where the multiplier $\gamma_k^\theta(p)$ is given by:
\begin{equation*}
\gamma_k^\theta(p) = \tfrac{1}{\alpha_k |I^\theta|}\Big[\textstyle\sum_{t\in \llbracket T \rrbracket}x^0_{t,k} -\sum_{t\in I^\theta}\big(x^0_{t,k}-\alpha_k p_t\big) -\sum_{t\in U^\theta}x^{\upperb}_{t,k} -\sum_{t\in L^\theta}x^{\lowerb}_{t,k}\Big].
\end{equation*}
The function $F_k$ can then be expressed as an elementary function $F_{k,\theta}$ over each phase $\theta$, defined by:
\begin{equation*}
F_{k,\theta}(p)=\big[\Bill(p,\xi^\theta_k(p))-\Cost(\xi^\theta_k(p))\big] \eta^\theta_k(p).
\end{equation*}
Since $\xi^\theta_k$ and $\gamma^\theta_k$ are affine in $p$ and $\Bill, \Cost, \Discomfort_k \in \mathcal{E}$, each $F_{k,\theta} \in \mathcal{E}$. Define $\sigma_k(p) = \theta$ on the region $R^\theta$ cut out by the saturation conditions: writing $v^\theta_{t,k}(p) := x^0_{t,k} - \alpha_k (p_t - \gamma^\theta_k(p))$ and
$z^\theta_k(p) := \beta_k\big(\Vill^0_k - [\Bill(p,\xi^\theta_k(p)) + \Discomfort_k(\xi^\theta_k(p), x^0_k)]\big)$,
\begin{equation*}
R^\theta = \left\{ p \;\middle|\; \begin{aligned}
v^\theta_{t,k}(p) \le x^{\lowerb}_{t,k}, && t \in L^\theta, \\
x^{\lowerb}_{t,k} \le v^\theta_{t,k}(p) \le x^{\upperb}_{t,k}, && t \in I^\theta, \\
v^\theta_{t,k}(p) \ge x^{\upperb}_{t,k}, && t \in U^\theta
\end{aligned} \right\} \cap Y^\theta,
\end{equation*}
where $Y^\theta = \{z^\theta_k(p) \le 0\}$, $\{0 \le z^\theta_k(p) \le 1\}$, or $\{z^\theta_k(p) \ge 1\}$ when $r = 0, \star, 1$. Each $R^\theta$ is the solution set of finitely many inequalities between elementary functions, so $\sigma_k$ is an elementary selection ($\sigma_k \in \mathcal{I}$). When ties occur at phase boundaries and multiple phases are valid, we break them using a fixed lexicographic rule. Since $p \mapsto (x_k^\star(p), y_k^\star(p))$ is continuous, the competing pieces coincide in value at these points, making the tie-breaking choice inconsequential with respect to the value. On $R^\theta$ the active-set optimality conditions give $x_k^\star(p) = \xi^\theta_k(p)$ and $y_k^\star(p) = \eta^\theta_k(p)$, hence $F_k(p) = F_{k,\sigma_k(p)}(p)$, and the selection derivative is $\widehat\nabla_{\sigma_k} F_k(p) = \nabla F_{k,\sigma_k(p)}(p)$.

\smallskip

\emph{Complexity.} Evaluating the selection derivative splits into identifying the active phase $\theta = \sigma_k(p)$ and differentiating the selected piece $F_{k,\theta}$. Identifying $\theta$ is exactly the active-set computation of $(x_k^\star(p), y_k^\star(p))$: the sorting-based algorithm of \cite[Section~1.C]{jacquot2019game} returns $L_k(p), I_k(p), U_k(p)$ in $O(T\log{T})$, after which a single $O(T)$ pass evaluates $z^\theta_k(p)$ (a dot product for $\Bill$, a separable sum for $\Discomfort_k$) and fixes $r$; the algorithm never enumerates the $m$ phases. With $\theta$ fixed, $F_{k,\theta}$ is a single elementary expression whose evaluation costs $O(T)$: the maps $p \mapsto \gamma^\theta_k(p) \mapsto \xi^\theta_k(p)$ are affine with a Jacobian ($\partial \xi^\theta_{t,k}/\partial p_s = - \alpha_k \mathbf 1[s{=}t] + \alpha_k \mathbf 1[s,t\in I^\theta]/|I^\theta|$), $\Bill, \Discomfort_k$ each cost $O(T)$, and $\Cost$ has a known selection derivative. The input clipping $\Proj_{[p^{\lowerb},p^{\upperb}]}$ is also $O(T)$. Through the chain rule, we can evaluate $\widehat\nabla_{\sigma_k} F_k(p)$ in $O(T\log{T})$.\qedhere
\end{proof}

We handle the box constraints by composition rather than by projecting the iterate, we optimize the unconstrained function:
\begin{equation*}
\label{eq:aug-ho} \tag{$\AHO$}
\widetilde{F}(p) = F\big(\Proj_{[p^{\lowerb},p^{\upperb}]}(p)\big),
\end{equation*}
whose maximization is equivalent to that of $F$ over the box. The stochastic gradient ascent (SGA) recursion is then:
\begin{equation} \label{eq:iterate-recursion}
p_{i+1} = p_{i} + \eta_i \widehat\nabla_{\sigma} {\widetilde{F}_{B_i}}(p_i),
\end{equation}
where $\eta_i$ is the step size, $\widehat\nabla_{\sigma}$ the selection derivative, and $\widetilde{F}_{B_i} := \textstyle\sum_{k \in B_i} \widetilde{F}_k(p)$ the mini-batch estimate built from the batch $B_i$, with $\widetilde{F}_k(p) = F_k\big(\Proj_{[p^{\lowerb},p^{\upperb}]}(p)\big)$. The selection from Theorem~\ref{thm:pointwise-selection} is augmented with the clipping onto the box to yield a selection for $\widetilde{F}_k$. The selection derivative is then:
\begin{equation*}
\widehat\nabla_{\sigma} \widetilde{F}_k(p) = \widehat\nabla_{\sigma} \Proj_{[p^{\lowerb},p^{\upperb}]}(p) \cdot \widehat\nabla_{\sigma} F_k\big(\Proj_{[p^{\lowerb},p^{\upperb}]}(p)\big),
\end{equation*}
where the selection for the projection is:
\begin{equation*}
\big(\widehat\nabla_{\sigma} \Proj_{[p^{\lowerb},p^{\upperb}]}(p)\big)_t = \begin{cases}
0, & p_t < p^{\lowerb} \text{ or } p_t > p^{\upperb},\\
1, & \text{otherwise}.
\end{cases}
\end{equation*}
Algorithm~\ref{alg:sgd-pointwise} implements~\eqref{eq:iterate-recursion}.

\begin{algorithm}[!ht]
\caption{Stochastic gradient ascent for pointwise pricing}
\begin{algorithmic}[1]
\Require initial effective prices $p_0 \in \R^T$; price bounds $p^{\lowerb},p^{\upperb}$; client data; step sizes $(\eta_i)_{i\ge 0}$; batch size $b$; number of iterations $N_{\mathrm{it}}$
\Ensure an iterate sequence whose accumulation points are almost surely Clarke critical for $\widetilde{F}$ (Corollary~\ref{corollary:sgd-convergence}); the last iterate is returned
\For{$i = 0,\dots,N_{\mathrm{it}}-1$}
\State sample a mini-batch $B_i \subseteq \llbracket K \rrbracket$ uniformly without replacement, $|B_i| = b$
\State $g \gets 0$
\State $\hat{p} \gets \Proj_{[p^{\lowerb},p^{\upperb}]}(p)$
\For{$k \in B_i$}
\State compute the active sets $(L_k,I_k,U_k)$ of $x_k^\star(\hat{p})$ and the multiplier $\gamma_k^x(p)$ 
\State $x_k^\star \gets \Proj_{\mathcal{X}_k}(x_k^0 - \alpha_k \hat{p})$;\quad
$y_k^\star \gets \Proj_{[0,1]}\!\big(\beta_k(\Vill^0_k - \Vill_k(\hat{p}))\big)$
\State compute the phase $\sigma_k(p)$ and the selection derivative $\widehat\nabla_{\sigma_k} \widetilde{F}_k(p)$
\State $g \gets g + b^{-1} K w_k \widehat\nabla_{\sigma_k} \widetilde{F}_k(p)$
\EndFor
\State $p \gets p + \eta_i g$ 
\EndFor
\State \Return $p$
\end{algorithmic}
\label{alg:sgd-pointwise}
\end{algorithm}

\begin{proposition}
\label{prop:iterate-boundness}
The iterate sequence $(p_i)_{i \ge 0}$ generated by Algorithm~\ref{alg:sgd-pointwise} is bounded for every choice of $p_0$.
\end{proposition}

\begin{proof} Recall the update of Algorithm~\ref{alg:sgd-pointwise},
\begin{equation*}
p_{i+1}=p_i+\eta_i g, \qquad g=\sum_{k\in B_i} b^{-1} K w_k \widehat\nabla_{\sigma} \big(F_k\circ\Proj_{[p^{\lowerb},p^{\upperb}]}\big)(p_i),
\end{equation*}
where $B_i\subseteq \llbracket K \rrbracket$ is the batch drawn at iteration $i$ and $b$ is its size. 
\begin{itemize}[leftmargin=0pt, itemindent=*]
\item \emph{Step 1: outside the box the selection gradient vanishes in the corresponding coordinate.}
The projection acts coordinatewise, $\Proj_{[p^{\lowerb},p^{\upperb}]}(p)_t=\pi(p_t)$ with $\pi(u)=\min\{p^{\upperb},\max\{p^{\lowerb},u\}\}$. The selection breaks ties at the nonsmooth points and the chain rule applies to selection derivatives, for every $p\in\mathbb R^{T}$, $k\in \llbracket K \rrbracket$ and every coordinate $t$:
\begin{equation*}
\big(\widehat\nabla_{\sigma} (F_k\circ \Proj_{[p^{\lowerb},p^{\upperb}]})(p)\big)_t =\widehat\nabla_{\sigma} \pi(p_t) \big(\widehat\nabla_{\sigma} F_k(\Proj_{[p^{\lowerb},p^{\upperb}]}(p))\big)_t.
\end{equation*}
Fix a coordinate $t$ and set $O_t:=\{p\in\mathbb R^{T}\mid p_t\notin[p^{\lowerb},p^{\upperb}]\}$. If $p\in O_t$, then $p_t$ lies strictly outside $[p^{\lowerb},p^{\upperb}]$, and $\pi$ is constant on a neighbourhood of $p_t$, so $\widehat\nabla_{\sigma} \pi(p_t)=0$ irrespective of the values assigned to the selection at $p^{\lowerb},p^{\upperb}$. Hence:
\begin{equation*}
\big(\widehat\nabla_{\sigma} (F_k\circ \Proj_{[p^{\lowerb},p^{\upperb}]})(p)\big)_t=0 \qquad \forall p \in O_t, k \in \llbracket K \rrbracket,
\end{equation*}
and, $g_i$ being a linear combination of these vectors, $(g_i)_t=0$ whenever $p_i\in O_t$, whatever
the batch $B_i$.
\item \emph{Step 2: a coordinate leaving the box is frozen.} Suppose $p_{i_0}\in O_t$. By Step 1, $(g_{i_0})_t=0$, so $p_{i_0+1,t}=p_{i_0,t}$; membership in $O_t$ depends on the $t$-th coordinate only, whence $p_{i_0+1}\in O_t$. By induction, $p_i\in O_t$ and $p_{i,t}=p_{i_0,t}$ for all $i\ge i_0$.
\item \emph{Step 3: conclusion.} Fix $t$ and distinguish two cases. If $p_i \notin O_t$ for all $i \ge 0$, then $p_{i,t} \in [p^{\lowerb}, p^{\upperb}]$ for all $i$. Otherwise let $i_0$ such that $p_{i_0} \in O_t$. Then, $p_{i,t}=p_{i_0,t}$ for $i\ge i_0$ by Step 2. Either way, $\{p_{i,t}\}_{i \ge 0}$ is bounded. Since there are finitely many coordinates, the iterate $\{p_i\}_{i \ge 0}$ remains bounded. \qedhere
\end{itemize}
\end{proof}

\begin{corollary} \label{corollary:sgd-convergence}
Let for all $i$, $\eta_i = c r_i$, and $r_i = o(1 / \log{i})$. Then, for almost every $c \in (0,1]$, almost surely every accumulation point of the sequence generated by Algorithm~\ref{alg:sgd-pointwise} is Clarke critical for the augmented hyper-objective \eqref{eq:aug-ho}, $\widetilde{F}$.
\end{corollary}

\begin{proof}
The proof follows from Theorem~\ref{thm:pointwise-selection}, Proposition \ref{prop:iterate-boundness} and \cite[Theorem 4]{bolte2020mathematical}, which states that any accumulation point of the SGD recursion is almost surely Clarke critical point of the objective, provided that the objective is an elementary selection and the iterates remain bounded almost surely. Note that here, we considered the term $b^{-1} K w_k$ in the gradient computation because our initial objective is a weighted sum of the individual client objectives. \qedhere
\end{proof}

\subsection{Time-of-use pricing}
\label{subsec:tou-pricing}
We now return to the ToU problem, in which the retailer is constrained to choose a price scheme that is constant over each period. The integrality of $\mathcal{H}$ makes the problem combinatorial, so the direct gradient approach of Section~\ref{subsec:pointwise-pricing} no longer applies unchanged. We present four methods, trading computational cost against profit:
\begin{enumerate}[leftmargin=0pt, itemindent=*]
\item \emph{a complete enumeration method:} which enumerates all admissible schedules and solves the resulting problem for each schedule;
\item \emph{an approximation rounding method:} which projects a pointwise solution onto $\mathcal{H}$ and re-optimizes the period prices over the recovered schedule;
\item \emph{an indirect method:} which augments the mixed-integer reformulation with the integer schedule variables and solves the resulting single-level program;
\item \emph{a direct method:} which reparametrizes the schedule by continuous start instants and lengths and applies the gradient machinery of Section~\ref{subsec:pointwise-pricing} to the relaxed problem (Algorithm \ref{alg:sgd-tou}), recovering a feasible schedule at the end by using the second method (Approximation rounding method).
\end{enumerate}

\subsubsection{Complete enumeration method}
\label{subsubsec:enumeration}
The complete enumeration method is the most straightforward: it simply enumerates every admissible schedule $h \in \mathcal{H}$ and solves the inner continuous problem over $q \in \mathcal{Q}$ for each.
\begin{equation*} \label{eq:enumeration}
\max_{h \in \mathcal{H}} \Big[\max_{q \in \mathcal{Q}} F\big(p(q,h)\big)\Big].
\end{equation*}
The inner problem falls into the pointwise pricing problem studied in Section~\ref{subsec:pointwise-pricing}, and the outer enumeration can be parallelized. The method grows linearly in $|\mathcal{H}|$, which can be large depending on the regulatory constraints considered and on $(N, S, L)$. When available, we use it as a \emph{baseline}: it yields the value $F^\star_{\mathrm{ToU}}$ of Definition~\ref{def:por} up to the accuracy of the inner solves, and it therefore serves as a benchmark for the methods below. 

\smallskip

This does not make the remaining methods redundant, but it does determine what they are for. The cardinality of $\mathcal{H}$ grows quickly with the number of time periods and the complexity of the regulatory constraints. Refining the time grid below the hour multiplies the number of admissible boundary positions; adding periods multiplies the number of subperiod orderings; and moving from a single representative day to a seasonal or weekday/weekend calendar composes the schedule sets multiplicatively across days (cf. Appendix \ref{app:multi-day}). 

\begin{remark}
\label{rem:enumeration-cost}
Given the schedule, one needs around three minutes to solve the inner problem for a single representative day with $K=100$ clients on a standard laptop. The enumeration method is therefore inefficient for the experiments of Section~\ref{sec:num}, which consider just one typical day and basic regulatory constraints: at that rate, using the experimental setup of Section~\ref{sec:num} and the cardinalities in Table~\ref{tab:schedule:card}, complete enumeration per choice of parameters would take approximately $1284 \times 2:16 = 2910$ minutes (about 2 days) for the two-period schedule, and about $29016 \times 2:37 = 75925$ minutes (about 53 days) for the three-period schedule.
\end{remark}

\subsubsection{Approximation rounding method}
\label{subsubsec:rounding}
The approximation rounding method reuses the pointwise solution as a starting point. Let $p^\star$ be the price vector returned by the pointwise problem (Section~\ref{subsec:pointwise-pricing}). Since $p^\star$ is not necessarily representable by any admissible schedule, we first project it onto $\mathcal{H}$:
\begin{equation*}
(q^\star, h^\star) \in \argmin_{q \in \mathcal{Q}, h \in \mathcal{H}}\; \|p^\star - p(q,h)\|_2^2.
\end{equation*}
The bilinear terms in $p(q,h)$ are linearized exactly by the big-$M$ indicator constraints already used in Corollary~\ref{cor:cp:kkt-conditions}, namely $p_t - q_n \le (p^{\upperb}-p^{\lowerb})(1-h_{t,n})$ and $p_t - q_n \ge (p^{\lowerb}-p^{\upperb})(1-h_{t,n})$, which are exact because $h$ is binary with $\sum_n h_{t,n} = 1$ and $q$ is bounded. With the squared Euclidean distance the projection is therefore a mixed-integer \emph{quadratic} program, which is how it is solved in Section~\ref{sec:num}; replacing $\|\cdot\|_2^2$ by $\|\cdot\|_1$ or $\|\cdot\|_\infty$ would give a mixed-integer linear program instead, at the cost of no longer being a Euclidean projection. 

\begin{corollary}
Let $\epsilon = L_F \|p^\star - p(q^\star, h^\star)\|$ where $L_F$ is a Lipschitz constant of $F$ over the box $[p^{\lowerb}, p^{\upperb}]$, we have $F(p(q^\star, h^\star))$ is an $\epsilon$-approximation of $F(p^\star)$.
\end{corollary}

The projection minimizes the distance to $p^\star$, not profit, so its profit may be markedly below the pointwise benchmark. We therefore fix the recovered schedule $h^\star$ and re-optimize the period prices alone:
\begin{equation*}
\max_{q \in \mathcal{Q}}\; F\big(p(q, h^\star)\big).
\end{equation*}
At fixed $h^\star$ the map $q \mapsto p(q, h^\star)$ is linear, so $q \mapsto F(p(q,h^\star))$ is again an elementary selection~\cite[Stability of $\mathcal{S}$, Proposition 1]{bolte2020mathematical} and is solved by the methods of Section~\ref{subsec:pointwise-pricing}.

\subsubsection{Indirect method}
\label{subsubsec:indirect}
The indirect method solves the ToU problem as a single-level problem, using the formulation of Corollary~\ref{cor:cp:kkt-conditions} with the schedule variables $h$ left free and the switch constraint written in the linear form~\eqref{eq:switch-counting}. Unlike the rounding method, it optimizes the schedule and the price levels jointly and fixes nothing a priori; The price to pay is a much larger search space: the model gains $T \cdot N$ binary variables and the constraints regulating the schedule. 

\subsubsection{Direct method}
\label{subsubsec:direct}
The direct method optimizes the schedule by gradient descent, which requires a \emph{continuous} parametrization of the schedule: one in which period boundaries may fall at real instants rather than integer time steps. We first introduce a low-dimensional parametrization of $\mathcal{H}$, and then exploit it in the same way as in Section~\ref{subsec:pointwise-pricing} to compute a selection derivative of the hyper-objective with respect to the schedule parameters.

\begin{proposition}[Lower-dimensional schedule representation] \label{prop:charec-calen}
Let $\mathcal{M}$ be the set of triples $(\mathcal{s}, \omega, l)$ such that:
\begin{enumerate}[leftmargin=0pt, itemindent=*]
\item $\mathcal{s} \in \{1, \dots, T\}$;
\item $\omega : \{1,\dots,\bar{m}\} \to \llbracket N \rrbracket $ for some $m \in \N^N$, with $1 \le m_n = |\omega^{-1}(n)| \le S_n$ for every $n \in \llbracket N \rrbracket$, $\bar{m}=\sum_{n=1}^N m_n$, $\omega(1) = 1$, and with no two cyclically consecutive values equal, i.e. $\omega(s) \neq \omega[(s + 1) \bmod \bar{m}]$ for all $s$; The set of such maps is denoted $\mathcal{O}_{m}$.
\item $l = (l_s)_{1 \le s \le \bar{m}}$ with $l_s \in \mathbb{N}$, $l_s \ge 1$, and $\sum_{s \in \omega^{-1}(n)} l_s = L_n$ for every $n \in \llbracket N \rrbracket$;
\item \emph{(cyclic anchoring)} $\mathcal{s} = \min\{a_s : \omega(s) = 1\}$, where $a_s := \big[\big(\mathcal{s} + \textstyle\sum_{r < s} l_r\big) \bmod T \big]$ is the start instant of the $s$-th subperiod; that is, $\mathcal{s}$ is the smallest start instant, in $\{1,\dots,T\}$, among the subperiods of period $1$.
\end{enumerate}
Then the map $\Psi$ sending $(\mathcal{s},\omega,l) \in \mathcal{M}$ to the schedule that assigns period $\omega(s)$ to the $l_s$ consecutive cyclic time steps starting at $a_s$, for $s = 1,\dots,\bar{m}$, is a bijection from $\mathcal{M}$ onto $\mathcal{H}$.
\end{proposition}

\begin{proof}
\begin{itemize}[leftmargin=0pt, itemindent=*]
\item \emph{$\Psi$ maps into $\mathcal{H}$.} Since consecutive segments carry distinct labels, the number of subperiods $n$ is exactly $|\omega^{-1}(n)| \le S_n$ by (2); The length of period $n$ is equal to $L_n$ by (3); Each instant $t$ is assigned a period by exactly one segment, by the cyclic construction; 

\item \emph{Injectivity.} Let $(\mathcal{s},\omega,l)$ and $(\mathcal{s}',\omega',l')$ have the same image $h$. 
\begin{itemize}[leftmargin=0pt, itemindent=*]
\item \emph{Start instant.} The start instant $\mathcal{s}$ is the smallest start instant of a subperiod of period $1$. Hence, $\mathcal{s} = \mathcal{s}' = \min\{t \in \llbracket T \rrbracket \mid h_{t,1} = 1 \text{ and } h_{(t-1) \bmod T, 1} = 0\}$.
\item \emph{Period order.} The cyclic sequence of blocks of $h$ is determined by the sequence of labels of the blocks, which is exactly $\omega$ up to a cyclic rotation. Condition (4) fixes the rotation: it forces the indexing to begin at the block of period $1$ whose start instant is smallest in $\{1,\dots,T\}$ ($\mathcal{s}=\mathcal{s}'$), and by (2) that block is unique since distinct blocks have distinct start instants. Hence $\omega = \omega'$.
\item \emph{Block lengths.} The length of each block is determined by $h$, since we have already established that $\mathcal{s} = \mathcal{s}'$ and $\omega = \omega'$, we iterate over the blocks in the same order and find that $l = l'$.
\end{itemize}

\item \emph{Surjectivity.} Given $h \in \mathcal{H}$, consider the start instant $\mathcal{s}$ of the subperiod of period $1$ that is:
\begin{equation*}
\mathcal{s} = \min\{t \in \llbracket T \rrbracket \mid h_{t,1} = 1 \text{ and } h_{(t-1) \bmod T, 1} = 0\}.
\end{equation*}
Then define $\omega$ as the sequence of labels of the blocks of $h$ starting from the block that contains $\mathcal{s}$, and define $l$ as the lengths of those blocks. By construction, $(\mathcal{s},\omega,l)$ satisfies (1)--(4) and $\Psi(\mathcal{s},\omega,l) = h$. \qedhere
\end{itemize}
\end{proof}

For the two instances used in Section~\ref{sec:num}, namely $T=24$, $S \equiv 2$ and $L = (8,16)$ or $L = (8,4,12)$, we have that $|\mathcal{H}| = 1284$ and $|\mathcal{H}| = 29016$, respectively (Table~\ref{tab:schedule:card}). Nevertheless, the number of admissible orders $\omega$ is much smaller than the number of admissible schedules $h$: for the same two instances $|\mathcal{O}| = 2$ and $|\mathcal{O}| = 24$, so that the orders account for a factor of $642$ and $1209$ less than the schedules themselves (the cardinalities were computed by brute enumeration). Here, we suggest to exploit that fact by enumerating the admissible orders $\omega$ and optimizing over the continuous relaxation of the lengths and starting instant for each order. Hence, the direct method is a hybrid: it is continuous in the lengths and starting instant, but discrete in the order. From here onward, we fix the subperiod order $\omega$ and relax the integrality of the starting instant and the lengths. The feasible set of the starting instant is $[0, T]$, and the feasible set of lengths is:
\begin{equation*}
\mathcal{L}(\omega) = \Big\{ \{l_s\}_{s \in \omega^{-1}(n), n \in N} \subseteq [1, \infty) \;\Big|\; \textstyle\sum_{s \in \omega^{-1}(n)} l_s = L_n,\forall n \in \llbracket N \rrbracket \Big\},
\end{equation*}
which factorizes by period.

\begin{lemma} \label{lem:boxed-simplex}
The Euclidean projection onto $\mathcal{Q} = \{q \in \R^{N} \mid p^{\lowerb} \le q_1 \le \cdots \le q_{N} \le p^{\upperb}\}$ can be computed in $O(N)$ operations, and the Euclidean projection onto $\mathcal{L}(\omega)$ in $O(\bar{S} \log \bar{S})$ operations.
\end{lemma}

\begin{proof}
\begin{itemize}[leftmargin=0pt, itemindent=*]
\item \emph{Projection onto $\mathcal{Q}$.} The set $\mathcal{Q}$ is the intersection of the isotonic cone $\mathcal{C} := \{q \in \R^N \mid q_1 \le \cdots \le q_N\}$ with the box $[p^{\lowerb},p^{\upperb}]^N$. The Euclidean projection onto $\mathcal{C}$ alone is the isotonic regression of the point to be projected, and is computed exactly by the pool-adjacent-violators algorithm (PAVA) in $O(N)$ operations (\cite{de2010isotone}). For an isotonic cone intersected with a box whose bounds are the same in every coordinate, the projection is obtained from the isotonic solution by componentwise clipping to $[p^{\lowerb},p^{\upperb}]$; note that clipping preserves monotonicity, since $u \mapsto \min\{p^{\upperb}, \max\{p^{\lowerb}, u\}\}$ is nondecreasing, so the clipped vector is indeed feasible. The total cost is $O(N)$.

\item \emph{Projection onto $\mathcal{L}(\omega)$.} The set $\mathcal{L}(\omega)$ is a Cartesian product of $N$ boxed simplices, the $n$-th one carrying the $m_n$ coordinates indexed by $\omega^{-1}(n)$, lower bounded by $1$ and that sum up to $L_n$. The projection therefore decomposes across periods, and the projection onto a boxed simplex of size $m_n \leq S_n$ is computed in $O(S_n \log S_n)$ operations by the algorithm of~\cite{condat2016fast}. Summing over $n$ gives $O\big(\sum_n S_n \log S_n\big) = O(\bar{S} \log \bar{S})$. \qedhere
\end{itemize}
\end{proof}

Let $\omega$ be fixed, and let $\mathcal{s}$ and $l$ be continuous start instant and lengths, respectively. The $s$-th subperiod interval is defined as:
\begin{equation*}
I(\mathcal{s}, l)(s) = \begin{cases}
\big(\mathcal{s} + \textstyle\sum_{r<s} l_r,\; \mathcal{s} + \textstyle\sum_{r\le s} l_r\big], & \text{if } \mathcal{s} + \textstyle\sum_{r\le s} l_r \le T, \\[4pt]
\big(\mathcal{s} + \textstyle\sum_{r<s} l_r,\; T\big] \cup \big[0,\; \mathcal{s} + \textstyle\sum_{r\le s} l_r - T\big], & \text{otherwise},
\end{cases}
\end{equation*}
and the price map $G: \mathcal{Q} \times [0, T] \times \mathcal{L}(\omega) \to \R^{T}$,
\begin{equation*}
G(q, \mathcal{s}, l)_t = \textstyle\sum_{s=1}^{\bar{S}} q_{\omega(s)} \big|(t-1, t] \cap I(\mathcal{s}, l)(s)\big|.
\end{equation*}
The relaxed price at instant $t$ is thus a convex combination of the period prices, with weights given by the fraction of the hour $(t-1,t]$ spent in each subperiod. This is a meaningful object in its own right, and it is what makes the relaxation continuous in the boundary positions. 

\begin{lemma}[The price map is an elementary selection] \label{lem:g-selection}
Given the subperiod order $\omega$, the price map $G:(q,\mathcal{s},l)\mapsto \big(G(q,\mathcal{s},l)_t\big)_{t\in \llbracket T \rrbracket}$ is an elementary selection whose selection derivative can be computed in $O(T\bar{S}^2)$. 
\end{lemma}

\begin{proof}
Write the endpoints of the $s$-th subperiod, as $a_s(l)=\mathcal{s} + \textstyle\sum_{r<s} l_r$ and $b_s(l)=a_s(l)+l_s$, both affine in $l$. The actual arc of subperiod $s$ on the cycle $[0,T]$ is $(\mathcal{s}+a_s, \mathcal{s}+b_s]$ taken modulo $T$, as in $I(\mathcal{s},l)(s)$. The overlap of this arc with the $t$-th cell $(t-1,t]$, $o_{t}^s(\mathcal{s},l) = \big|(t-1,t]\cap I(\mathcal{s},l)(s)\big|$, is equal to:
\begin{equation*}
o_{t}^s(\mathcal{s},l) = (\min\{b_s,\;t\}-\max\{a_s,\;t-1\})^+ + (\min\{b_s-T,\;t\}-\max\{a_s-T,\;t-1\})^+.
\end{equation*}
where $(\cdot)^+$ denotes the positive part, i.e., $(x)^+ = \max\{0, x\}$. Each term is a composition of $\min$, $\max$, the positive part and affine maps of $(\mathcal{s},l)$, all of which are elementary selections; by stability of $\mathcal S$ under composition, sum and product~\cite[Proposition 1]{bolte2020mathematical}, so is $G(q,\mathcal{s},l)_t = \textstyle\sum_{s} q_{\omega(s)} o_{t}^s(\mathcal{s},l)$, componentwise in $t$. The overlaps $o_{t}^s$ are piecewise affine in $(\mathcal{s},l)$, so the selected pieces are affine and their derivatives are read off directly. Computing all $T\bar S$ overlaps and the associated derivatives naively costs $O(T\bar S^2)$, since each $a_s$ is a prefix sum of $l$ and each $o_{t}^s$ depends on all $l_r$ with $r \le s$. \qedhere
\end{proof}

Before stating the algorithm we must address a point on which the pointwise case and the ToU case genuinely differ. The boundedness argument of Proposition~\ref{prop:iterate-boundness} rests on the projection onto a \emph{box} being locally constant outside that box, so that a coordinate which leaves it has zero selection derivative and freezes. The projections $\Proj_{\mathcal{Q}}$ and $\Proj_{\mathcal{L}(\omega)}$ are onto polyhedra that are not boxes: outside the set, such a projection varies with the iterate in every direction except the normal ones, so its selection derivative does not vanish and no freezing occurs. The argument therefore does \emph{not} extend, and without a further device the iterates $(q, \mathcal{s}, l)$ could drift without bound. We restore boundedness at negligible cost by composing with an outer clipping onto a box.

\smallskip

Given $\omega \in \mathcal{O}$, write $u := (q, \mathcal{s}, l) \in \R^{N} \times \R \times \R^{\bar S}$ for the decision vector of the relaxed problem, and let $B := [-\Lambda, \Lambda]^{N + 1 + \bar S}$ be any box with $\Lambda$ large enough that $\mathcal{Q} \times [0,T] \times \mathcal{L}(\omega) \subseteq B$; taking $\Lambda \ge \max\{\overline{p}, T\}$ suffices. The relaxed time-of-use pricing problem is:
\begin{equation} \label{eq:rtou-safeguard}
\widetilde{F}(u) := F\Big(G\big(\Proj_{\mathcal{Q}}(\hat u_q), \Proj_{[0, T]}(\hat u_{\mathcal{s}}), \Proj_{\mathcal{L}(\omega)}(\hat u_l)\big)\Big), \tag{$\RToU$}
\end{equation}
where $\hat u := \Proj_{B}(u)$. Since $\mathcal{Q} \times [0,T] \times \mathcal{L}(\omega) \subseteq B$ and the inner projections are the identity on their own sets, every admissible triple is attained by some $u \in B$, so maximizing~\eqref{eq:rtou-safeguard} over $\R^{N+1+\bar S}$ is equivalent to maximizing $F(G(\cdot))$ over $\mathcal{Q} \times [0,T] \times \mathcal{L}(\omega)$.

\begin{corollary} \label{corollary:relaxed-tou-selection}
Given the subperiod order $\omega \in \mathcal{O}$, the relaxed time-of-use pricing problem \eqref{eq:rtou-safeguard} is equivalent to maximizing an elementary selection whose selection derivative can be computed in $O(T(\bar{S}^2 + \log T))$.
\end{corollary}

\begin{proof}
The relaxed objective is $(q,\mathcal{s},l)\mapsto F\!\big(G(\Proj_{\mathcal{Q}}(q), \Proj_{[0, T]}(\mathcal{s}), \Proj_{\mathcal{L}(\omega)}(l))\big)$. Its three constituents lie in $\mathcal S$: the projections $\Proj_{\mathcal{Q}}$, $\Proj_{[0, T]}$, and $\Proj_{\mathcal{L}(\omega)}$ are projections onto polyhedra, hence piecewise affine and elementary selections, as in Theorem~\ref{thm:cp:closed-form}; the price map $G$ is an elementary selection (Lemma~\ref{lem:g-selection}); and the pointwise objective $F$ is an elementary selection in $p$ (Theorem~\ref{thm:pointwise-selection}). Since $\mathcal S$ is closed under composition~\cite[Proposition 1]{bolte2020mathematical}, the relaxed objective belongs to $\mathcal S$.\qedhere
\end{proof}

\begin{proposition} \label{prop:iterate-boundness-tou}
The iterate sequence $(u_i)_{i \ge 0}$ generated by Algorithm~\ref{alg:sgd-tou} applied to the objective~\eqref{eq:rtou-safeguard} is bounded for every initial point.
\end{proposition}

\begin{proof}
The outermost operation applied to the iterate in~\eqref{eq:rtou-safeguard} is the clipping $\Proj_B$ onto a box, which acts coordinatewise and is locally constant at every point lying strictly outside $B$ in the corresponding coordinate. The three steps of the proof of Proposition~\ref{prop:iterate-boundness} therefore apply verbatim, with $\Proj_{[\underline{p},\overline{p}]}$ replaced by $\Proj_B$: a coordinate that leaves $B$ has zero selection derivative and is frozen thereafter, and every coordinate is consequently bounded. \qedhere
\end{proof}

The safeguard makes the guarantee available at the price of an absorbing set, and it is only informative at accumulation points lying in $B$. Since $B$ can be taken as large as desired without affecting the value of the problem, this is a mild price; but the safeguard must actually be implemented, and its activation monitored, for Corollary~\ref{cor:sgd-convergence-tou} to apply to a given run.

\begin{algorithm}[!ht]
\caption{Stochastic gradient ascent for time-of-use pricing}
\begin{algorithmic}[1]
\Require subperiod order $\omega \in \mathcal{O}$; initial start instant $\mathcal{s}$, lengths $l$ and level prices $q_0$; safeguard box $B$; price bounds $p^{\lowerb},p^{\upperb}$; client data; step sizes $(\eta_i)_{i\ge 0}$; batch size $b$; number of iterations $N_{\mathrm{it}}$; 
\Ensure an iterate sequence whose accumulation points are almost surely Clarke critical for the objective~\eqref{eq:rtou-safeguard}
\For{$i = 0,\dots,N_{\mathrm{it}}-1$}
\State sample a mini-batch $B_i \subseteq \llbracket K \rrbracket$ uniformly without replacement, $|B_i| = b$
\State $g \gets 0$
\State $\hat{p} \gets G\big(\Proj_{\mathcal{Q}}(q), \Proj_{[0, T]}(\mathcal{s}), \Proj_{\mathcal{L}(\omega)}(l)\big)$
\For{$k \in B_i$}
\State compute the active sets $(L_k,I_k,U_k)$ of $x_k^\star(p)$ and the multiplier $\gamma_k^x(p)$ 
\State $x_k^\star \gets \Proj_{\mathcal{X}_k}(x_k^0 - \alpha_k \hat{p})$;\quad
$y_k^\star \gets \Proj_{[0,1]}\!\big(\beta_k(\Vill^0_k - \Vill_k({\hat{p}}))\big)$
\State compute the phase $\sigma_k(p)$ and the selection derivative $\widehat\nabla_{\sigma_k} \widetilde{F}_k(p)$
\State $g \gets g + b^{-1} K w_k \widehat\nabla_{\sigma_k} \widetilde{F}_k(p)$
\EndFor
\State $\mathcal{s} \gets \mathcal{s} + \eta_i g_{\mathcal{s}}$, $l \gets l + \eta_i g_{l}$, $q \gets q + \eta_i g_q$ 
\EndFor
\State \Return $\mathcal{s}, l, q$
\end{algorithmic}
\label{alg:sgd-tou}
\end{algorithm}

\begin{corollary} \label{cor:sgd-convergence-tou}
Fix $\omega \in \mathcal{O}$. Let for all $i$, $\eta_i = c r_i$, and $r_i = o(1 / \log{i})$. Then, for almost every $c \in (0,1]$, almost surely every accumulation point of the sequence generated by Algorithm~\ref{alg:sgd-tou} is Clarke critical for the objective~\eqref{eq:rtou-safeguard}.
\end{corollary}

\section{Numerical experiments}
\label{sec:num}

We apply our model and solution methods to a realistic case study inspired by the French retail electricity market. The numerical experiments pursue four goals: (i) to validate the gradient-based solver against a commercial solver on the pointwise problem; (ii) to quantify how the retailer's prices, the induced consumption shifts, and client participation respond to the flexibility of the clients; (iii) to measure the \emph{price of representability} (PoR) of Definition~\ref{def:por}, that is, the profit the retailer forgoes when unconstrained pointwise prices are restricted to admissible time-of-use (ToU) schedules; and (iv) to compare the ToU methods that remain tractable at this scale, namely the rounding method of Section~\ref{subsubsec:rounding}, the direct method of Section~\ref{subsubsec:direct} and, in the configurations where it terminates, the indirect method of Section~\ref{subsubsec:indirect}. As established in Remark~\ref{rem:enumeration-cost}, the complete enumeration benchmark of Section~\ref{subsubsec:enumeration} is out of reach for the instances considered here, so the optimal ToU profit $F^\star_{\mathrm{ToU}}$ entering Definition~\ref{def:por} is replaced by the best value found by those three methods; every price of representability reported below is therefore an \emph{upper estimate} of the true one, and the comparison between methods is a comparison of local optima of a nonconvex problem rather than of certified optima.

\subsection{Experimental setup}
We describe in turn the dataset, the model parameters, the schedule constraints, and the computational environment.

\subsubsection{Dataset}
Data were generated by the SMACH simulator of EDF~R\&D~\cite{Huraux}, which simulates $5000$ residential households (clients). Each household is described by numerical attributes: number of occupants, average consumed power, subscribed power (PS), and an electric domestic-hot-water indicator (ECS); and categorical attributes: dwelling surface, dwelling type, electric heating, air conditioning, and electric-vehicle ownership. The numerical attributes are standardized to zero mean and unit variance and the categorical attributes are one-hot encoded, after which the feature vectors are partitioned by $k$-means into $K\in\{10, 50, 100\}$ clusters. For each cluster, the client closest (in Euclidean distance) to the centroid is retained as the representative load profile $x^0_k$, and the cluster cardinality, normalized so that $\sum_{k}w_k=1$, is used as the weight $w_k$; the retailer's objective is therefore the population-weighted \emph{average} yearly profit per representative client. Because the weights sum to one, profits obtained for different values of $K$ are directly comparable: they all estimate the same population average, with an accuracy that improves as the clustering becomes finer. The experiments reported below use $K \in \{10, 50, 100\}$, which is the range over which the indirect method remains solvable within the time budget. Figure~\ref{fig:ref-consumption} displays the resulting weighted-average reference load, which exhibits the two consumption peaks typical of a French residential day, around $8$--$10$, and a nightly trough between $22$ and $24$. 

\begin{figure}[!htbp]
\centering
\begin{tikzpicture}
\begin{axis}[
width=10cm, height=4.5cm, 
xlabel={Instant $t$}, ylabel={{$x^0$}},
grid=major, ymin=0.47, ymax=.9,
xtick={1,5,10,15,20,24},
minor xtick={2,3,4,6,7,8,9,11,12,13,14,16,17,18,19,21,22,23},
grid=major,
xminorgrids=true,
minor grid style={very thin, gray!15},
]
\addplot[myblue, mark=*] table[
col sep=comma,
x=Hour,
y=Mean5000,
] {experiments/consumption_summary.csv};
\end{axis}
\end{tikzpicture}
\caption{The weighted-average reference consumption $\sum_{k=1}^K x^0_k w_k$ (in kWh per hour).}
\label{fig:ref-consumption}
\end{figure}
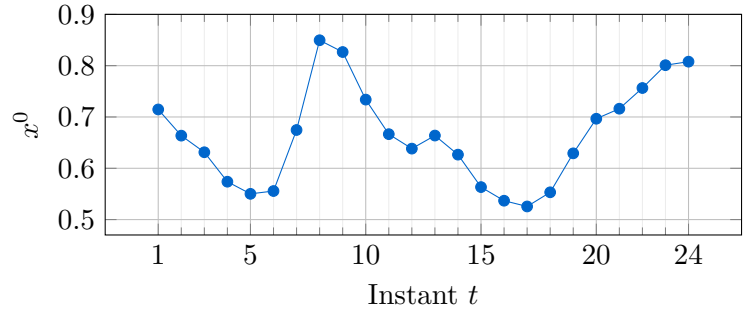

\subsubsection{Parameters}
The consumption data were aggregated into one-hour intervals over a full day, so $T=24$. All monetary quantities are expressed in $\EUR$/kWh for unit prices and $\EUR$ for bills and profit per year ($365.25$ days). For simplicity, we considered a linear cost function $\Cost$: 
\begin{equation*}
\Cost(x) = \textstyle\sum_{t=1}^T c_t x_{t},
\end{equation*}
where the marginal cost $c_t$ was computed from \href{https://www.epexspot.com/en/market-results}{spot prices}, transmission tariffs, and French taxes. 

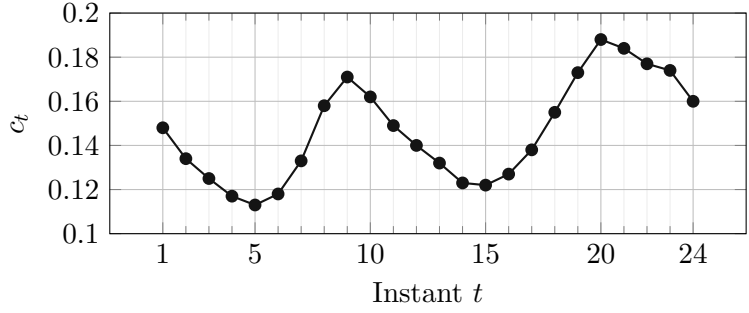
\begin{figure}[!htbp]
\centering
\begin{tikzpicture}
\begin{axis}[width=10cm, height=4.5cm, xlabel={Instant $t$}, ylabel={$c_t$}, grid=major, ymin=0.1, ymax=0.2,
xtick={1,5,10,15,20,24},
minor xtick={2,3,4,6,7,8,9,11,12,13,14,16,17,18,19,21,22,23},
grid=major,
xminorgrids=true,
minor grid style={very thin, gray!15},]
\addplot[myblack, thick, mark=*] coordinates {(1, 0.148) (2, 0.134) (3, 0.125) (4, 0.117) (5, 0.113) (6, 0.118) (7, 0.133) (8, 0.158) (9, 0.171) (10, 0.162) (11, 0.149) (12, 0.14) (13, 0.132) (14, 0.123) (15, 0.122) (16, 0.127) (17, 0.138) (18, 0.155) (19, 0.173) (20, 0.188) (21, 0.184) (22, 0.177) (23, 0.174) (24, 0.16) }; 
\end{axis}
\end{tikzpicture}
\caption{The cost profile $c_t$ in $\EUR$/kWh.}
\label{fig:punc:prices}
\end{figure}

\smallskip

The value $\Vill^0_k$ of the outside offer was set to be equal to $\Vill_k((1+\xi)c)$ (pointwise prices equal to costs inflated by a margin~$\xi$). Client flexibility is encoded through the bounds $x^{\lowerb}_k=(1-f_k) x^0_k$ and $x^{\upperb}_k=(1+f_k) x^0_k$, where $f_k$ is the fraction of consumption movable at each hour. The admissible price range is $[p^{\lowerb},p^{\upperb}]=[0.05, 0.35]\EUR$/kWh. Unless stated otherwise we set, for every client $k$, $\xi=0.05$, $\beta_k=365.25/100$, and $f_k=0.1$, and we vary a single common flexibility level $\alpha_k\equiv\alpha\in\{0.1,1,10\}$. Thus the outside offer is $5\%$ above cost, each client can shift up to $10\%$ of its hourly consumption, and $100\EUR$ is the width of the participation transition region per year: a client's participation level $y_k$ rises from $0$ to $1$ as the accepted bill undercuts the outside option $\Vill^0_k$ by up to $100\EUR$ (cf. Figure~\ref{fig:client-choice}). 

\subsubsection{Schedule constraints}
\label{subsubsec:num:schedule}
We consider two scheduling models that instantiate the set $\mathcal{H}$ of Section~\ref{sec:mod}. Both mirror the structure of the regulated French \emph{heures pleines / heures creuses} tariff.
\begin{itemize}[leftmargin=0pt, itemindent=*]
\item \emph{Two-period.} Two periods (off-peak, on-peak), each split into at most $2$ subperiods; $8$ off-peak hours and $16$ on-peak hours: $N=2, S=2, L_1=8, L_2=16$.
\item \emph{Three-period.} Three periods (off-peak, mid-peak, on-peak), each split into at most $2$ subperiods; $8$ off-peak, $4$ mid-peak, and $12$ on-peak hours: $N=3, S=2, L_1=8, L_2=4, L_3=12$.
\end{itemize}
The cardinalities of the two schedule sets are reported in Table~\ref{tab:schedule:card}: the extra period enlarges the schedule space by a factor $29016/1284 \approx 22.6$, even though only one additional block boundary is introduced. This combinatorial growth motivates the algorithmic choices of Section~\ref{sec:algorithmic}, especially when considering the multi-day scenario (cf. Appendix~\ref{app:multi-day}).

\begin{table}[!htbp]
\centering
\begin{tabular}{c c c}
\toprule
& Two-period schedules & Three-period schedules \\
\midrule
Cardinality & 1284 & 29016 \\
\bottomrule
\end{tabular}
\caption{Cardinality of the schedule sets for the two-period and three-period models.}
\label{tab:schedule:card}
\end{table}

\subsubsection{Computational setup}
The experiments were conducted on a 13th~Gen Intel\textregistered{} Core\texttrademark{} i5-13600H CPU (up to 4.8~GHz) with 32~GB of RAM. The single-level mixed integer program~\eqref{model:cp:kkt-conditions} was formulated in AMPL and solved with \textsc{Knitro} on the NEOS Server; being a local solver applied to a nonconvex mixed-integer quadratically constrained program, it returns a feasible incumbent and a local optimum rather than a global certificate. The direct, closed-form formulation (the maximization of $F(p)$ from Theorem~\ref{thm:pointwise-selection}) was implemented in Python using PyTorch and optimized with stochastic gradient ascent. For reproducibility we report the solver settings. The recursion was run from the initial point $p_0=0.12$ with learning rate~$5 \times 10^{-5}$ for the pointwise and $5 \times 10^{-6}$ for the ToU version, with a scaling of $10^{6}$ for the gradient with respect to the continuous start instant and the lengths, mini-batch size $b=10$, and $N_{\text{it}}=10^{5}$ iterations; unless stated otherwise a single start was used. The rounding step of Section~\ref{subsec:tou-pricing} was solved as a mixed-integer quadratic program using \textsc{CPLEX} with default parameters.

\subsection{Pointwise pricing}
\label{subsec:num:pointwise}
In the current section, we treat the pointwise model (cf. Section~\ref{subsec:pointwise-pricing}). First, to validate the direct method by comparison with the indirect method, we report the relative profit difference and the CPU time. Then, we analyze how the retailer's prices, the induced consumption shifts, and client participation respond to the flexibility of the clients.

\subsubsection{Profit}
\label{subsubsec:num:pointwise-profit}

Table~\ref{tab:pointwise:profit} compares the direct method that implements Algorithm~\ref{alg:sgd-pointwise} with the indirect method that solves the single-level reformulation of Corollary~\ref{cor:cp:kkt-conditions} using \textsc{Knitro}. The purpose of the comparison is validation: since neither route certifies global optimality, agreement between two methods that treat the optimization problem from different perspectives (closed-form differentations, and MINLP) is an important sign that both approaches are reliable.

\smallskip

The two methods agree closely: over the nine configurations the relative gap $\Delta$ lies between $-0.13\%$ and $0.29\%$, i.e. at most three cents on a profit of about $10\EUR$ per representative client. The results read as both methods reaching comparable local optima of a nonconvex objective rather than as a systematic advantage of either. Two regularities are visible across the table. First, profit decreases with the number of representative clients, from $10.26\EUR$ at $K=10$ to $9.98\EUR$ at $K=100$ for $\alpha=0.1$: a coarse clustering concentrates the population on a few profiles that the retailer can target precisely, whereas a finer one exposes the heterogeneity that a single tariff must serve simultaneously. Second, and more surprisingly at first sight, profit decreases as clients become more flexible, from $10.26\EUR$ to $10.05\EUR$ at $K=10$ and from $9.98\EUR$ to $9.74\EUR$ at $K=100$ when $\alpha$ moves from $0.1$ to $10$. Flexibility is valuable to the retailer only through the load it displaces towards cheap hours, but it is equally available to the client under the outside option, whose value $\Vill^0_k = \Vill_k((1+\xi)c)$ decreases with $\alpha$, i.e. becomes more attractive to the client; the second effect dominates here, so flexibility mostly strengthens the client's fallback position. 

\begin{table}[!htbp]
\centering
\begin{tabular}{c | cc | cc | cc}
\toprule
\multicolumn{1}{c}{} & \multicolumn{2}{c}{10} & \multicolumn{2}{c}{50} & \multicolumn{2}{c}{100} \\
\cmidrule(lr){1-1} \cmidrule(lr){2-3} \cmidrule(lr){4-5} \cmidrule(lr){6-7}
$\alpha$ & Direct & $\Delta$ & Direct & $\Delta$ & Direct & $\Delta$ \\
\midrule
0.1 & 10.26 & 0.09\% & 10.18 & -0.13\% & 9.98 & 0.11\% \\
1 & 10.20 & 0.16\% & 10.08 & 0.07\% & 9.88 & 0.29\% \\
10 & 10.05 & 0.08\% & 9.87 & 0.16\% & 9.74 & 0.07\% \\
\bottomrule
\end{tabular}
\caption{Pointwise pricing model profit comparison using the direct and the indirect methods. For each client count: \emph{Profit} is the weighted-average profit per representative client ($\EUR$); $\Delta$ is the relative profit difference versus the indirect method.}
\label{tab:pointwise:profit}
\end{table}

\subsubsection{Computational time}
Table~\ref{tab:pointwise:cpu} reports CPU time (hh:mm:ss). The two methods scale in opposite ways. The indirect method solves a mixed-integer program whose size grows linearly in $K$: the reformulation~\eqref{model:cp:kkt-conditions} carries $2TK$ binary variables for the consumption bounds and $2K$ for the participation bounds, so at $T=24$ the model grows from $500$ to $5000$ binaries between $K=10$ and $K=100$, which makes the branch-and-bound more demanding. The direct method, by contrast, performs a fixed number $N_{\mathrm{it}}$ of iterations whose cost is $O(bT\log T)$ by Corollary~\ref{cor:cp:closed-form-complexity}, hence independent of $K$ once the mini-batch size $b$ is fixed.

\smallskip

The indirect method is essentially instantaneous at $K=10$ ($1$ to $11$ seconds) but requires between $8$ and $20$ minutes at $K=100$, a growth of three orders of magnitude for a tenfold increase in the number of clients. The direct method stays between $2$ and $9$ minutes throughout, with no trend in $K$. Since the retailer's interest lies in fine clusterings, the flat profile of the direct method is the operationally relevant property, and it is what makes the closed-form route the method of choice for such models.

\begin{table}[!htbp]
\centering
\begin{tabular}{c | cc | cc | cc}
\toprule
\multicolumn{1}{c}{} & \multicolumn{2}{c}{10} & \multicolumn{2}{c}{50} & \multicolumn{2}{c}{100} \\
\cmidrule{1-1} \cmidrule(lr){2-3} \cmidrule(lr){4-5} \cmidrule(lr){6-7}
$\alpha$ & Direct & Indirect & Direct & Indirect & Direct & Indirect \\
\midrule
0.1 & 02:19 & 00:01 & 05:18 & 01:51 & 09:00 & 16:06 \\
1 & 02:19 & 00:11 & 03:03 & 03:54 & 02:34 & 20:13 \\
10 & 07:56 & 00:01 & 02:05 & 03:49 & 02:55 & 08:32 \\
\bottomrule
\end{tabular}
\caption{Pointwise pricing model CPU time comparison using the direct and the indirect methods.}
\label{tab:pointwise:cpu}
\end{table}

\subsubsection{Prices}

Figure~\ref{fig:pointwise:prices} shows the hourly pointwise prices for $K=50$ clients, obtained by both methods and for the three flexibility levels. At low flexibility ($\alpha=0.1$) the price profile is highly variable and tracks the cost profile of Figure~\ref{fig:punc:prices}. As $\alpha$ increases to $1$ and then $10$ the profile flattens markedly. At several off-peak hours the optimal price falls below the marginal cost $c_t$ of Figure~\ref{fig:punc:prices}: the retailer uses loss-leading off-peak prices to induce load shifting and recovers margin at peak, which is possible precisely because the daily energy total is conserved by the constraint defining $\mathcal{X}_k$. The direct and the indirect methods yield visually similar profiles, hour by hour and at all three flexibility levels, which is the qualitative counterpart of the profit gap of Table~\ref{tab:pointwise:profit}: the two methods do not merely reach the same value, they reach nearly the same price vector.

\begin{figure}[!htbp]
\centering
\resizebox{\textwidth}{!}{%
\begin{tabular}{ccc}
\begin{tikzpicture}
\begin{axis}[width=5cm, height=4.5cm, grid=major, ymin=0.06, ymax=0.28,
xtick={1,5,10,15,20,24},
minor xtick={2,3,4,6,7,8,9,11,12,13,14,16,17,18,19,21,22,23},
grid=major,
xminorgrids=true,
minor grid style={very thin, gray!15},]
\addplot[myblue, thick, mark=*] table[
col sep=comma,
x=Hour,
y=Prices_Direct_0.1_50,
] {experiments/pointwise_prices_summary.csv};
\end{axis}
\end{tikzpicture} &

\begin{tikzpicture}
\begin{axis}[width=5cm, height=4.5cm, grid=major, ymin=0.06, ymax=0.28,
xtick={1,5,10,15,20,24},
minor xtick={2,3,4,6,7,8,9,11,12,13,14,16,17,18,19,21,22,23},
grid=major,
xminorgrids=true,
minor grid style={very thin, gray!15},]
\addplot[myred, thick, mark=*] table[
col sep=comma,
x=Hour,
y=Prices_Direct_1.0_50,
] {experiments/pointwise_prices_summary.csv};
\end{axis}
\end{tikzpicture} &

\begin{tikzpicture}
\begin{axis}[width=5cm, height=4.5cm, grid=major, ymin=0.06, ymax=0.28,
xtick={1,5,10,15,20,24},
minor xtick={2,3,4,6,7,8,9,11,12,13,14,16,17,18,19,21,22,23},
grid=major,
xminorgrids=true,
minor grid style={very thin, gray!15},]
\addplot[mygreen, thick, mark=*] table[
col sep=comma,
x=Hour,
y=Prices_Direct_10.0_50,
] {experiments/pointwise_prices_summary.csv};
\end{axis}
\end{tikzpicture} \\

\begin{tikzpicture}
\begin{axis}[width=5cm, height=4.5cm, xlabel={Instant $t$}, grid=major, ymin=0.06, ymax=0.28,
xtick={1,5,10,15,20,24},
minor xtick={2,3,4,6,7,8,9,11,12,13,14,16,17,18,19,21,22,23},
grid=major,
xminorgrids=true,
minor grid style={very thin, gray!15},]
\addplot[myblue, thick, mark=*] table[
col sep=comma,
x=Hour,
y=Prices_Indirect_0.1_50,
] {experiments/pointwise_prices_summary.csv};
\end{axis}
\end{tikzpicture} &
\begin{tikzpicture}
\begin{axis}[width=5cm, height=4.5cm, xlabel={Instant $t$}, grid=major, ymin=0.06, ymax=0.28,
xtick={1,5,10,15,20,24},
minor xtick={2,3,4,6,7,8,9,11,12,13,14,16,17,18,19,21,22,23},
grid=major,
xminorgrids=true,
minor grid style={very thin, gray!15},]
\addplot[myred, thick, mark=*] table[
col sep=comma,
x=Hour,
y=Prices_Indirect_1.0_50,
] {experiments/pointwise_prices_summary.csv};
\end{axis}
\end{tikzpicture} &

\begin{tikzpicture}
\begin{axis}[width=5cm, height=4.5cm, xlabel={Instant $t$}, grid=major, ymin=0.06, ymax=0.28,
xtick={1,5,10,15,20,24},
minor xtick={2,3,4,6,7,8,9,11,12,13,14,16,17,18,19,21,22,23},
grid=major,
xminorgrids=true,
minor grid style={very thin, gray!15},]
\addplot[mygreen, thick, mark=*] table[
col sep=comma,
x=Hour,
y=Prices_Indirect_10.0_50,
] {experiments/pointwise_prices_summary.csv};
\end{axis}
\end{tikzpicture}
\end{tabular}
}
\caption{Pointwise prices for $50$ clients by the direct method (top) and the indirect method (bottom); $\alpha=0.1$ (blue, left), $\alpha=1$ (red, middle), $\alpha=10$ (green, right). Prices are in $\EUR$/kWh.}
\label{fig:pointwise:prices}
\end{figure}
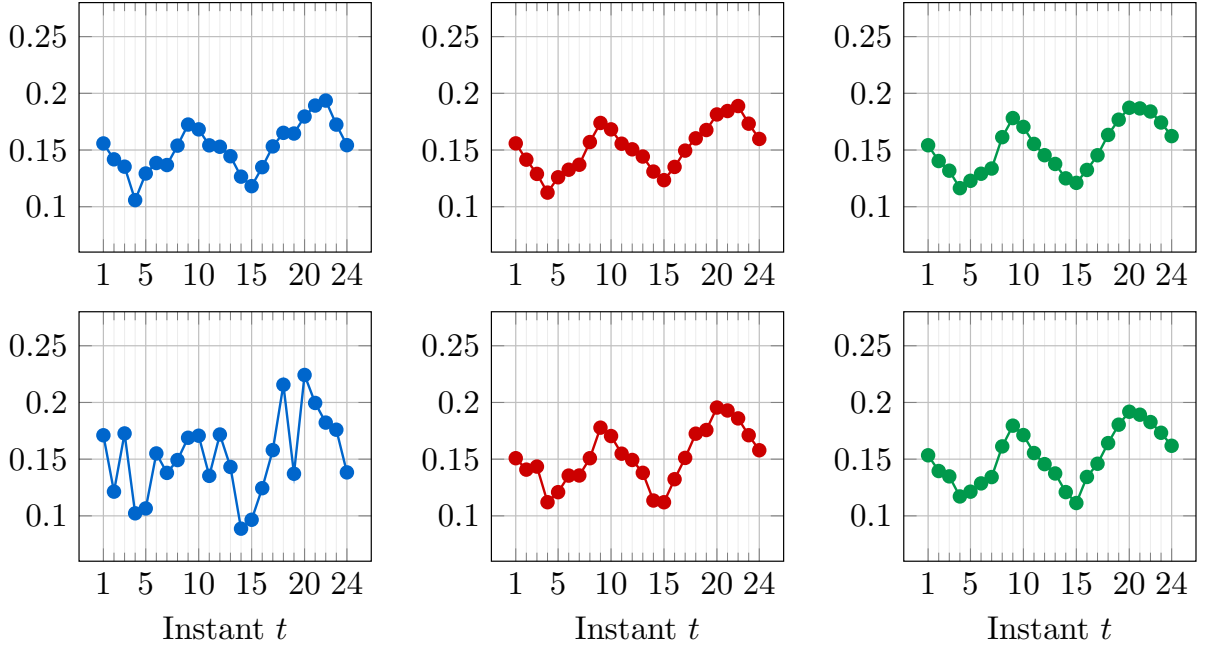

\subsubsection{Load shifting}

Figure~\ref{fig:pointwise:shifts} shows the relative consumption shift $(x^\star-x^0)/x^0$ induced by the direct method's pointwise prices, for $K=50$ clients. At low flexibility ($\alpha=0.1$) clients barely move load: the shift stays within a few percentage points of the baseline across the day. As $\alpha$ increases to $1$ and $10$, the shift widens until it is pinned at the $\pm 10\%$ bounds over most of the day, and its shape mirrors the price profile of Figure~\ref{fig:pointwise:prices}, with consumption drawn away from the priciest hours and into the cheapest ones.

\begin{figure}[!htbp]
\centering
\resizebox{\textwidth}{!}{%
\begin{tabular}{ccc}
\begin{tikzpicture}
\begin{axis}[width=5cm, height=4.5cm, grid=major, ymin=-11, ymax=11,
xtick={1,5,10,15,20,24},
minor xtick={2,3,4,6,7,8,9,11,12,13,14,16,17,18,19,21,22,23},
grid=major,
xminorgrids=true,
minor grid style={very thin, gray!15},]
\addplot[myblue, thick, mark=*] table[
col sep=comma,
x=Hour,
y=shifts_01_50,
] {experiments/pointwise_shifts_summary.csv};
\end{axis}
\end{tikzpicture} &

\begin{tikzpicture}
\begin{axis}[width=5cm, height=4.5cm, grid=major, ymin=-11, ymax=11,
xtick={1,5,10,15,20,24},
minor xtick={2,3,4,6,7,8,9,11,12,13,14,16,17,18,19,21,22,23},
grid=major,
xminorgrids=true,
minor grid style={very thin, gray!15},]
\addplot[myred, thick, mark=*] table[
col sep=comma,
x=Hour,
y=shifts_1_50,
] {experiments/pointwise_shifts_summary.csv};
\end{axis}
\end{tikzpicture} &

\begin{tikzpicture}
\begin{axis}[width=5cm, height=4.5cm, grid=major, ymin=-11, ymax=11,
xtick={1,5,10,15,20,24},
minor xtick={2,3,4,6,7,8,9,11,12,13,14,16,17,18,19,21,22,23},
grid=major,
xminorgrids=true,
minor grid style={very thin, gray!15},]
\addplot[mygreen, thick, mark=*] table[
col sep=comma,
x=Hour,
y=shifts_10_50,
] {experiments/pointwise_shifts_summary.csv};
\end{axis}
\end{tikzpicture}
\end{tabular}
}
\caption{Average relative consumption shift $\sum_{k=1}^K \frac{x_k^\star - x_k^0}{x_k^0} w_k$ in $\%$ for 50 clients under the direct method; $\alpha=0.1$ (blue, left), $\alpha=1$ (red, middle), $\alpha=10$ (green, right).}
\label{fig:pointwise:shifts}
\end{figure}
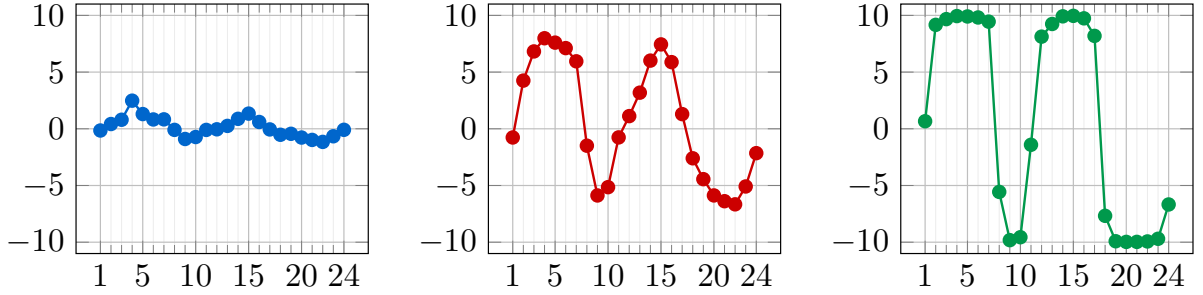

\subsubsection{Participation}

Table~\ref{tab:pointwise:participation} reports the average participation level $\sum_{k=1}^K y_k^\star w_k$ of the clients. The direct and the indirect methods again agree, the largest discrepancy is $1\%$ (at $K=10$, $\alpha=1$), which corroborates once again the reading of Table~\ref{tab:pointwise:profit}. Participation lies in a narrow band, between $18.2\%$ and $21.3\%$ over all 18 entries, and moves in two systematic ways. It decreases with the number of representative clients, from $21.3\%$ to $19.1\%$ at $\alpha=0.1$ between $K=10$ and $K=100$, for the same reason that profit does: a single price vector cannot be equally attractive to a more heterogeneous population. It also decreases with flexibility, by $0.4$ to $1.3$ percentage points between $\alpha=0.1$ and $\alpha=10$ depending on $K$. The latter effect is the participation counterpart of the profit erosion noted in Section~\ref{subsubsec:num:pointwise-profit}: a more flexible client also reshapes its consumption more aggressively under the outside offer, so $\Vill^0_k$ falls and the proposed tariff must concede more to close the same gap. This shows that the proposed offer has the potential to attract around $20$\% of customers, given the parameters chosen.

\begin{table}[!htbp]
\centering
\begin{tabular}{c | cc | cc | cc}
\toprule
\multicolumn{1}{c}{} & \multicolumn{2}{c}{10} & \multicolumn{2}{c}{50} & \multicolumn{2}{c}{100} \\
\cmidrule{1-1} \cmidrule(lr){2-3} \cmidrule(lr){4-5} \cmidrule(lr){6-7}
$\alpha$ & Direct & Indirect & Direct & Indirect & Direct & Indirect \\
\midrule
0.1 & 21.30 \% & 20.68 \% & 19.47 \% & 18.85 \% & 19.11 \% & 19.01 \% \\
1 & 21.27 \% & 20.27 \% & 18.91 \% & 18.88 \% & 18.73 \% & 18.69 \% \\
10 & 20.57 \% & 19.89 \% & 18.22 \% & 18.43 \% & 18.73 \% & 18.50 \% \\
\bottomrule
\end{tabular}
\caption{Pointwise pricing model participation comparison using the direct and the indirect methods.}
\label{tab:pointwise:participation}
\end{table}

\subsection{Relaxed time-of-use pricing}
\label{subsec:num:rtou}
In the current section, we treat the relaxed time-of-use model \eqref{eq:rtou-safeguard}, a middle ground between the pointwise pricing and the binary time-of-use pricing, where the start instant and the lengths are not necessarily integers, and where Algorithm \ref{alg:sgd-tou} was run for all $\omega \in \mathcal{O}$.

\subsubsection{Profit}

Table~\ref{tab:rtou:profit} reports, for each configuration, the price of representability $\PoR_N$ of the relaxed profit with respect to the pointwise profit obtained by the indirect method, where $N$ is the number of periods. Looking at Table~\ref{tab:rtou:profit}, we see that $\PoR_N$ increases with the flexibility in all six columns, from $1.37\%$ to $27.19\%$ for the two-period model at $K=10$ and from $1.23\%$ to $2.56\%$ for the three-period model at $K=100$. The interpretation is similar to the one given for the prices in Section~\ref{subsubsec:num:pointwise-profit}: a flexible client reacts to the shape of the price signal, and a pointwise signal can be shaped hour by hour (which is the case for the outside option), whereas a ToU signal (which is the case for the proposed offer) is constant on each period. When $\alpha$ is small the client barely moves, the shape hardly matters and the two models nearly coincide; when $\alpha$ is large the shape is what produces the profit, and the loss caused by flattening it grows accordingly. In addition, we note that the three-period model dominates the two-period one in eight of the nine cells, the only exception being $K=10$, $\alpha=0.1$, where the two are within $0.23$ percentage points of each other. 

\begin{table}[!htbp]
\centering
\begin{tabular}{c | cc | cc | cc}
\toprule
\multicolumn{1}{c}{} & \multicolumn{2}{c}{10} & \multicolumn{2}{c}{50} & \multicolumn{2}{c}{100} \\
\cmidrule{1-1} \cmidrule(lr){2-3} \cmidrule(lr){4-5} \cmidrule(lr){6-7}
$\alpha$ & $\PoR_2$ & $\PoR_3$ & $\PoR_2$ & $\PoR_3$ & $\PoR_2$ & $\PoR_3$ \\
\midrule
0.1 & \textbf{1.37\%} & 1.60\% & 3.89\% & \textbf{1.66\%} & 1.72\% & \textbf{1.23\%} \\
1 & 2.84\% & \textbf{2.83\%} & 4.57\% & \textbf{1.96\%} & 2.67\% & \textbf{1.46\%} \\
10 & 27.19\% & \textbf{2.83\%} & 10.26\% & \textbf{2.93\%} & 8.01\% & \textbf{2.56\%} \\
\bottomrule
\end{tabular}
\caption{Relaxed time-of-use pricing model profit for the two-period and the three-period models. For each client count: \emph{$\PoR_N$} is the relative profit difference versus the pointwise profit obtained for the indirect method.}
\label{tab:rtou:profit}
\end{table}

\subsubsection{Computational Time}

Table~\ref{tab:rtou:cpu} reports the cumulative CPU time of Algorithm~\ref{alg:sgd-tou} over all orders. Recall that the two-period model has $|\mathcal{O}|=2$ orders while the three-period model has $|\mathcal{O}|=24$, so the three-period runs are expected to be roughly an order of magnitude more expensive than the two-period ones. The measured ratio ranges from $9.37$ (10 clients, $\alpha=0.1$) to $17.66$ (100 clients, $\alpha=0.1$), with an average of $13.97$. Moreover, the per-order computation time is essentially identical for the two schedule models and remains comparable to that of the pointwise model, as predicted by Corollary~\ref{corollary:relaxed-tou-selection}. The most expensive component of the direct method is the projection of the optimal consumption onto each client's feasible set $\mathcal{X}_k$, which requires $O(T\log T)$ time. Since $b$ projections are performed per iteration and the algorithm runs for $N_{\text{it}}$ iterations, the total cost is $O(b\,T\log T\,N_{\text{it}})$; in practice, the stopping criterion terminates the algorithm well before this bound.

\begin{table}[!htbp]
\centering
\begin{tabular}{c | cc | cc | cc}
\toprule
\multicolumn{1}{c}{} & \multicolumn{2}{c}{10} & \multicolumn{2}{c}{50} & \multicolumn{2}{c}{100} \\
\cmidrule{1-1} \cmidrule(lr){2-3} \cmidrule(lr){4-5} \cmidrule(lr){6-7}
$\alpha$ & $N=2$ & $N=3$ & $N=2$ & $N=3$ & $N=2$ & $N=3$ \\
\midrule
0.1 & 03:57 & 37:01 & 03:44 & 01:00:06 & 04:56 & 01:27:07 \\
1 & 03:40 & 42:13 & 03:54 & 01:03:37 & 05:13 & 01:15:34 \\
10 & 03:38 & 52:37 & 04:20 & 01:05:02 & 07:39 & 01:22:33 \\
\bottomrule
\end{tabular}
\caption{Relaxed time-of-use pricing model CPU time using Algorithm \ref{alg:sgd-tou} for all orders.}
\label{tab:rtou:cpu}
\end{table}

\subsubsection{Prices}

Figure~\ref{fig:rtou:prices} shows the relaxed hourly prices for $K=50$ clients. We note that most hours carry one of the period levels exactly, but a few carry an intermediate value. Those are the hours that fall between two periods, in which, the effective price is the convex combination of the period levels weighted by the fraction of the hour spent in each subperiod. 

\smallskip

For the two-period case: the off-peak level is close to $0.12\EUR$/kWh and the on-peak level to $0.167\EUR$/kWh at all three flexibilities; four hours are fractional, and the off-peak weights they carry, added to the fully off-peak hours, return exactly $L_1 = 8$, as the projection onto $\mathcal{L}(\omega)$ of Lemma~\ref{lem:boxed-simplex} enforces. The two off-peak subperiods so recovered sit on the two low points of the cost profile of Figure~\ref{fig:punc:prices}. 

\smallskip

For the three-period case: we note that even though the three levels are available, only two were used in the solution, and the off-peak level accounts for twelve hours, which is exactly $L_1 + L_2 = 8 + 4$. The optimizer therefore sets $q_1 = q_2$: the mid-peak period is priced at the off-peak level, and the third level is never used. The constraint $q_1 \le q_2$ of $\mathcal{Q}$ is thus active at the solution, so the isotonic part of the projection of Lemma~\ref{lem:boxed-simplex} pools the two coordinates. Adding the third period did not create a third distinct price level ($q_1=q_2$), but it expanded the cheap period from 8 to 12 hours ($L_1+L_2 = 12$). Finally, all levels lie strictly inside $[p^{\lowerb}, p^{\upperb}] = [0.05, 0.35]$.

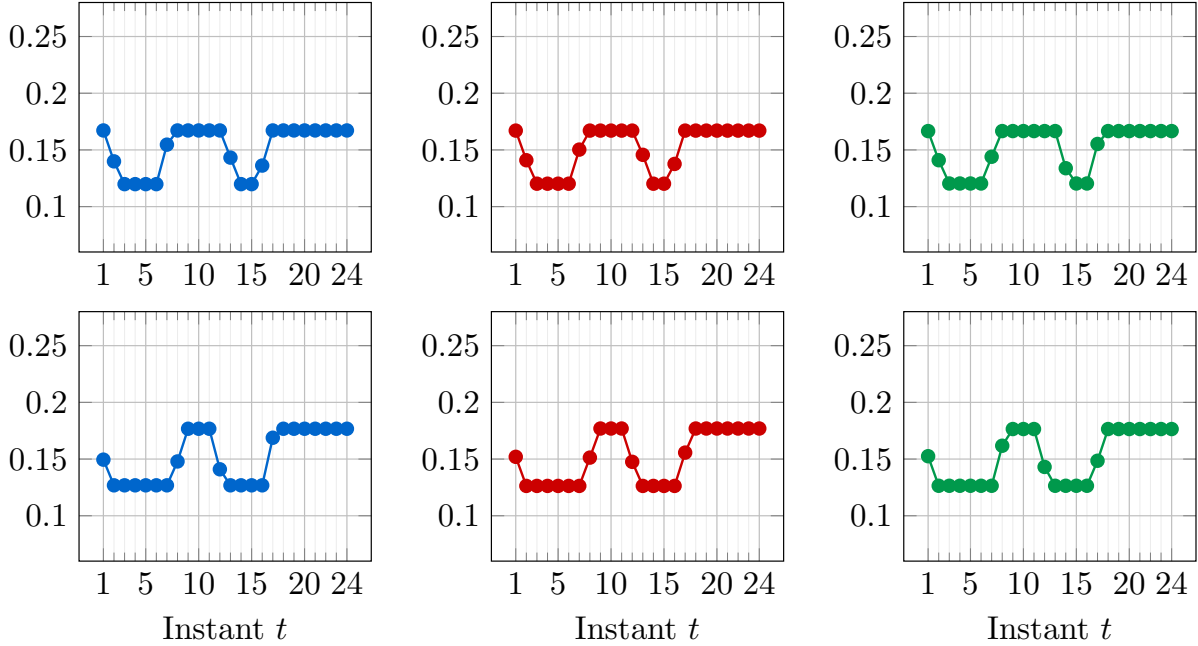
\begin{figure}[!htbp]
\centering
\resizebox{\textwidth}{!}{%
\begin{tabular}{ccc}
\begin{tikzpicture}
\begin{axis}[width=5cm, height=4.5cm, grid=major, ymin=0.06, ymax=0.28,
xtick={1,5,10,15,20,24},
minor xtick={2,3,4,6,7,8,9,11,12,13,14,16,17,18,19,21,22,23},
grid=major,
xminorgrids=true,
minor grid style={very thin, gray!15},]
\addplot[myblue, thick, mark=*] table[
col sep=comma,
x=Hour,
y=Prices_Direct_N2_0.1_50,
] {experiments/rtou_prices_summary.csv};
\end{axis}
\end{tikzpicture} &

\begin{tikzpicture}
\begin{axis}[width=5cm, height=4.5cm, grid=major, ymin=0.06, ymax=0.28,
xtick={1,5,10,15,20,24},
minor xtick={2,3,4,6,7,8,9,11,12,13,14,16,17,18,19,21,22,23},
grid=major,
xminorgrids=true,
minor grid style={very thin, gray!15},]
\addplot[myred, thick, mark=*] table[
col sep=comma,
x=Hour,
y=Prices_Direct_N2_1.0_50,
] {experiments/rtou_prices_summary.csv};
\end{axis}
\end{tikzpicture} &

\begin{tikzpicture}
\begin{axis}[width=5cm, height=4.5cm, grid=major, ymin=0.06, ymax=0.28,
xtick={1,5,10,15,20,24},
minor xtick={2,3,4,6,7,8,9,11,12,13,14,16,17,18,19,21,22,23},
grid=major,
xminorgrids=true,
minor grid style={very thin, gray!15},]
\addplot[mygreen, thick, mark=*] table[
col sep=comma,
x=Hour,
y=Prices_Direct_N2_10.0_50,
] {experiments/rtou_prices_summary.csv};
\end{axis}
\end{tikzpicture} \\

\begin{tikzpicture}
\begin{axis}[width=5cm, height=4.5cm, xlabel={Instant $t$}, grid=major, ymin=0.06, ymax=0.28,
xtick={1,5,10,15,20,24},
minor xtick={2,3,4,6,7,8,9,11,12,13,14,16,17,18,19,21,22,23},
grid=major,
xminorgrids=true,
minor grid style={very thin, gray!15},]
\addplot[myblue, thick, mark=*] table[
col sep=comma,
x=Hour,
y=Prices_Direct_N3_0.1_50,
] {experiments/rtou_prices_summary.csv};
\end{axis}
\end{tikzpicture} &
\begin{tikzpicture}
\begin{axis}[width=5cm, height=4.5cm, xlabel={Instant $t$}, grid=major, ymin=0.06, ymax=0.28,
xtick={1,5,10,15,20,24},
minor xtick={2,3,4,6,7,8,9,11,12,13,14,16,17,18,19,21,22,23},
grid=major,
xminorgrids=true,
minor grid style={very thin, gray!15},]
\addplot[myred, thick, mark=*] table[
col sep=comma,
x=Hour,
y=Prices_Direct_N3_1.0_50,
] {experiments/rtou_prices_summary.csv};
\end{axis}
\end{tikzpicture} &

\begin{tikzpicture}
\begin{axis}[width=5cm, height=4.5cm, xlabel={Instant $t$}, grid=major, ymin=0.06, ymax=0.28,
xtick={1,5,10,15,20,24},
minor xtick={2,3,4,6,7,8,9,11,12,13,14,16,17,18,19,21,22,23},
grid=major,
xminorgrids=true,
minor grid style={very thin, gray!15},]
\addplot[mygreen, thick, mark=*] table[
col sep=comma,
x=Hour,
y=Prices_Direct_N3_10.0_50,
] {experiments/rtou_prices_summary.csv};
\end{axis}
\end{tikzpicture}
\end{tabular}
}
\caption{Relaxed time-of-use prices for $50$ clients under the two-period model (top) and the three-period model (bottom); $\alpha=0.1$ (blue, left), $\alpha=1$ (red, middle), $\alpha=10$ (green, right). Prices are in $\EUR$/kWh.}
\label{fig:rtou:prices}
\end{figure}

\subsubsection{Load shifting and participation}
The load-shifting and participation results for the relaxed time-of-use model closely mirror those of the pointwise model. To avoid repetition, we present them in Appendix~\ref{app:rtou-exp}, where the interested reader will find the full set of figures.

\subsection{Time-of-use pricing}
\label{subsec:num:tou}
In the current section, we treat the time-of-use model (cf. Section~\ref{subsec:tou-pricing}). We compare the three methods that remain tractable at this scale, the rounding method of Section~\ref{subsubsec:rounding} (denoted by \emph{R}), the direct method of Section~\ref{subsubsec:direct} and the indirect method of Section~\ref{subsubsec:indirect}, on the two schedule families of Section~\ref{subsubsec:num:schedule}. An entry N/A means that the solver returned no feasible incumbent within the time budget. We first compare profits and derive the price of representability, then describe the resulting tariffs.

\subsubsection{Profit}

Table~\ref{tab:tou:profit} reports the profits and Table~\ref{tab:tou:por} the associated prices of representability, that is, the relative loss of each ToU profit with respect to the pointwise profit of the indirect method. Since no method certifies global optimality, the price of representability of a configuration is more of an estimate than a precise measure.

\smallskip

The estimated price of representability lies between $1.31\%$ and $13.83\%$ for the two-period family and between $0.59\%$ and $6.55\%$ for the three-period one. In absolute terms the two-period tariff returns between $8.43\EUR$ and $10.14\EUR$ per representative client against a pointwise benchmark of $9.74\EUR$ to $10.27\EUR$. Restricting a retailer that could price every hour separately to a regulator-admissible two-block schedule therefore costs on average $6.85\%$ in the two-period case and $3.63\%$ in the three-period case. The estimate increases with $\alpha$ in every column, and this is consistent with the trend already recorded on the relaxed one in Table~\ref{tab:rtou:profit}. The three-period estimate is below the two-period, which was expected, since the three-period family reproduces every two-period tariff. Indeed, choosing $q_1$ equal to the off-peak price and $q_2 = q_3$ equal to the on-peak price is admissible for $\mathcal{Q}$ and induces the same effective price vector.

\begin{table}[!htbp]
\centering
\setlength{\tabcolsep}{4pt}
\begin{tabular}{cc | ccc | ccc}
\toprule
\multicolumn{1}{c}{} & \multicolumn{1}{c}{} & \multicolumn{3}{c}{$N = 2$} & \multicolumn{3}{c}{$N = 3$} \\
\cmidrule(lr){1-2} \cmidrule(lr){3-5} \cmidrule(lr){6-8}
Clients & $\alpha$ & Direct & Indirect & R & Direct & Indirect & R \\
\cmidrule(lr){1-2} \cmidrule(lr){3-5} \cmidrule(lr){6-8}
\multirow[t]{3}{*}{} & 0.1 & 10.07 & \textbf{10.14} & 9.82 & 9.93 & \textbf{10.21} & 9.87 \\
10 & 1 & 9.84 & \textbf{9.87} & 9.84 & 9.47 & \textbf{9.91} & 9.70 \\
& 10 & 7.30 & \textbf{8.89} & 8.87 & 9.39 & \textbf{9.49} & 9.33 \\
\cmidrule(lr){1-2} \cmidrule(lr){3-5} \cmidrule(lr){6-8}
\multirow[t]{3}{*}{} & 0.1 & 9.69 & \textbf{9.69} & 9.30 & 9.68 & \textbf{9.96} & 9.79 \\
50 & 1 & \textbf{9.48} & 7.48 & 9.34 & 9.50 & 9.02 & \textbf{9.57} \\
& 10 & \textbf{8.52} & 3.90 & 8.41 & \textbf{9.24} & N/A & 9.24 \\
\cmidrule(lr){1-2} \cmidrule(lr){3-5} \cmidrule(lr){6-8}
\multirow[t]{3}{*}{} & 0.1 & \textbf{9.70} & 9.69 & 9.69 & 9.75 & \textbf{9.88} & 9.46 \\
100 & 1 & \textbf{9.47} & N/A & 9.37 & \textbf{9.60} & N/A & 9.59 \\
& 10 & \textbf{8.43} & 6.14 & \textbf{8.43} & \textbf{9.21} & 6.76 & 9.21 \\
\bottomrule
\end{tabular}
\caption{Time-of-use pricing model profit ($\EUR$ per representative client) for the two-period ($N = 2$) and the three-period ($N = 3$) models. For each number of clients: Direct, Indirect and R are the profits for the direct (SGA), indirect (\textsc{Knitro}) and rounding methods, respectively; the best value of each block is in bold and N/A marks a run that returned no incumbent.}
\label{tab:tou:profit}
\end{table}

\begin{table}[!htbp]
\centering
\setlength{\tabcolsep}{4pt}
\begin{tabular}{cc | ccc | ccc}
\toprule
\multicolumn{1}{c}{} & \multicolumn{1}{c}{} & \multicolumn{3}{c}{$N = 2$} & \multicolumn{3}{c}{$N = 3$} \\
\cmidrule(lr){1-2} \cmidrule(lr){3-5} \cmidrule(lr){6-8}
Clients & $\alpha$ & Direct & Indirect & R & Direct & Indirect & R \\
\cmidrule(lr){1-2} \cmidrule(lr){3-5} \cmidrule(lr){6-8}
\multirow[t]{3}{*}{} & 0.1 & 1.96\% & \textbf{1.31\%} & 4.37\% & 3.35\% & \textbf{0.59\%} & 3.94\% \\
10 & 1 & 3.66\% & \textbf{3.40\%} & 3.66\% & 7.31\% & \textbf{2.98\%} & 5.02\% \\
& 10 & 27.40\% & \textbf{11.55\%} & 11.77\% & 6.64\% & \textbf{5.62\%} & 7.18\% \\
\cmidrule(lr){1-2} \cmidrule(lr){3-5} \cmidrule(lr){6-8}
\multirow[t]{3}{*}{} & 0.1 & 4.75\% & \textbf{4.69\%} & 8.57\% & 4.81\% & \textbf{2.10\%} & 3.75\% \\
50 & 1 & \textbf{5.99\%} & 25.81\% & 7.34\% & 5.84\% & 10.55\% & \textbf{5.13\%} \\
& 10 & \textbf{13.83\%} & 60.59\% & 14.94\% & \textbf{6.55\%} & N/A & 6.55\% \\
\cmidrule(lr){1-2} \cmidrule(lr){3-5} \cmidrule(lr){6-8}
\multirow[t]{3}{*}{} & 0.1 & \textbf{2.91\%} & 2.96\% & 2.98\% & 2.38\% & \textbf{1.11\%} & 5.31\% \\
100 & 1 & \textbf{4.43\%} & N/A & 5.43\% & \textbf{3.13\%} & N/A & 3.19\% \\
& 10 & \textbf{13.50\%} & 37.01\% & \textbf{13.50\%} & \textbf{5.43\%} & 30.65\% & 5.43\% \\
\bottomrule
\end{tabular}
\caption{Time-of-use pricing model PoR for the two-period ($N = 2$) and the three-period ($N = 3$) models. For each number of clients: Direct, Indirect and R correspond to the direct (SGA), indirect (\textsc{Knitro}) and rounding methods, respectively; the smallest entry of each block, in bold, is the estimate of $\PoR(\mathcal{Q}, \mathcal{H})$ retained for that configuration.}
\label{tab:tou:por}
\end{table}

\subsubsection{Computational time}

Table~\ref{tab:tou:cpu} reports CPU time. The rounding method is the cheapest and by far the most stable. Its cost is essentially independent of the schedule family because it never enumerates schedules: it performs one pointwise solve, one projection onto $\mathcal{H}$ written as a mixed-integer quadratic program, and one price re-optimization at the fixed schedule $h^\star$, and only the middle step sees $\mathcal{H}$ at all. The direct method costs on average $10$:$08$ for $N=2$, and on average $1$:$51$:$28$ for $N=3$. The ratio of the two averages is $11.0$, which is close to the ratio $|\mathcal{O}_3|/|\mathcal{O}_2| = 24/2 = 12$ of the numbers of admissible subperiod orders, the per-order cost being the same for the two families by Table~\ref{tab:rtou:cpu}. 

\smallskip

The indirect method is the only one whose cost grows with the number of clients, and it is the only one that fails. The reformulation~\eqref{model:cp:kkt-conditions} already carries $2TK + 2K$ binaries for the lower level, as noted in Section~\ref{subsubsec:num:pointwise-profit}, and the schedules add $T \cdot N$ more integer variables together with the switch-counting constraints~\eqref{eq:switch-counting}. As a result, the solver may struggle to find a solution within the allotted time, hence the N/A entries in the table.

\begin{table}[!htbp]
\centering
\begin{tabular}{cc | ccc | ccc}
\toprule
\multicolumn{1}{c}{} & \multicolumn{1}{c}{} & \multicolumn{3}{c}{$N = 2$} & \multicolumn{3}{c}{$N = 3$} \\
\cmidrule(lr){1-2} \cmidrule(lr){3-5} \cmidrule(lr){6-8}
Clients & $\alpha$ & Direct & Indirect & R & Direct & Indirect & R \\
\cmidrule(lr){1-2} \cmidrule(lr){3-5} \cmidrule(lr){6-8}
\multirow[t]{3}{*}{} & 0.1 & 05:59 & \textbf{00:40} & 04:58 & 1:05:10 & 14:08 & \textbf{05:26} \\
10 & 1 & 05:07 & \textbf{02:05} & 03:12 & 57:17 & \textbf{02:54} & 03:24 \\
& 10 & 05:31 & 11:38 & \textbf{04:44} & 1:09:51 & 40:10 & \textbf{05:30} \\
\cmidrule(lr){1-2} \cmidrule(lr){3-5} \cmidrule(lr){6-8}
\multirow[t]{3}{*}{} & 0.1 & \textbf{09:19} & 28:38 & 10:33 & 2:04:24 & 2:17:13 & \textbf{08:32} \\
50 & 1 & 07:59 & 24:33 & \textbf{05:50} & 2:03:04 & 20:45 & \textbf{05:18} \\
& 10 & 10:52 & 1:37:42 & \textbf{06:45} & 2:15:59 & N/A & \textbf{07:49} \\
\cmidrule(lr){1-2} \cmidrule(lr){3-5} \cmidrule(lr){6-8}
\multirow[t]{3}{*}{} & 0.1 & \textbf{12:52} & 1:47:03 & 13:16 & 2:31:38 & 2:28:54 & \textbf{16:07} \\
100 & 1 & \textbf{10:30} & N/A & 13:55 & 2:09:54 & N/A & \textbf{14:09} \\
& 10 & 23:00 & 1:46:16 & \textbf{08:53} & 2:25:53 & 2:10:59 & \textbf{09:11} \\
\bottomrule
\end{tabular}
\caption{Time-of-use pricing model CPU time for the two-period ($N = 2$) and the three-period ($N = 3$) models. For each number of clients: Direct and Indirect are the cumulative CPU time for the direct (SGA) and the CPU time of the indirect (\textsc{Knitro}) methods, respectively, and R that of the rounding method.}
\label{tab:tou:cpu}
\end{table}

\subsubsection{Prices} 

Figure~\ref{fig:tou:prices} shows the ToU tariffs for $K=50$ clients. We note that the same schedule is selected at all three flexibility levels for the two-period model, whereas for the three-period model the schedule differs at $\alpha = 10$; in each period configuration the tariff catches the off-peak periods of the cost profile of Figure~\ref{fig:punc:prices}.

\smallskip

For the two-period case: off-peak hours are $2$ to $6$ and $14$ to $16$, on-peak elsewhere, with levels close to $0.123\EUR$/kWh and $0.166\EUR$/kWh. Unlike the pointwise prices of Figure~\ref{fig:pointwise:prices}, which range from about $0.11$ to $0.19\EUR$/kWh and at several hours fall below the marginal cost, the ToU levels follow the cost profile far more coarsely: the tariff prices the whole $7$--$13$ stretch at a single level although $c_t$ varies from $0.132$ to $0.171\EUR$/kWh over it.

\smallskip

For the three-period case: the collapse observed on the relaxed solutions persists, as the tariff again shows two levels only ($q_1 = q_2$), close to $0.128\EUR$/kWh and $0.176\EUR$/kWh, and the cheap plateau covers twelve hours ($L_1+L_2=12$). All the levels lie strictly inside $[p^{\lowerb}, p^{\upperb}] = [0.05, 0.35]$, so the regulatory bounds are inactive and it is the ordering constraint of $\mathcal{Q}$, active through $q_1 = q_2$, that binds in the three-period model.

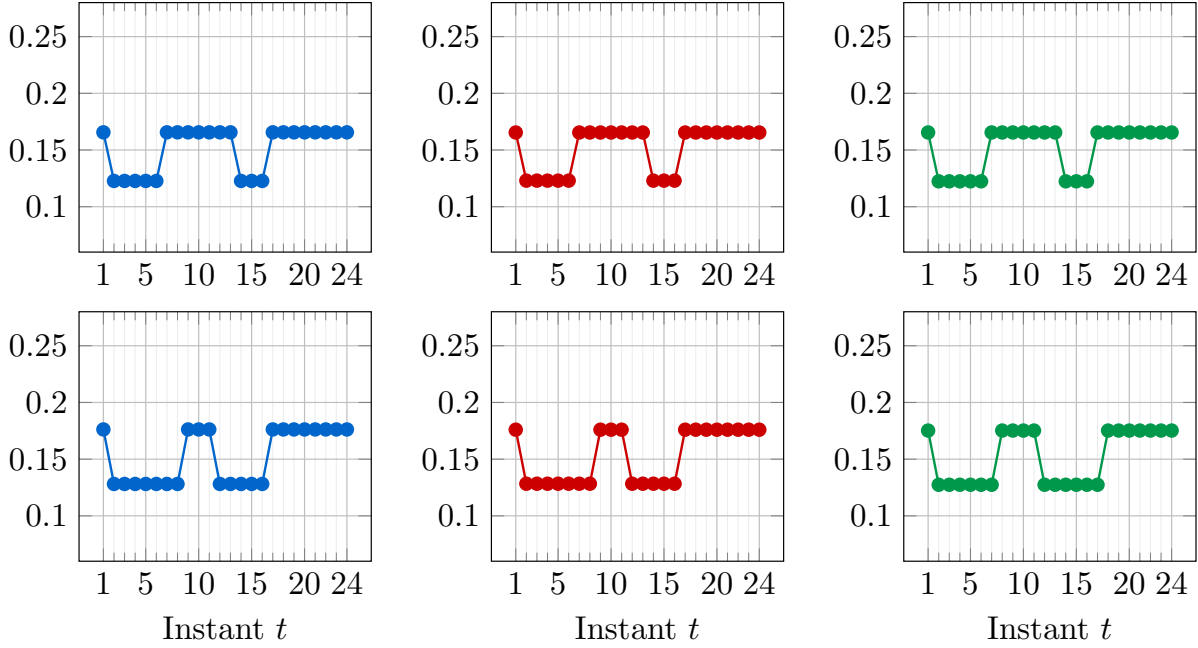
\begin{figure}[!htbp]
\centering
\resizebox{\textwidth}{!}{%
\begin{tabular}{ccc}
\begin{tikzpicture}
\begin{axis}[width=5cm, height=4.5cm, grid=major, ymin=0.06, ymax=0.28,
xtick={1,5,10,15,20,24},
minor xtick={2,3,4,6,7,8,9,11,12,13,14,16,17,18,19,21,22,23},
grid=major,
xminorgrids=true,
minor grid style={very thin, gray!15},]
\addplot[myblue, thick, mark=*] table[
col sep=comma,
x=Hour,
y=Prices_Direct_N2_0.1_50,
] {experiments/tou_prices_summary.csv};
\end{axis}
\end{tikzpicture} &

\begin{tikzpicture}
\begin{axis}[width=5cm, height=4.5cm, grid=major, ymin=0.06, ymax=0.28,
xtick={1,5,10,15,20,24},
minor xtick={2,3,4,6,7,8,9,11,12,13,14,16,17,18,19,21,22,23},
grid=major,
xminorgrids=true,
minor grid style={very thin, gray!15},]
\addplot[myred, thick, mark=*] table[
col sep=comma,
x=Hour,
y=Prices_Direct_N2_1.0_50,
] {experiments/tou_prices_summary.csv};
\end{axis}
\end{tikzpicture} &

\begin{tikzpicture}
\begin{axis}[width=5cm, height=4.5cm, grid=major, ymin=0.06, ymax=0.28,
xtick={1,5,10,15,20,24},
minor xtick={2,3,4,6,7,8,9,11,12,13,14,16,17,18,19,21,22,23},
grid=major,
xminorgrids=true,
minor grid style={very thin, gray!15},]
\addplot[mygreen, thick, mark=*] table[
col sep=comma,
x=Hour,
y=Prices_Direct_N2_10.0_50,
] {experiments/tou_prices_summary.csv};
\end{axis}
\end{tikzpicture} \\

\begin{tikzpicture}
\begin{axis}[width=5cm, height=4.5cm, xlabel={Instant $t$}, grid=major, ymin=0.06, ymax=0.28,
xtick={1,5,10,15,20,24},
minor xtick={2,3,4,6,7,8,9,11,12,13,14,16,17,18,19,21,22,23},
grid=major,
xminorgrids=true,
minor grid style={very thin, gray!15},]
\addplot[myblue, thick, mark=*] table[
col sep=comma,
x=Hour,
y=Prices_Direct_N3_0.1_50,
] {experiments/tou_prices_summary.csv};
\end{axis}
\end{tikzpicture} &
\begin{tikzpicture}
\begin{axis}[width=5cm, height=4.5cm, xlabel={Instant $t$}, grid=major, ymin=0.06, ymax=0.28,
xtick={1,5,10,15,20,24},
minor xtick={2,3,4,6,7,8,9,11,12,13,14,16,17,18,19,21,22,23},
grid=major,
xminorgrids=true,
minor grid style={very thin, gray!15},]
\addplot[myred, thick, mark=*] table[
col sep=comma,
x=Hour,
y=Prices_Direct_N3_1.0_50,
] {experiments/tou_prices_summary.csv};
\end{axis}
\end{tikzpicture} &

\begin{tikzpicture}
\begin{axis}[width=5cm, height=4.5cm, xlabel={Instant $t$}, grid=major, ymin=0.06, ymax=0.28,
xtick={1,5,10,15,20,24},
minor xtick={2,3,4,6,7,8,9,11,12,13,14,16,17,18,19,21,22,23},
grid=major,
xminorgrids=true,
minor grid style={very thin, gray!15},]
\addplot[mygreen, thick, mark=*] table[
col sep=comma,
x=Hour,
y=Prices_Direct_N3_10.0_50,
] {experiments/tou_prices_summary.csv};
\end{axis}
\end{tikzpicture}
\end{tabular}
}
\caption{Time-of-use prices for $50$ clients under the two-period model (top) and the three-period model (bottom); $\alpha=0.1$ (blue, left), $\alpha=1$ (red, middle), $\alpha=10$ (green, right). Prices are in $\EUR$/kWh.}
\label{fig:tou:prices}
\end{figure}

\subsubsection{Load shifting and participation}
Once again, the load-shifting and participation results for the relaxed time-of-use model closely mirror those of the pointwise model. To avoid repetition, we present them in Appendix~\ref{app:rtou-exp}, where the interested reader will find the full set of figures.

\section{Concluding observations and perspectives}
\label{sec:concl}

We have studied the joint design of the temporal structure and the price levels of a time-of-use tariff, in a bilevel model in which each client both reshapes its consumption and decides whether to accept the offer. The three ingredients (a combinatorial schedule, continuous price levels, and an explicit client-level participation decision) have been treated within a single model.

\smallskip

The analysis rests on one structural fact: although the client's program is nonconvex, because the participation level multiplies the value of the consumption problem, it decomposes sequentially into two Euclidean projections (Theorem~\ref{thm:cp:closed-form}) evaluable in $O(T\log T)$ operations (Corollary~\ref{cor:cp:closed-form-complexity}). Two consequences follow. First, the optimality system of that decomposition, rather than the KKT system of the joint program, yields an exact single-level mixed-integer reformulation with instance-computable big-$M$ bounds on every multiplier (Theorem~\ref{thm:cp:kkt-conditions}, Corollary~\ref{cor:cp:kkt-conditions}). Second, the leader's objective is an elementary selection in the sense of Bolte and Pauwels, so a stochastic gradient ascent driven by a cheap selection derivative has almost surely Clarke critical accumulation points (Theorem~\ref{thm:pointwise-selection}, Corollary~\ref{corollary:sgd-convergence}), and the same holds for the schedule once it is reparametrized by continuous start instants and lengths (Proposition~\ref{prop:charec-calen}, Corollary~\ref{cor:sgd-convergence-tou}).

\smallskip

The experiments on the French residential case study support three conclusions. The gradient route is a valid alternative to a commercial mixed-integer solver on the pointwise problem: the two agree within $0.3\%$ on profit and yield visually indistinguishable price vectors, while their computational costs scale in opposite directions, the direct method being independent of the number of clients by construction and two to eight times faster already at $K=100$. On the ToU problem, the mixed-integer route is the best of the three on the small instances, but lags behind for larger instances. The price of representability of a two-period tariff is estimated at $6.85\%$ and that of a three-period tariff at $3.63\%$, which is consistent with the fact that the three-period family dominates the two-period tariff. Finally, the ToU tariffs are much coarser than the pointwise ones, and they do not follow the cost profile as closely, but they still capture the off-peak periods.

\smallskip

Three limitations should be kept in mind. The reported prices of representability are upper estimates: complete enumeration is out of reach at this scale (Remark~\ref{rem:enumeration-cost}), so both the numerator and the denominator of Definition~\ref{def:por} come from local solves of a nonconvex problem. The experiments cover a single representative day, a single population, and a single calibration of the participation sensitivity $\beta_k$ and of the flexibility envelope $f_k$; the levels of participation reported here are consequences of that calibration and are meaningful only as comparisons across pricing schemes at fixed $\beta_k$. Finally, the guarantee of Corollary~\ref{cor:sgd-convergence-tou} is a criticality guarantee, not a global one, and it is informative only at accumulation points interior to the safeguard box (which was the case in our experiments).

\smallskip

Several extensions preserve the structure on which the analysis rests, and are developed in Appendix~\ref{app:extensions}: an anisotropic discomfort matrix, a cost that is nonseparable across clients, a general strictly convex regularizer of the participation decision, a menu of several offers, and day-dependent schedules for seasonal or weekday/weekend tariffs. The last two are the most consequential, since they multiply the schedule space by $I$ offers or $D$ days and therefore remove any remaining hope of enumeration, leaving the gradient-based methods as the only practicable route. 

\printbibliography

\appendix

\section{Notation}
\label{app:notation}

\begin{xltabular}{\textwidth}{@{}l X@{}}
\label{tab:notation} \\
\toprule
\textbf{Symbol} & \textbf{Description} \\
\midrule
\endfirsthead

\multicolumn{2}{@{}l}{\small\itshape Table~\thetable{} continued from the previous page.} \\
\toprule
\textbf{Symbol} & \textbf{Description} \\
\midrule
\endhead

\midrule
\multicolumn{2}{r@{}}{\small\itshape continued on the next page} \\
\endfoot

\bottomrule
\caption{Summary of notation}
\endlastfoot

\multicolumn{2}{@{}l}{\textbf{Sets and indices}} \\
\addlinespace[2pt]
\multicolumn{2}{@{}l}{\textit{Indices and cardinalities}} \\
$\llbracket A \rrbracket$ & Set of integers $\{1,2,\dots,A\}$. \\
$t \in \llbracket T \rrbracket$ & Index of time steps. \\
$T$ & Number of time steps; a time step typically corresponds to one hour, but may be a finer subdivision. \\
$k \in \llbracket K \rrbracket$ & Index of clients. \\
$K$ & Number of clients. \\
$n \in \llbracket N \rrbracket$ & Index of period types (e.g.\ off-peak, shoulder, peak). \\
$N$ & Number of period types. \\
$s \in \llbracket \bar{S} \rrbracket$ & Index of subperiods. \\
\addlinespace[2pt]
\multicolumn{2}{@{}l}{\textit{Model sets}} \\
$\mathcal{Q}$ & Set of admissible price-level vectors $q=(q_n)_{n \in \llbracket N \rrbracket}$, see~\eqref{set:price-levels}. \\
$\mathcal{P}$ & Set of admissible pointwise price vectors, $\mathcal{P} := [p^{\lowerb},p^{\upperb}]^{T}$. \\
$\mathcal{H}$ & Set of admissible schedules $h=(h_{t,n})$, see~\eqref{set:schedules}. \\
$\mathcal{M}$ & Set of lower-dimensional representations of the set $\mathcal{H}$, see Proposition~\ref{prop:charec-calen}. \\
$\mathcal{O}$ & Set of admissible order maps. \\
$\mathcal{L}(\omega)$ & Relaxed set of admissible subperiod-length vectors associated with the order map $\omega$. \\
$\mathcal{X}_k$ & Consumption feasible set of client $k$ (a boxed simplex), see~\eqref{set:consumption-feasible}. \\
\addlinespace

\multicolumn{2}{@{}l}{\textbf{Parameters}} \\
\addlinespace[2pt]
\multicolumn{2}{@{}l}{\textit{Schedule structure}} \\
$L_n$ & Prescribed total duration (number of time steps) of period type $n$. \\
$S_n$ & Maximum number of subperiods (contiguous cyclic blocks) of period type $n$. \\
$\bar{S}$ & Maximum total number of subperiods of a schedule, $\bar{S} = \textstyle\sum_{n=1}^N S_n$. \\
\addlinespace[2pt]
\multicolumn{2}{@{}l}{\textit{Prices and supply cost}} \\
$p^{\lowerb}, p^{\upperb}$ & Minimum and maximum admissible price, respectively, in $\EUR$/kWh. \\
$c_t$ & Marginal cost of supplying electricity at time step $t$ in $\EUR$/kWh. \\
\addlinespace[2pt]
\multicolumn{2}{@{}l}{\textit{Client demand data (kWh)}} \\
$x^0_{t,k}$ & Baseline consumption of client $k$ at time step $t$. \\
$x^{\lowerb}_{t,k}, x^{\upperb}_{t,k}$ & Minimum and maximum consumption of client $k$ at time step $t$, respectively. \\
$f_k$ & Fraction of hourly consumption that client $k$ can shift; used to set $x^{\lowerb}_k, x^{\upperb}_k$ in Section~\ref{sec:num}. \\
\addlinespace[2pt]
\multicolumn{2}{@{}l}{\textit{Client preferences and outside option}} \\
$w_k$ & Weight of client $k$ in the retailer's objective. \\
$\alpha_k$ & Flexibility of client $k$, i.e.\ willingness to deviate from the baseline $x^0_k$. \\
$\beta_k$ & Choice sensitivity of client $k$ in $1/\EUR$; $1/\beta_k$ is the width, in $\EUR$, of the band over which acceptance moves from $0$ to $1$. \\
$p^0_k$ & Price vector of the outside option of client $k$ in $\EUR$/kWh. \\
$\Vill^0_k$ & Value of the outside option of client $k$ in $\EUR$ (e.g.\ a current contract or a regulated default tariff), \emph{including} the associated discomfort. \\
\addlinespace[2pt]
\multicolumn{2}{@{}l}{\textit{Big-$M$ constants}} \\
$M_k, M^\lambda_k, M^\mu_k$ & Instance-computable big-$M$ constants bounding the multipliers of client $k$, see Theorem~\ref{thm:cp:kkt-conditions}. \\
\addlinespace

\multicolumn{2}{@{}l}{\textbf{Decision variables}} \\
\addlinespace[2pt]
\multicolumn{2}{@{}l}{\textit{Upper level (retailer): schedule}} \\
$h_{t,n}$ & Binary variable equal to $1$ if time step $t$ belongs to period type $n$, and $0$ otherwise. \\
$\omega \in \mathcal{O}$ & Order map assigning each subperiod index to a period type. \\
$\mathcal{s}$ & Starting time step of the anchoring subperiod, see Proposition~\ref{prop:charec-calen}. \\
$l_s$ & Length of subperiod $s$. \\
\addlinespace[2pt]
\multicolumn{2}{@{}l}{\textit{Upper level (retailer): prices}} \\
$q_n$ & Continuous variable representing the price of period type $n$ in $\EUR$/kWh. \\
\addlinespace[2pt]
\multicolumn{2}{@{}l}{\textit{Lower level (clients)}} \\
$x_{t,k}$ & Continuous variable representing the consumption of client $k$ at time step $t$ in kWh. \\
$y_k$ & Participation level of client $k$, $y_k \in [0,1]$: probability (equivalently, within a cluster, fraction) of client $k$ accepting the retailer's offer rather than the outside option. \\
\addlinespace[2pt]
\multicolumn{2}{@{}l}{\textit{Dual variables of the lower-level problem}} \\
$\lambda^x_{t,k}, \mu^x_{t,k}$ & Multipliers of the upper and lower consumption bounds of client $k$ at time step $t$, respectively. \\
$\gamma^x_k$ & Multiplier of the daily energy equality constraint of client $k$. \\
$\lambda^y_k, \mu^y_k$ & Multipliers of the upper and lower bounds on $y_k$, respectively. \\
\addlinespace

\multicolumn{2}{@{}l}{\textbf{Auxiliary quantities}} \\
$p_t$ & Pointwise (effective) price at time step $t$ in $\EUR$/kWh, $p_t = \textstyle\sum_{n=1}^N h_{t,n} q_n$. \\
$x^\star_k(p)$ & Optimal consumption of client $k$ under the price signal $p$. \\
$y^\star_k(p)$ & Optimal participation level of client $k$ under the price signal $p$. \\
$\theta$ & Phase label: a pair of an active-set partition and a participation state, see Section~\ref{subsec:pointwise-pricing}. \\
\addlinespace

\multicolumn{2}{@{}l}{\textbf{Functions and operators}} \\
$\Bill(p,x)$ & Bill incurred by the consumption profile $x$ under the prices $p$ in $\EUR$: $\Bill(p,x)=\sum_{t=1}^T p_t x_t$. \\
$\Discomfort_k(x,x^0_k)$ & Discomfort of client $k$ when consuming $x$ instead of the baseline $x^0_k$, in $\EUR$. \\
$\Vill_k(p)$ & Value to client $k$ of accepting the offer $p$, i.e.\ bill plus discomfort at its optimal consumption, in $\EUR$: $\Vill_k(p) = \Bill(p,x^\star_k(p)) + \Discomfort_k(x^\star_k(p),x^0_k)$. \\
$\Cost(x)$ & Cost incurred by the retailer to supply the consumption $x$, in $\EUR$. \\
$F(p)$, $F_k(p)$ & Retailer's hyper-objective at the price signal $p$ and its per-client contribution, see Section~\ref{subsec:pointwise-pricing}. \\
$\Proj_{\mathcal{C}}$ & Euclidean projection onto a convex set $\mathcal{C}$. \\
$\sigma_k$ & Phase-index map $p \mapsto \theta$ of client $k$, see Section~\ref{subsec:pointwise-pricing}. \\
$\widehat\nabla_{\sigma}$ & Selection derivative associated with an elementary selection of index map $\sigma$. \\
$\mathcal{E}, \mathcal{I}, \mathcal{S}$ & Classes of elementary log-exp functions, elementary index functions, and elementary selections, respectively, see Appendix~\ref{app:elementary-selections}. \\
\end{xltabular}

\section{Background: elementary selections}
\label{app:elementary-selections}

\begin{definition}[Elementary log-exp functions]
\label{def:elementary-functions}
A function $f:\R^p\to\R$ is \emph{elementary (log-exp)} if it admits a finite compositional expression involving $+,-,\times,/$, affine maps, $\exp$ and $\log$, on its domain of definition. We denote by $\mathcal{E}$ the set of such functions.
\end{definition}

\begin{definition}[Elementary index]
\label{def:elementary-index}
A function $\sigma:\R^p\to\{1,\dots,m\}$ is an \emph{elementary index} if, for every $i$, the set $\{x\in\R^p : \sigma(x)=i\}$ is the solution set of finitely many equalities and inequalities involving elementary functions on $\R^p$. We denote by $\mathcal{I}$ the set of elementary index functions.
\end{definition}

\begin{definition}[Elementary selection]
\label{def:elementary-selection}
A continuous function $f:\R^p\to\R$ has an \emph{elementary
selection} $(\sigma,f_1,\dots,f_m)$ if $\sigma\in\mathcal{I}$ and $f_1,\dots,f_m\in
\mathcal{E}$ satisfy $f(x)=f_{\sigma(x)}(x)$ for all $x$. We denote by $\mathcal{S}$ the set of elementary selections.
\end{definition}

\begin{definition}[Selection gradient]
\label{def:selection-gradient}
Given $f\in\mathcal{S}$ with selection $(\sigma,f_1,\dots,f_m)$, the
\emph{selection derivative} of $f$ with respect to $\sigma$ is:
\begin{equation*}
\widehat\nabla_\sigma f : x \mapsto \nabla f_{\sigma(x)}(x).
\end{equation*}
\end{definition}

\section{Further extensions}
\label{app:extensions}
The model studied here can be extended in several directions without losing the structure on which the analysis rests, namely the closed-form follower response and the elementary-selection property of the leader's objective. The following extensions are direct applications of the current framework.

\subsection{Intra-time discomfort}
In \eqref{model:cco}, we considered a quadratic discomfort function $\Discomfort_k(x_k, x^0_k) = \tfrac{1}{2\alpha_k} \|(x_k - x^0_k)\|^2$. A natural extension is to consider a diagonal discomfort function of the form:
\begin{equation*}
\Discomfort_k(x_k, x^0_k) = \tfrac{1}{2\alpha_{k}} \| A_k (x_{k} - x^0_{k}) \|, 
\end{equation*}
where $\alpha_{k}$ is the flexibility parameter of client $k$, and $A_k$ is an invertible matrix that models the intra-time discomfort. The consumption best response remains in closed form, and it is given by:
\begin{equation*}
x_k^\star(p) = A^{-1}_k \Proj_{\mathcal{X}_k}\big(A_k x^0_k - \alpha_k p\big).
\end{equation*}
The participation best response remains unchanged, using the value function $\Vill_k(p) = \Bill(p, x^\star_k(p)) + \Discomfort_k(x^\star_k(p), x^0_k)$, and the leader's objective remains an elementary selection.

\subsection{Nonseparable cost}
In the study we conducted so far, we assumed that the retailer's cost function is separable in clients, the total cost is given by the sum over clients weighted by their participation and weights: $\Cost((x_k), (y_k)) = \textstyle\sum_{k=1}^K \Cost(x_k) y_k w_k$. A natural extension is to consider a cost function that is non-separable in clients, but considers just the total load at each instant:
\begin{equation*}
\Cost((x_k), (y_k)) = \Cost\left(\left(\sum_{k=1}^Kx_{t,k} y_k w_k \right)_{t \in \llbracket T \rrbracket}\right),
\end{equation*}
The consumption and participation best responses are unchanged, and the leader's objective remains an elementary selection, so the analysis of Section~\ref{subsec:pointwise-pricing} applies only with a gradient ascent and not stochastic gradient ascent, since the leader's objective is no longer separable in clients.

\subsection{Beyond the quadratic penalization}
\label{subsec:ext-penalization}
In~\eqref{model:cp}, we choose a quadratic regularization for the participation level. Nevertheless, we can always replace $\tfrac{1}{2\beta_k}y_k^2$ by $\tfrac{1}{\beta_k}\psi(y_k)$, where $\psi\colon\R\to\R\cup\{+\infty\}$ is proper, lower semicontinuous and strictly convex with $\operatorname{dom}\psi \subseteq [0,1]$; the model of Section~\ref{sec:mod} is the case $\psi = \tfrac12(\cdot)^2 + \delta_{[0,1]}$ with $\delta_{[0,1]}$ is equal to zero inside $[0,1]$ and equal to $+\infty$ outside $[0,1]$; the consumption best response is unchanged, since $\psi$ does not involve $x_k$, and after eliminating the additive term $\Vill^0_k$ the participation subproblem is:
\begin{equation*}
\min_{y_k} \tfrac{1}{\beta_k}\psi(y_k) + \big(\Vill_k(p) - \Vill^0_k\big)\,y_k ,
\end{equation*}
whose unique solution is
\begin{equation}
\label{eq:ext-participation-br}
y_k = \nabla\psi^*\big(\beta_k(\Vill^0_k - \Vill_k(p))\big),
\end{equation}
where $\psi^*$ is the Legendre--Fenchel conjugate of $\psi$, the indicator $\delta_{[0,1]}$ being included in $\psi$ when the box constraint is imposed separately.

\smallskip

Two hypotheses are needed for~\eqref{eq:ext-participation-br} to be usable. First, $\psi^*$ must be differentiable. Since $\operatorname{dom}\psi$ is bounded, $\psi^*$ is finite on all of $\R$, and it is differentiable there as soon as $\psi$ is strictly convex. Second, $\nabla\psi^*$ must be available in closed form for the complexity statements to survive. Note that $\nabla\psi^*$ is automatically nondecreasing and, because $\operatorname{dom}\psi \subseteq [0,1]$, takes values in $[0,1]$: the best response is feasible by construction and remains monotone in the surplus $\Vill^0_k - \Vill_k(p)$, so the qualitative reading of the model is preserved. For the elementary-selection property (Theorem~\ref{thm:pointwise-selection}), it suffices in addition that $\nabla\psi^*$ be an elementary selection; the rest of the proof is obtained by stability of $\mathcal{S}$.

\begin{example}[Logit model]
Taking $\psi$ to be the negative entropy of the accept/reject choice,
\begin{equation*}
\psi(y) = y\log y + (1-y)\log(1-y), \qquad y \in [0,1],
\end{equation*}
is a case of interest, because its conjugate gradient is the logit map $\nabla\psi^*(s) = (1 + e^{-s})^{-1}$, so that
\begin{equation*}
y_k = \big(1 + \exp(-\beta_k(\Vill^0_k - \Vill_k(p)))\big)^{-1},
\end{equation*}
with $\beta_k$ playing the role of an inverse temperature: entropic regularization of the accept/reject choice induces logit (softmax) acceptance probabilities, whereas the quadratic regularization induces the truncated-linear probabilities of Figure~\ref{fig:client-choice}.
\end{example}

\subsection{Time-of-use menu design}
\label{app:ext-menu}
We have so far considered that the retailer proposes a single offer $(q,h)$ to all clients. We now consider the case where the retailer instead proposes a menu $(q^i,h^i)_{i \in \llbracket I \rrbracket}$ of $I$ offers, where each client has the right to select the offer or fractions of the offers they prefer. In this case, the lower-level problem of client $k$ can be described by:
\begin{align} \label{model:cp:menu}
\min_{x_k, y_k} \quad & \textstyle\sum_{i=1}^I \left[ \Bill(p^i, x^i_k) + \Discomfort_k(x^i_k, x^0_k) \right] y^i_k + \Vill^0_k (1-\sum_{i=1}^I y^i_k) + \tfrac{1}{2\beta_k} \|y_k\|^2 \\
\text{s.t.} \quad & x_k^i \in \mathcal{X}_k, \quad \forall i \in \llbracket I \rrbracket, \notag \\
&y_k \in \mathcal{Y}, \notag
\end{align}
where $p^i = p(q^i,h^i)$ for all $i \in \llbracket I \rrbracket$, and where the set $\mathcal{Y}$ is given by:
\begin{equation*}
\mathcal{Y} = \{y \in [0,1]^{I} \mid \textstyle\sum_{i=1}^I y^i \le 1\}.
\end{equation*}
The variable $y_k^i$ is the fraction of client $k$ selecting offer $i$, or equivalently, is the probability of client $k$ selecting offer $i$, with the convention that $i = 0$ represents the outside option. The variable $x^i_k$ is the consumption profile of client $k$ if they select offer $i$.

\smallskip

The leader's model in this case is:
\begin{align*}
\max_{q, h} \quad & \textstyle\sum_{i=1}^I \textstyle\sum_{k=1}^K \left[ \Bill(p(q^i,h^i), x_k^i) - \Cost(x_k^i) \right] y_k^i w_k \\
\text{s.t.} \quad & q^i \in \mathcal{Q}, \quad h^i \in \mathcal{H}, \qquad \forall i \in \llbracket I \rrbracket, \\
& (x_k,y_k) \text{ solves $\eqref{model:cp:menu}$}, \qquad \forall k \in \llbracket K \rrbracket,
\end{align*}

The lower-level problem inherits the nonconvexity from the original model \ref{model:cp}, nevertheless, the nested nature survives. Each inner minimization over $x^i_k$ may be carried out independently of $y_k$, and the outer problem in $y_k$ is a strongly convex quadratic program over the simplex. The best responses, in this case, are:
\begin{align*}
x_k^i &= \Proj_{\mathcal{X}_k}\!\big(x_k^0 - \alpha_k\, p(q^i, h^i)\big), \\
y_k &= \Proj_{\mathcal{Y}}\!\big(\beta_k (\Vill_k^0 \mathbf{1} - \hat{\Vill}_k)\big),
\end{align*}
where $\hat{\Vill}_k = (\Vill_k(p^i))_{i \in \llbracket I \rrbracket}$, and where $\Vill_k(p)$ was introduced in \eqref{model:cco}. Only offers with $y^i_k > 0$ have a determined consumption profile; for the others (where $y^i_k = 0$), the formula above is a valid selection, as the objective does not depend on $x^i_k$. Projection onto $\mathcal{Y}$ costs $O(I\log{I})$, so the follower oracle costs $O(I\,T\log{T}+I\log{I})$. The structure of the analysis is unchanged, with complexity scaling in $I$. The leader's variable count is multiplied by $I$, so the schedule set becomes $\mathcal{H}^{I}$ and complete enumeration is no longer a viable option, as it was already too costly to use even in the single-offer case. Consistently with Section~\ref{subsec:ext-penalization}, the quadratic regularizer makes the projection onto $\mathcal{Y}$ sparse, so offers dominated by more than a threshold receive exactly zero share, whereas the entropic regularizer would spread strictly positive shares over the whole menu. 

\subsection{Day-dependent time-of-use pricing}
\label{app:multi-day}
We now extend the model, moving from the single typical day considered so far to a setting with multiple days, each with its own schedule and price. This is relevant for seasonal pricing, where the retailer may want to propose different offers for different seasons or day types (e.g., weekdays vs weekends). Let's consider $D$ representative days, and for each day $d \in \llbracket D \rrbracket$, the retailer proposes a schedule $h^d$ and a price vector $q^d$. The consumption of client $k$ on day $d$ is denoted by $x_k^d$, and the participation variable $y_k$ is now a single variable that represents the probability of client $k$ accepting the offer across all days. For each day $d$, the baseline profile of client $k$ is denoted $x^{0,d}_k$ and its consumption feasible set is:
\begin{equation*}
\mathcal{X}^d_k := \{x \in \R^{T} \mid \textstyle\sum_{t = 1}^T x_t = \textstyle\sum_{t = 1}^T x^{0,d}_{t,k}, \, x^{\lowerb, d}_{t,k} \le x_t \le x^{\upperb, d}_{t,k}, \, \forall t \in T\}.
\end{equation*}
So the lower-level problem for client $k$ becomes a multi-day problem, where the client chooses consumption profiles $x_k^d$ for each day $d$ and a single participation variable $y_k$:
\begin{align*} \label{model:cp:seas}
\min_{x_k, y_k} \quad & \left\{ \textstyle\sum_{d = 1}^D \left[ \Bill(p(q^d, h^d), x_k^d) + \Discomfort_k(x_k^d, x^{0,d}_k) \right] w^d \right\} y_k + \Vill^0_k (1-y_k) + \tfrac{1}{2\beta_k} y_k^2, \\
\text{s.t.} \quad & x_k^d \in \mathcal{X}^d_k, \qquad \forall d \in \llbracket D \rrbracket, \\
& y_k \in [0,1],
\end{align*}
where $y_k$ is the probability (or fraction) of client $k$ accepting the offer as a whole, and $w^d \ge 0$ with $\sum_{d=1}^D w^d = 1$ weights day $d$ (its frequency in the horizon). The same reasoning as in Section~\ref{sec:analytical} applies, and the best responses can be computed in closed form:
\begin{align*}
x_k^d &= \Proj_{\mathcal{X}^d_k}\big(x_k^{0,d} - \alpha_k\, p(q^d, h^d)\big), \\
y_k &= \Proj_{[0,1]}\Big(\beta_k\Big(\Vill^0_k - \textstyle\sum_{d=1}^D w^d \Vill^d_k\Big)\Big),
\end{align*}
where $\Vill^d_k = \Bill(p(q^d,h^d), x^d_k) + \Discomfort_k(x^d_k, x^{0,d}_k)$. The follower oracle costs scales with the number of days, and the participation response is again a truncated affine function, now of the aggregated value $\sum_d w^d \Vill^d_k$, so Theorem~\ref{thm:pointwise-selection} and the complexity statements still hold. The retailer's optimization problem is now:
\begin{align*}
\max_{q, h} \quad & \textstyle\sum_{k=1}^K \textstyle\sum_{d=1}^D \left[\Bill(p(q^d,h^d), x_k^d) - \Cost(x_k^d) \right] y_k w^d w_k \\
\text{s.t.} \quad & q^d \in \mathcal{Q}, \quad h^d \in \mathcal{H}, \qquad \forall d \in \llbracket D \rrbracket, \\
& (x_k,y_k) \text{ solves the multi-day lower-level problem}, \qquad \forall k \in \llbracket K \rrbracket.
\end{align*}
The consumption responses decompose across days, but the participation decision does not: a single $y_k$ aggregates all days, so an offer that is unattractive on one day can be compensated on another, and the leader's problem does \emph{not} separate into $D$ independent single-day problems. This coupling, rather than the dimension, is what distinguishes the day-dependent model from $D$ copies of the one-day model. The price to pay is that the admissible schedule set becomes $\mathcal{H}^{D}$, so the enumeration baseline of Section~\ref{subsubsec:enumeration} is not practical, and the gradient-based methods become the only ones that remain practicable. The two extensions above combine without difficulty: a multi-day multi-offer is obtained by letting $y_k$ range over the simplex $\mathcal{Y}$ and aggregating $\Vill^{i}_k = \textstyle\sum_{d} w^d \Vill^{i,d}_k$ over days before projecting. 

\section{Supplementary experimental results}
In this appendix, we report additional experimental results that complement the results presented in Section~\ref{subsec:num:rtou} and in Section~\ref{subsec:num:tou}.

\subsection{Relaxed time-of-use pricing}
\label{app:rtou-exp}

Figure~\ref{fig:rtou:shifts} represents the average relative consumption shift under the relaxed time-of-use pricing model studied in Section~\ref{subsec:num:rtou}, while Table~\ref{tab:rtou:participation} represents the average participation in the same setting. 

\begin{figure}[!htbp]
\centering
\resizebox{\textwidth}{!}{%
\begin{tabular}{ccc}
\begin{tikzpicture}
\begin{axis}[width=5cm, height=4.5cm, grid=major, ymin=-11, ymax=11,
xtick={1,5,10,15,20,24},
minor xtick={2,3,4,6,7,8,9,11,12,13,14,16,17,18,19,21,22,23},
grid=major,
xminorgrids=true,
minor grid style={very thin, gray!15},]
\addplot[myblue, thick, mark=*] table[
col sep=comma,
x=Hour,
y=shifts_01_50_Direct_N2,
] {experiments/rtou_shifts_summary.csv};
\end{axis}
\end{tikzpicture} &

\begin{tikzpicture}
\begin{axis}[width=5cm, height=4.5cm, grid=major, ymin=-11, ymax=11,
xtick={1,5,10,15,20,24},
minor xtick={2,3,4,6,7,8,9,11,12,13,14,16,17,18,19,21,22,23},
grid=major,
xminorgrids=true,
minor grid style={very thin, gray!15},]
\addplot[myred, thick, mark=*] table[
col sep=comma,
x=Hour,
y=shifts_1_50_Direct_N2,
] {experiments/rtou_shifts_summary.csv};
\end{axis}
\end{tikzpicture} &

\begin{tikzpicture}
\begin{axis}[width=5cm, height=4.5cm, grid=major, ymin=-11, ymax=11,
xtick={1,5,10,15,20,24},
minor xtick={2,3,4,6,7,8,9,11,12,13,14,16,17,18,19,21,22,23},
grid=major,
xminorgrids=true,
minor grid style={very thin, gray!15},]
\addplot[mygreen, thick, mark=*] table[
col sep=comma,
x=Hour,
y=shifts_10_50_Direct_N2,
] {experiments/rtou_shifts_summary.csv};
\end{axis}
\end{tikzpicture} \\

\begin{tikzpicture}
\begin{axis}[width=5cm, height=4.5cm, grid=major, ymin=-11, ymax=11,
xtick={1,5,10,15,20,24},
minor xtick={2,3,4,6,7,8,9,11,12,13,14,16,17,18,19,21,22,23},
grid=major,
xminorgrids=true,
minor grid style={very thin, gray!15},]
\addplot[myblue, thick, mark=*] table[
col sep=comma,
x=Hour,
y=shifts_01_50_Direct_N3,
] {experiments/rtou_shifts_summary.csv};
\end{axis}
\end{tikzpicture} &

\begin{tikzpicture}
\begin{axis}[width=5cm, height=4.5cm, grid=major, ymin=-11, ymax=11,
xtick={1,5,10,15,20,24},
minor xtick={2,3,4,6,7,8,9,11,12,13,14,16,17,18,19,21,22,23},
grid=major,
xminorgrids=true,
minor grid style={very thin, gray!15},]
\addplot[myred, thick, mark=*] table[
col sep=comma,
x=Hour,
y=shifts_1_50_Direct_N3,
] {experiments/rtou_shifts_summary.csv};
\end{axis}
\end{tikzpicture} &

\begin{tikzpicture}
\begin{axis}[width=5cm, height=4.5cm, grid=major, ymin=-11, ymax=11,
xtick={1,5,10,15,20,24},
minor xtick={2,3,4,6,7,8,9,11,12,13,14,16,17,18,19,21,22,23},
grid=major,
xminorgrids=true,
minor grid style={very thin, gray!15},]
\addplot[mygreen, thick, mark=*] table[
col sep=comma,
x=Hour,
y=shifts_10_50_Direct_N3,
] {experiments/rtou_shifts_summary.csv};
\end{axis}
\end{tikzpicture}
\end{tabular}
}
\caption{Average relative consumption shift $\sum_{k=1}^K \frac{x_k^\star - x_k^0}{x_k^0} w_k$ in $\%$ for 50 clients under the two-period model (top) and the three-period model (bottom); $\alpha=0.1$ (blue, left), $\alpha=1$ (red, middle), $\alpha=10$ (green, right).}
\label{fig:rtou:shifts}
\end{figure}
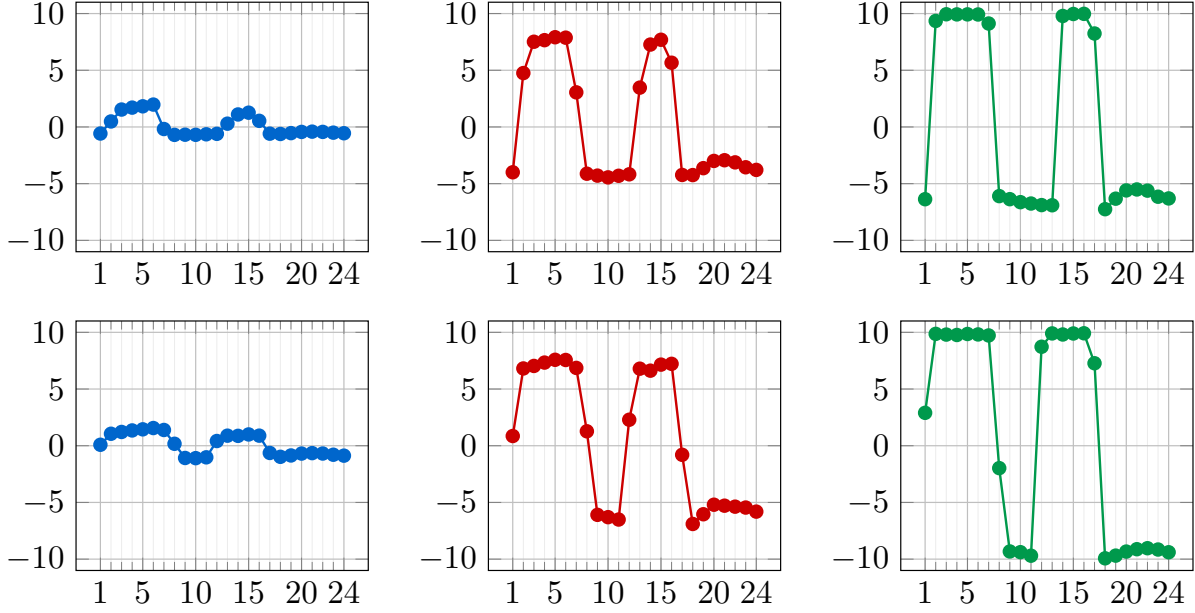
\FloatBarrier

\begin{table}[!htbp]
\centering
\begin{tabular}{c | cc | cc | cc}
\toprule
\multicolumn{1}{c}{} & \multicolumn{2}{c}{10} & \multicolumn{2}{c}{50} & \multicolumn{2}{c}{100} \\
\cmidrule{1-1} \cmidrule(lr){2-3} \cmidrule(lr){4-5} \cmidrule(lr){6-7}
$\alpha$ & $N=2$ & $N=3$ & $N=2$ & $N=3$ & $N=2$ & $N=3$ \\
\midrule
0.1 & 20.99\% & 22.78\% & 20.66\% & 19.70\% & 19.56\% & 19.67\%\\
1 & 20.08\% & 22.55\% & 20.12\% & 19.62\% & 19.33\% & 19.62\%\\
10 & 23.05\% & 21.45\% & 19.10\% & 19.70\% & 18.14\% & 19.16\% \\
\bottomrule
\end{tabular}
\caption{Relaxed time-of-use pricing model participation, for the two-period ($N=2$) and the three-period ($N=3$) models, obtained with Algorithm~\ref{alg:sgd-tou}.}
\label{tab:rtou:participation}
\end{table}
\FloatBarrier

\subsection{Time-of-use pricing}
\label{app:tou-exp}

Figure~\ref{fig:tou:shifts} represents the average relative consumption shift under the time-of-use pricing model studied in Section~\ref{subsec:num:tou}, while Table~\ref{tab:tou:participation} represents the average participation in the same setting. 

\begin{figure}[!htbp]
\centering
\resizebox{\textwidth}{!}{%
\begin{tabular}{ccc}
\begin{tikzpicture}
\begin{axis}[width=5cm, height=4.5cm, grid=major, ymin=-11, ymax=11,
xtick={1,5,10,15,20,24},
minor xtick={2,3,4,6,7,8,9,11,12,13,14,16,17,18,19,21,22,23},
grid=major,
xminorgrids=true,
minor grid style={very thin, gray!15},]
\addplot[myblue, thick, mark=*] table[
col sep=comma,
x=Hour,
y=shifts_01_50_Direct_N2,
] {experiments/tou_shifts_summary.csv};
\end{axis}
\end{tikzpicture} &

\begin{tikzpicture}
\begin{axis}[width=5cm, height=4.5cm, grid=major, ymin=-11, ymax=11,
xtick={1,5,10,15,20,24},
minor xtick={2,3,4,6,7,8,9,11,12,13,14,16,17,18,19,21,22,23},
grid=major,
xminorgrids=true,
minor grid style={very thin, gray!15},]
\addplot[myred, thick, mark=*] table[
col sep=comma,
x=Hour,
y=shifts_1_50_Direct_N2,
] {experiments/tou_shifts_summary.csv};
\end{axis}
\end{tikzpicture} &

\begin{tikzpicture}
\begin{axis}[width=5cm, height=4.5cm, grid=major, ymin=-11, ymax=11,
xtick={1,5,10,15,20,24},
minor xtick={2,3,4,6,7,8,9,11,12,13,14,16,17,18,19,21,22,23},
grid=major,
xminorgrids=true,
minor grid style={very thin, gray!15},]
\addplot[mygreen, thick, mark=*] table[
col sep=comma,
x=Hour,
y=shifts_10_50_Direct_N2,
] {experiments/tou_shifts_summary.csv};
\end{axis}
\end{tikzpicture} \\

\begin{tikzpicture}
\begin{axis}[width=5cm, height=4.5cm, grid=major, ymin=-11, ymax=11,
xtick={1,5,10,15,20,24},
minor xtick={2,3,4,6,7,8,9,11,12,13,14,16,17,18,19,21,22,23},
grid=major,
xminorgrids=true,
minor grid style={very thin, gray!15},]
\addplot[myblue, thick, mark=*] table[
col sep=comma,
x=Hour,
y=shifts_01_50_Direct_N3,
] {experiments/tou_shifts_summary.csv};
\end{axis}
\end{tikzpicture} &

\begin{tikzpicture}
\begin{axis}[width=5cm, height=4.5cm, grid=major, ymin=-11, ymax=11,
xtick={1,5,10,15,20,24},
minor xtick={2,3,4,6,7,8,9,11,12,13,14,16,17,18,19,21,22,23},
grid=major,
xminorgrids=true,
minor grid style={very thin, gray!15},]
\addplot[myred, thick, mark=*] table[
col sep=comma,
x=Hour,
y=shifts_1_50_Direct_N3,
] {experiments/tou_shifts_summary.csv};
\end{axis}
\end{tikzpicture} &

\begin{tikzpicture}
\begin{axis}[width=5cm, height=4.5cm, grid=major, ymin=-11, ymax=11,
xtick={1,5,10,15,20,24},
minor xtick={2,3,4,6,7,8,9,11,12,13,14,16,17,18,19,21,22,23},
grid=major,
xminorgrids=true,
minor grid style={very thin, gray!15},]
\addplot[mygreen, thick, mark=*] table[
col sep=comma,
x=Hour,
y=shifts_10_50_Direct_N3,
] {experiments/tou_shifts_summary.csv};
\end{axis}
\end{tikzpicture}
\end{tabular}
}
\caption{Average relative consumption shift $\sum_{k=1}^K \frac{x_k^\star - x_k^0}{x_k^0} w_k$ in $\%$ for 50 clients under the two-period model (top) and the three-period model (bottom); $\alpha=0.1$ (blue, left), $\alpha=1$ (red, middle), $\alpha=10$ (green, right).}
\label{fig:tou:shifts}
\end{figure}
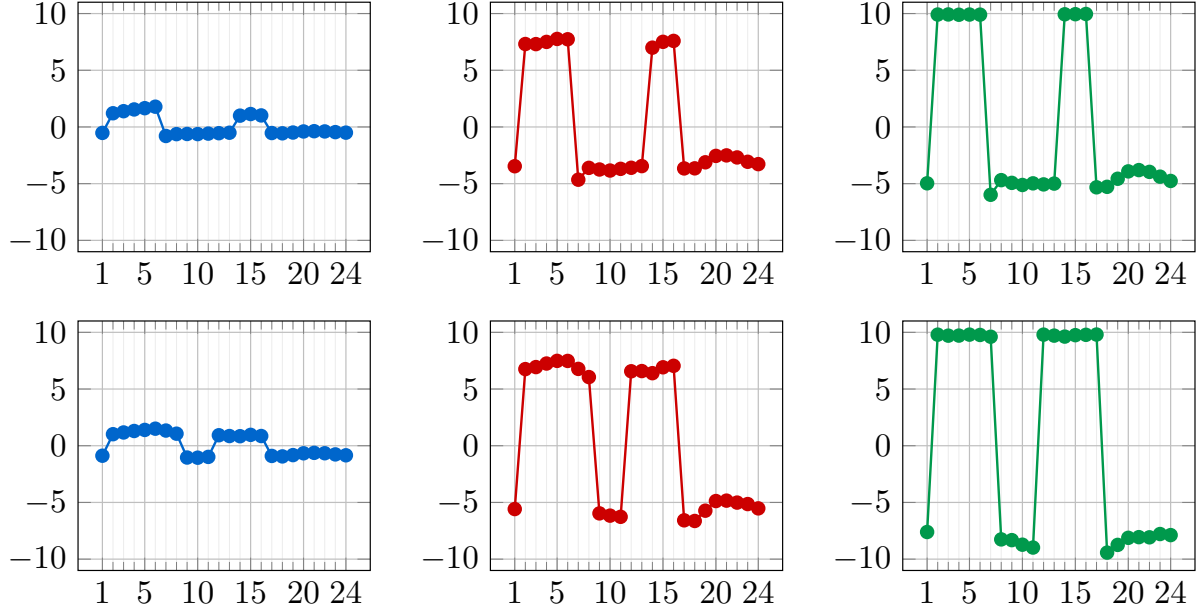
\FloatBarrier

\begin{table}[!htbp]
\centering
\begin{tabular}{c | cc | cc | cc}
\toprule
\multicolumn{1}{c}{} & \multicolumn{2}{c}{10} & \multicolumn{2}{c}{50} & \multicolumn{2}{c}{100} \\
\cmidrule{1-1} \cmidrule(lr){2-3} \cmidrule(lr){4-5} \cmidrule(lr){6-7}
$\alpha$ & $N=2$ & $N=3$ & $N=2$ & $N=3$ & $N=2$ & $N=3$ \\
\midrule
0.1 & 20.55\% & 20.49\% & 21.31\% & 19.11\% & 19.29\% & 19.87\% \\
1 & 20.77\% & 26.23\%& 20.84\% & 19.30\% & 19.54\% & 20.01\% \\
10 & 24.19\% & 23.11\% & 19.73\% & 18.92\% & 18.90\% & 18.56\% \\
\bottomrule
\end{tabular}
\caption{Time-of-Use pricing model participation comparison using the direct method.}
\label{tab:tou:participation}
\end{table}
\FloatBarrier

\end{document}